%% file: driver.tex
\documentclass[11pt,letterpaper]{article}
\usepackage{amsmath,amsthm,amsfonts,color,hyperref,anysize,amssymb,color,graphicx,epstopdf,mathtools,mathrsfs}
\usepackage[section]{placeins}
\usepackage[usenames,dvipsnames]{xcolor}
\usepackage[normalem]{ulem}
\usepackage{algorithm,algorithmic}
\usepackage[margin=1.0in]{geometry}
\usepackage{multirow}
\usepackage{multicol}
\usepackage{subfiles}
\usepackage[inline]{enumitem}
\usepackage{pgf,tikz}
\usepackage{soul}
\usepackage{dsfont}

\input{header}
\title{Fully Asynchronous Stochastic Coordinate Descent:\\
A Tight Lower Bound on the Parallelism\\
Achieving Linear Speedup
\thanks{Part of the work done while Yun Kuen Cheung held positions at Courant Institute, NYU,
at Faculty of Computer Science, University of Vienna and at Max Planck Institute for Informatics, Saarland Informatics Campus.
He was supported in part by NSF Grant CCF-1217989, the Vienna Science and Technology Fund (WWTF) project ICT10-002,
Singapore NRF 2018 Fellowship NRF-NRFF2018-07 and MOE AcRF Tier 2 Grant 2016-T2-1-170.
Additionally the research leading to these results has received funding from
the European Research Council under the European Union's Seventh Framework Programme (FP7/2007-2013) / ERC Grant Agreement no.~340506.}
\thanks{Richard Cole and Yixin Tao's work was supported in part by NSF Grants CCF-1217989, CCF-1527568 and CCF-1909538.}}
\author{Yun Kuen Cheung
\\
		Singapore University of\\Technology and Design
\and
		Richard Cole~~~~~~~~~~Yixin Tao
\\
		Courant Institute, NYU
}
\date{}
\begin{document}\setlength{\parindent}{0.2in}
\maketitle

\input{abstract}

\newpage
\thispagestyle{plain}\setcounter{page}{1}

\input{introduction}
\input{model}
\input{keyidea}
\input{simple-analysis}

\input{app-macros}
\input{general-framework}
\input{SACD-notation}
\input{bound-on-delta-x-span}
\input{SACD-gradient-bounds}
\input{Errt-bound}
\input{new-amortization}
\section*{Acknowledgment}

We thank several anonymous reviewers for their helpful and thoughtful suggestions regarding earlier versions of this paper.

\bibliographystyle{plain}
\bibliography{acd_references}
\newpage
\appendix
\input{appendix_general}
\input{app-exptd-prog}
\input{app-hw-and-dhat}
\input{app-SACD-analysis-gen-new-temp}

\input{appendix_recursive-new-temp}
\input{appendix_claim}
\input{app-gen-cvgce-thm}
\input{app-amort-section-proofs}

\input{parameter}

\end{document}

%% file: header.tex
\usepackage{xcolor}

\usepackage{amssymb,amsmath}

\usepackage{multirow}

\usepackage[inline]{enumitem}

\DeclareMathOperator*{\argmax}{arg\,max}
\DeclareMathOperator*{\res}{res}

\newcommand{\Dvar}[3]{\Delta^{#3}_{\mathsf{var}} x^{#1}_{#2}}
\newcommand{\Dbarvar}[2]{\overline{\Delta}_{\mathsf{var}} x^{#1}_{#2}}
\newcommand{\DbarvarR}[2]{\overline{\Delta}_{\mathsf{var}}^R x^{#1}_{#2}}
\newcommand{\Dbarvaralt}[1]{\overline{\Delta}_{\mathsf{var}}^{#1}}
\newcommand{\Dspan}[3]{\Delta^{#3}_{\mathsf{span}} x^{#1}_{#2}}
\newcommand{\Dbarspan}[2]{\overline{\Delta}_{\mathsf{span}} x^{#1}_{#2}}
\newcommand{\DbarspanR}[2]{\overline{\Delta}_{\mathsf{span}}^R x^{#1}_{#2}}
\newcommand{\Dbarspanalt}[1]{\overline{\Delta}_{\mathsf{span}}^{#1}}

\newcommand{\calD}{\mathcal{D}}

\newcommand{\calR}{\mathcal{R}}

\newcommand{\calU}{\mathcal{U}}

\newcommand{\bbx}{\mathbf{x}}

\newcommand{\rr}{\mathbb{R}}
\newcommand{\rrplus}{\mathbb{R}^+}

\newcommand{\la}{\leftarrow}
\newcommand{\ra}{\rightarrow}

\newcommand{\tg}{\tilde{g}}

\newcommand{\tx}{\tilde{x}}

\newcommand{\comm}[1]{\qquad\mbox{(#1)}}
\newcommand{\xs}{x^*}

\newcommand{\G}{\Gamma}
\newcommand{\pt}{x^t}
\newcommand{\ptone}{x^{t-1}}

\newcommand{\Dx}{\Delta x}

\newcommand{\Ljj}{L_{jj}}
\newcommand{\Lkj}{L_{kj}}
\newcommand{\Ljk}{L_{jk}}

\newcommand{\td}{\tilde{d}}
\newcommand{\PRG}{\text{\textsf{PRG}}}
\newcommand{\PRGe}{\text{\emph{\textsf{PRG}}}}
\newcommand{\hW}{\widehat{W}}
\newcommand{\hWj}{\widehat{W}_j}
\newcommand{\hWk}{\widehat{W}_k}
\newcommand{\hd}{\widehat{d}}
\newcommand{\hdj}{\widehat{d}_j}

\newcommand{\Wj}{W_j}

\newcommand{\ve}{\vec{e}}
\newcommand{\kt}{k_t}
\newcommand{\ks}{k_s}

\newcommand{\Lmax}{L_{\max}}

\newcommand{\Lres}{L_{\res}}
\newcommand{\Lresbar}{L_{\overline{\res}}}

\newcommand{\nuo}{\nu_1}
\newcommand{\nut}{\nu_2}
\newcommand{\hWkt}{\widehat{W}_{k_t}}

\newcommand{\Psij}{\Psi_j}
\newcommand{\Psik}{\Psi_k}
\newcommand{\Psikt}{\Psi_{k_t}}

\newcommand{\gej}{g_j}
\newcommand{\gejp}{g'_j}

\newcommand{\gkt}{g_{k_t}^t}

\newcommand{\pktt}{x_{k_t}^t}

\newcommand{\pkttm}{x_{k_t}^{t-1}}

\newcommand{\tgkt}{\tilde{g}_{k_t}^t}

\newcommand{\barT}{T+1}

\newcommand{\gjo}{g_j^1}
\newcommand{\gjtw}{g_j^2}

\newcommand{\muF}{\mu_F}
\newcommand{\muf}{\mu_f}

\newcommand{\Xs}{X^*}

\newcommand{\prob}[1]{\mathbb{P}\left[ #1 \right]}
\newcommand{\expect}[1]{\mathbb{E}\left[ #1 \right]}
\newcommand{\expectpi}[1]{\mathbb{E}\left[ #1 \right]} 
\newcommand{\expectsub}[2]{\mathbb{E}_{#1}\left[ #2 \right]}

\newcommand{\hide}[1]{}
\hide{
     \newcommand{\qed}{\nobreak \ifvmode \relax \else
     \ifdim\lastskip<1.5em \hskip-\lastskip
     \hskip1.5em plus0em minus0.5em \fi \nobreak
     \vrule height0.75em width0.5em depth0.25em\fi}
}

\renewcommand{\paragraph}[1]{\smallskip

\noindent\textbf{#1}~}

\newcommand{\rjc}[1]{{\color{blue}#1}}
\newcommand{\RJC}[1]{{\textbf{\color{orange}RJC:~ #1}}}
\newcommand{\yxt}[1]{{\color{olive}#1}}
\newcommand{\YXT}[1]{\textbf{\color{olive}YXT:~ #1}}
\newcommand{\ykc}[1]{{\color{red}{#1}}}
\newcommand{\YKC}[1]{\textbf{\color{purple}MARCO:~#1}}

\newtheorem{theorem}{Theorem}
\newtheorem{lemma}{Lemma}

\newtheorem{clm}{Claim}
\newtheorem{assume}{Assumption}
\newtheorem{defn}{Definition}
\newtheorem{example}{Example}

\newtheorem{obs}{Observation}

\newenvironment{pfof}[1]{\begin{proof}[Proof of #1]}{\end{proof}}

\newcommand{\Df}{\mathcal{D}}
\newcommand{\DE}{\Delta^{\textit{X}}}
\newcommand\numberthis{\addtocounter{equation}{1}\tag{\theequation}}

%% file: abstract.tex
\begin{abstract}
  \hide{
  When solving massive optimization problems in areas such as machine learning, it is common practice to seek speedup via \rjc{substantial} 
  parallelism.
  However, especially in an asynchronous environment, there are limits on the possible parallelism.
  Accordingly, }
  We seek tight bounds on the viable parallelism in asynchronous implementations of coordinate descent that achieves linear speedup.
  We focus on asynchronous coordinate descent (ACD) algorithms on convex functions which consist of the sum of a smooth convex part and a possibly non-smooth separable convex part.
  \hide{
  $F:\rr^n\ra\rr$ of the form $$F(x) = f(x) ~+~ \sum_{k=1}^n \Psi_k(x_k),$$
  where $f:\rr^n\ra\rr$ is a smooth convex function, and each $\Psi_k:\rr\ra\rr$ is a univariate and possibly non-smooth convex function.
  }

  We quantify the shortfall in progress compared to the standard sequential stochastic gradient descent.
  This leads to a 
  simple yet tight analysis of the standard stochastic ACD in a partially asynchronous environment,
  generalizing and improving the bounds in prior work.
  We also give a considerably more involved analysis for general asynchronous environments in which
  the only constraint is that each update can overlap with at most $q$ others.
  The new lower bound on the maximum degree of parallelism attaining linear speedup is tight and improves the best prior bound almost quadratically.
  \hide{
  This improves Liu and Wright's~\cite{LiuW2015} lower bound on the maximum degree of parallelism attaining linear speedup almost quadratically;
  the new bound is essentially tight. \RJC{Omit essentially?}
  }

\end{abstract}

%% file: introduction.tex
\section{Introduction}\label{sect:intro}

We consider the problem of finding an (approximate) minimum point of a convex function $F:\rr^n\ra\rr$ of the form
\[
F(x) = f(x) + \sum_{k=1}^n \Psi_k(x_k),
\]
where $f:\rr^n\ra\rr$ is a smooth convex function\footnote{In fact, having a continuous gradient suffices.}, and each $\Psi_k:\rr\ra\rr$ is a univariate convex function, but may be non-smooth.
Such functions occur in many data analysis and machine learning problems, such as linear regression
(e.g., the Lasso approach to regularized least squares~\cite{Tib94}) where $\Psi_k(x_k) = |x_k|$,
logistic regression~\cite{MVdGB2008}, ridge regression~\cite{SGV1998} where $\Psi_k(x_k)$ is a quadratic function,
and Support Vector Machines~\cite{CV1995} where $\Psi_k(x_k)$ is often a quadratic function or a hinge loss (essentially, $\max\{0,x_k\}$).

Due to the enormous size of modern problems, there has been considerable interest in parallel algorithms for the problem
in order to achieve speedup, ideally in proportion to the number of processors or cores at hand, called linear speedup.
One of the most natural parallel algorithms is to simply have each of the multiple cores perform coordinate descent in an (almost) uncoordinated way.
In this work, we analyze the natural parallel version of the standard stochastic version of coordinate descent (SCD):
each core, at each of its iterations, chooses the next coordinate to update  uniformly at random\footnote{There
are also versions of the sequential algorithm in which different coordinates can be selected with different probabilities.}.

One important issue in parallel implementations is whether the different cores are all using up-to-date information for their computations.
To ensure this requires considerable synchronization, locking, and consequent waiting.
Avoiding the need for the up-to-date requirement, i.e., enabling asynchronous updating, was a significant advance.
The advantage of asynchronous updating is to reduce and potentially eliminate the need for waiting.
At the same time, as some of the data being used in calculating updates will be out of date,
one has to ensure that the out-of-datedness is bounded in some fashion.
This is captured by the assumption of $q$-bounded asynchrony: each update can overlap with at most $q$ others;
$q$ is at most the number of cores times the ratio of the lengths of the longest and shortest updates.

\hide{This is the end of the first part in the old introduction.
\begin{itemize}
\item It feels like a sudden ending, so I am adding a bit more information.
\item The second reviewer wrote: ``First of all, the introduction is too long. It discusses many technical details and distinctions
between different approaches without introducing sufficient technical background
and notations. Specifically, the introduction up to first half of page 5 are informative, by
giving relevant high-level background information and stating main contributions. However,
I got lost on page 5 with the description of the Single Coordinate Consistent (SCC) order
(without formal definition) and its interactions with the SCV assumption. Then the discussions on relevant work
from page 6 to page 8 are mostly not clear without first presenting necessary notations and technical details.
These may only make sense for people who either work on exactly the same topic or have first read the technical parts and understand them.
Most of these can be discussed after presenting the main results if necessary.''
\item Except what the second reviewer suggested, I also think that during
the very long period of submitting and modifying this script,
we got a lot of reviews from STOC/FOCS/SODA about changing the introduction (mostly about adding stuff about background and other relevant work);
but now we are submitting to MP-A.
This makes the current introduction seem unnecessarily long. I want to make it shorter.
One way is delegate all discussions not directly related to ACD to an appendix.
\end{itemize}
}

The performance of an asynchronous algorithm is typically measured against its sequential counterpart by the \emph{linear speedup} benchmark:
if $p$ cores are used in the asynchronous algorithm, the running time is a factor of $\Theta(p)$ faster than the sequential counterpart.
In the context of minimizing a convex function, the running time is measured by the convergence rate towards the minimum point.

The asynchronous version of SCD is called \emph{Stochastic Asynchronous Coordinate Descent} (\textsf{SACD}).
The question we address in this paper is:\vspace*{0.08in}

\centerline{\emph{What is the maximum possible value of $\widetilde{q}$ such that whenever $q\le \widetilde{q}$,}}

\centerline{\emph{\textsf{SACD} is guaranteed to achieve linear speedup}}

\centerline{\emph{under the $q$-bounded asynchrony assumption?}}\vspace*{0.08in}

\noindent In some prior analyses, in addition to $q$-bounded asynchrony,
several other seemingly natural assumptions were (implicitly) made, but they are unlikely to hold in practice.
Several works have successfully avoided the use of some or all of these assumptions, but at the cost of having a substantially smaller $\widetilde{q}$.
The main contribution of this paper is to derive the asymptotically best possible value of $\widetilde{q}$,
while avoiding the use of every one of these assumptions.
We now state our result for strongly convex functions informally.

\begin{theorem}[Informal]\label{thm::informal}
Let $q$ be an upper bound on how many other updates a single update can overlap.
$\Lresbar$ and $\Lmax$ are Lipschitz parameters defined in Section \ref{sect:model}.
Let $F(x) = f(x) + \sum_{k=1}^n \Psi_k(x_k) $ be a strongly convex function with strongly convex parameter $\mu_F$, and suppose  $f(x)$ has strongly convex parameter $\mu_f$.
\emph{Without using any additional assumption}, we have:
if $q = O\big(\frac{\sqrt n \Lmax}{\Lresbar}\big)$,
then
$\expect{F(x^{\barT}) - F^*}  \le
\big(1 - \frac{1}{3} \frac{\mu_F}{n(\mu_F - \muf+\Lmax)}\big)^{T}\cdot \left(F(x^1) - F^*\right)$.
\end{theorem}

Standard sequential analyses~\cite{lu2015complexity,richtarik2014iteration} achieve similar bounds with the $\frac 13$ replaced by 2;
i.e., up to a factor of 6, this is the same rate of convergence.
Furthermore, the bound on $q$ is asymptotically tight, as we show in a companion work~\cite{CCT2018}.


Next, we discuss
the assumptions which were used or avoided by the prior works concerning \textsf{SACD}.
We will also compare their bounds on $\widetilde{q}$ with ours.
\hide{
We note that asynchrony itself is an important topic in parallel/distributed computing and in decentralized iterative systems such as market dynamics,
so needless to say there has been much prior work studying asynchrony in other settings. We defer their discussion to Appendix~\ref{app:related-work}.
\RJC{I don't think it is appropriate defer related work to the appendix.}\YKC{Let me think how to handle this.
I agree with one reviewer that the related work part is not very relevant, and it seems to me too long.
Perhaps we can see how to shorten it and put it back to the introduction.}
}

\medskip

\paragraph{Three Assumptions in Prior Work} 
The first analyses to prove rate of convergence bounds for stochastic asynchronous computations 
were those by Avron et al.~\cite{ADG2015} (for the Gauss-Seidel algorithm),
and by Liu et al.~\cite{LWRBS2015} and Liu and Wright~\cite{LiuW2015} (for coordinate descent).
Liu et al.~\cite{LWRBS2015} imposed a ``consistent read'' constraint on the asynchrony;
the other two works considered a more general  ``Inconsistent Read'' model.
Subsequent to Liu and Wright's work, several implicit assumptions, discussed below,
were identified by Mania et al.~\cite{MPPRRJ2017} and Sun et al.~\cite{Sun2017}.

Next, we give precise descriptions of the three assumptions used in prior work, and explain why they might not hold in practice.
Before doing so, we note that when a core makes an update, it typically comprises four steps:
(1) choose a coordinate $k$ to update uniformly at random;
(2) read the values of the coordinates needed for Step 3 from the main memory;
(3) use the coordinate values read in Step 2, denoted by $\tilde{x}$, to compute the gradient $\nabla_k f(\tilde{x})$; and 
(4) use the computed gradient to make an update to the value of coordinate $k$ in the main memory.
\hide{
\ykc{Asynchronous updating comprising the four steps in the above order is called the \emph{Before Read} (BR) approach, as the random choice of coordinate is made before reading the coordinate values.
Leblond et al.~\cite{LPL18} proposed the \emph{After Read} (AR) approach, which swaps the order of Steps 1 and 2.
The AR approach avoids some technical challenges in analyses, but it requires each core to read the whole vector $\tilde{x}$,
which is very inefficient compared with BR when $f$ is \emph{sparse}.
For this reason, BR is the approach many practical implementations take, and it is what our analysis focuses on.}
\RJC{This is not accurate. AR is a labeling of the updates, just
like ST,CT and SCC. Also, Leblond et al. are considering a different problem, not SACD.}
}
In the sequential case, the values read in Step 2 are the most updated, so the $t$-th update will read the values from right after the $(t-1)$-st update. 
But in an asynchronous setting, the values read by each update can be outdated;
Assumption 1 was used by Liu et al.~\cite{LWRBS2015} to constrain the form of this datedness.
In this setting, we let $x^t$ denote the coordinate values in memory right after the $(t-1)$-st update. 

\begin{assume}
\label{asm::CR}[Consistent Read (CR)]
All the coordinate values read by an update computation may have some delay,
but they must appear \emph{simultaneously} at some moment.
Precisely, the values read by the $t$-th update must be of the form $x^{t-\tau}$ for some $\tau\ge 0$.
\end{assume}


It is not hard to see why Assumption~\ref{asm::CR} does not hold in practice: in Step 2, the values of different coordinates are read one-by-one, not simultaneously.
This is why all the later work, including ours, uses the inconsistent read model:
the values read by the $t$-th update can be any of the $(x_1^{t-\tau_1},\cdots,x_n^{t-\tau_n})$,
where
each $\tau_j\ge 0$ and some or all of the $\tau_j$'s can be distinct.

To describe Assumption~\ref{asm::UP}, note that each update takes a non-trivial amount of time to finish, which we call the \emph{timespan} of the update.
Moreover, the timespan of different updates are typically not the same:
in an experimental study, Sun et al.~\cite{Sun2017} showed that iteration lengths in coordinate descent problem instances varied by factors of $2$ to $10$.
Thus, in general, the ordering of the updates based on their starting times (the ST order)
is not the same as the ordering of the updates based on their commit times (the CT order).
It is clear that in the ST order the random choice of coordinate for each update is independent of the other updates, and thus it is uniformly random,
which is a helpful property we desire when analyzing \textsf{SACD}.
In contrast, as illustrated in Example~\ref{expl::UoU} in Section~\ref{sect:model},
in the CT order the choice of coordinate for one update can be influenced by other recently committed updates,
and therefore, conditioned on the history of previous updates, the choice need not be uniformly random;
indeed, it is unclear what the distribution of choices of the coordinate to update becomes. 
We call this the \emph{Undoing of Uniformity}.
However, as first pointed out in Mania et al.~\cite{MPPRRJ2017} (see their Section 3.1),
several earlier works implicitly made the following Assumption~\ref{asm::UP}, which states that the CT order enjoys the same favorable property as the ST order.

\begin{assume}
\label{asm::UP}[Uniformity Preservation (UP)]
When the updates are enumerated using the CT order,
the random choice of coordinate for each update is independent of the other updates, and thus it is uniformly random.
\end{assume}

Avron et al.~\cite{ADG2015} also raised a similar issue w.r.t.~their asynchronous Gauss-Seidel algorithm.

To avoid using Assumption~\ref{asm::UP}, one simple solution is to use the ST order instead of the CT order,
as was done in~\cite{MPPRRJ2017,Sun2017} for the analysis of \textsf{SACD} on smooth functions.
For non-smooth functions, we need a slight twist to the ST order which we call the \emph{Single Coordinate Consistent} (SCC) order;
see Section~\ref{sect:model} for its definition and justification.
However, both the ST and SCC orders create several subtle challenges in the analysis of \textsf{SACD}.
Note that just before the $t$-th update makes its random choice of coordinate,
denoted by $k_t$, some earlier updates might not have committed yet.
We remark that \emph{the choice of $k_t$ might affect the updated values computed by those earlier updates}.
To see why, suppose that $k_t=1$, and the $t$-th update timespan is short. Further, assume no nearby updates pick coordinate $1$.
Then it is possible that the $t$-th update commits earlier than the $(t-1)$-st update,
and therefore the $(t-1)$-st update might read the value of coordinate $1$ computed by the $t$-th update.
For any other random choice of $k_t$, i.e., if $k_t\neq 1$, then as coordinate $1$ has not been updated recently, the $(t-1)$-st update will read an earlier
value of coordinate $1$.
As the reads by the $(t-1)$-st update can differ due to different choices of $k_t$, the change made by the $(t-1)$-st update is influenced by the choice of $k_t$.
Moreover, due to analogous reasoning, for the $t$-th update, the coordinate values it reads when $k_t=1$ can differ from those it reads when $k_t\neq 1$.

The subtlety here is: when we use the ST order, \emph{the ``future'' (an update which appears later in the ST order)
can influence the ``past'' (an update which appears earlier)}.
This apparent confusion of causality creates substantial challenges in obtaining a complete and rigorous analysis;
several prior work chose to bypass the issue with Assumption~\ref{asm::CV} or the stronger Assumption~\ref{asm::CV}* below.
Again, the fact that Assumption~\ref{asm::CV} had been used in earlier work was first pointed out in Mania et al.~\cite{MPPRRJ2017} (see their Assumption 5.1).

\begin{assume}
\label{asm::CV}[Common Value (CV)]
The random choice of coordinate for an update does not affect the values read by the update.
\end{assume}

\noindent \textbf{Assumption~\ref{asm::CV}*}~\emph{[Strong Common Value (SCV)] In addition to Assumption~\ref{asm::CV}, the values read
by an update are independent of subsequent choices of coordinate.
}

Yet another order, named \emph{After Read} (AR), was proposed by Leblond et al.~\cite{LPL18}, albeit for a different problem. Translated to the SACD algorithm, it would require swapping the order of Steps 1 and 2; i.e., in Step 1 all coordinate values are read, and then in Step 2 a random coordinate is chosen to be updated. Clearly, this will be highly inefficient if the problem is sparse.
The AR order would use the times at which Step 2 is started to order the updates; there is no Undoing of Uniformity in this order.
However, it will not suffice for non-smooth functions (see our justification of the SCC order in Section~\ref{sect:model}).


\smallskip

Table~\ref{tbl:results} provides a comparison of our results with prior work.

\newcommand{\Y}{\textbf{YES}}

\setlength\tabcolsep{0.05in}

\begin{table}[tbh]
\centering
\begin{tabular}{|l|c|c|c|c|c|c|} \hline
&  & Maximum &  & \multicolumn{3}{c|}{Avoiding Assumption} \\ 
\cline{5-7}
&  &  Parallelism $q$  &  Non    &  &  & \\
& Step Size &  with linear  & Smooth & CR?  & UP? & SCV? \\
&                 &  speedup  & $\Psi_k$ & & &  \\ \hline
Liu et al.~\cite{LWRBS2015} & $\G \ge \Lmax$ & $ \Theta \left( \frac{\Lmax \sqrt{n}}{\Lres} \right)$ & NO & NO & NO & NO \\ \hline
Liu and Wright~\cite{LiuW2015} & $\G \ge 2 \Lmax$ & $ \Theta \left( \frac{\Lmax \sqrt{n}}{\Lres} \right)^{1/2}$
& \Y & \Y & NO & NO \\ \hline
Mania et al.~\cite{MPPRRJ2017}  & $\G \ge \Theta\left(\frac{L^2}{\muf}\right)$ & See caption & NO & \Y & \Y & NO \\ \hline
Sun et al.~\cite{Sun2017}  &  $\G \ge \Theta(qL)$  & 1 & NO & \Y & \Y & \Y \\ \hline
{\bf Our Result}&  $\G \ge \Lmax$ & $ \Theta \left( \frac{\Lmax \sqrt{n}}{\Lresbar} \right)$
& \Y & \Y & \Y & \Y \\ \hline
\end{tabular}
\caption{\label{tbl:results}Comparisons of the analyses of \textsf{SACD}.
See Definition \ref{def:Lipschitz-parameters} for the specifications of Lipschitz parameters $L$,
$\Lmax$, $\Lres$ and $\Lresbar$; $\muf$ is the strong convexity parameter.
When there is no non-smooth $\Psik$, the update increment is the computed gradient divided by $\G$. Thus,
the larger the $\G$, the less aggressive the update.
Mania et al.~achieve linear speedup compared to the case $q =1$ for $q=O(n^{1/6})$;
however, the case $q=1$ is slower by a factor of $\Theta( L^2/(\muf \Lmax))$ compared to a standard stochastic algorithm.
In~\cite{LiuW2015}, Liu and Wright implicitly used the Strong Common Value (SCV) assumption,
namely that the choice of coordinate for update $t$ does not affect the value of $\tilde{x}^t$ read by update $t$
nor the values read by earlier updates.
This is the reason they can use the parameter $\Lres$ to bound gradient differences.
To avoid using the SCV assumption, we have introduced a new but similar parameter $\Lresbar$.}
\end{table}

\paragraph{Related Work}
Convex optimization is one of the most widely used methodologies in applications across multiple disciplines;
we refer readers to Nesterov's text~\cite{Nesterov2004} for an excellent overview. 
Coordinate Descent is a method that has been widely studied; see Wright~\cite{wright2015coordinate} for a recent survey.
Relevant works concerning sequential stochastic coordinate descent include Nesterov \cite{Nesterov2012},
Richt{\'{a}}rik and Tak{\'{a}}c \cite{richtarik2014iteration}, and Lu and Xiao \cite{lu2015complexity}.

Distributed and asynchronous computation has a long history in optimization,
going back at least to the work of Chazan and Miranker~\cite{ChazanMiranker1969} in 1969,
with subsequent milestones in the work of Baudet~\cite{Baudet1978},
and of Tsitsiklis, Bertsekas and Athans \cite{TBA1986,BertsakisTsiTsi1989};
subsequent results include~\cite{Borkar1998,BT2000}.
For a survey formalizing pre-2000 work, see Frommer and Szyld~\cite{FS2000}.
Also see Avron et al.~\cite{ADG2015} for an informative discussion of asynchronous linear system solvers.
\hide{
\YKC{We can settle with your write using the word ``informative''. Avron discussed on prior work on ACD which we have also addressed.
So I want to point out their discussion asynchronous linear system solver, which includes the very early work of Chazen and Miranker, the ``general'' work on Bertsekas,
and some recent results on asynchronization of classical (sequential) linear system solvers. ``Development'' is a vague descriptive term; I guess ``history'' might be better.
Or simply remove the work ``developments of''.}
}

In the last few years, there have been multiple analyses of various
asynchronous parallel implementations of stochastic coordinate descent~\cite{LWRBS2015,LiuW2015,MPPRRJ2017,Sun2017}.
We have already mentioned the results of
Liu et al.~\cite{LWRBS2015} and Liu and Wright~\cite{LiuW2015}. Both obtained bounds
for both convex and ``optimally'' strongly convex functions\footnote{This is a weakening of the standard strong convexity.},
attaining linear speedup so long as there are not too many cores.
Liu et al.~\cite{LWRBS2015} obtained bounds similar to ours (see their Corollary 2 and our Section~\ref{thm:main-SACD}),
but the version they analyzed is more restricted than ours in two respects:
first, they imposed the strong assumption of consistent reads,
and second, they considered only smooth functions (i.e., no non-smooth univariate components $\Psi_k$).
The version analyzed by Liu and Wright~\cite{LiuW2015} is the same as ours,
but their result requires both the UP and SCV assumptions.
Their bound degrades when the parallelism exceeds $\Theta(n^{1/4})$.\footnote{This is
expressed in terms of a parameter $\tau$, renamed $q$ in this paper, which is essentially the possible parallelism;
the connection between them depends on the relative times to calculate different updates.}
Our bound has a similar flavor but with a limit of $\Theta(n^{1/2})$.

The analysis by Mania et al.~\cite{MPPRRJ2017} removed the UP assumption and needs
only the SCV assumption. 
However, the maximum parallelism was much reduced (to at most $n^{1/6}$),
and their results applied only to smooth strongly convex functions,
and furthermore is efficient only on non-sparse problem instances.
We note that a major focus of their work concerned a simple analysis of \textsf{HOGWILD!},
an asynchronous stochastic gradient descent algorithm
used in data-intensive machine learning tasks, namely to learn
functions of the form $\sum_{e=1}^N f_e(\bbx)$, where $\bbx\in \rr^n$, and each $f_e$ is convex and corresponds to a loss function for one training data instance.
\textsf{HOGWILD!} is due to Niu et al.~\cite{NRRW2011}; it was the first asynchronous and lock-free SGD algorithm,
and it achieves linear speedup on sparse problems.

The analysis in Sun et al.~\cite{Sun2017} removed the CV assumption and partially removed the UP assumption.
However, this came at the cost of achieving no parallel speedup.
They also noted that a hard bound on the parameter $q$
could be replaced by a probabilistic bound, which in practice seems more plausible.

As already mentioned, a companion work~\cite{CCT2018} shows the bound on $q$ in this paper is tight.

\hide{
Avron et al.~\cite{ADG2015}
proposed and analyzed an asynchronous and randomized version of the Gauss-Seidel algorithm
for solving symmetric and positive definite matrix systems.
They pointed out that in practice
delays depend on the random choice of direction (which corresponds to coordinate choice in our case).
Their analysis
bypasses this issue with their Assumption A-4, which states that delays are independent of the coordinate being updated,
but the already mentioned experimental study of Sun et al.~indicates that this assumption does not hold in general.
}

Another widely studied approach to speeding up gradient and coordinate descent is the use of acceleration.
Recently, attempts have been made to combine acceleration and parallelism~\cite{hannah2018a2bcd,fang2018accelerating,CT2018}.
But at this point, these results do not extend to non-smooth functions.

In a companion work, Cheung and Cole~\cite{CC2018} analyzed asynchronous tatonnement in a class of economies for which
tatonnement is equivalent to gradient descent.
They gave worst-case analyses for a special family of convex functions arising in these settings~\cite{CCD2013},
while this work focuses on stochastic analyses.
\hide{
The convex functions studied in~\cite{CC2018} do not have global Lipschitz parameters,
so their update rule needs to be constrained to ensure that their analyses can proceed with local Lipschitz parameters.
}

\hide{
\RJC{I would like to reduce this.}
In statistical machine learning, the objective functions to be minimized typically have the form
$\sum_{e=1}^N f_e(\bbx)$, where $\bbx\in \rr^n$ and each $f_e$ corresponds to a loss function for one training data instance.
Usually, $f_e$ will only depend on a subset of entries in $\bbx$, and we denote this subset by $S_e$.
A well-known algorithm is Stochastic Gradient Descent (SGD), which proceeds by randomly sampling a number of training data instances,
uses the corresponding $f_e$'s to compute an unbiased estimator of the accurate gradient,
which in turn is used to make the standard gradient descent update.
Niu et al.~\cite{NRRW2011} introduced \textsf{HOGWILD!}, the first asynchronous and lock-free SGD algorithm;
\textsf{HOGWILD!} achieves linear speedup for \emph{sparse} problems, i.e., $|S_e|$ is small for every $e$,
each $i\in [n]$ appears only in a small number of $S_e$'s,
and any fixed $S_e$ intersects only a small number of other $S_{e'}$'s.
Many of the assumptions made for asynchronous SGD share similarities with our assumptions.
For instance, Tsianos and Rabbat~\cite{TR2012} extended the analysis of Duchi et al.~\cite{DAW2012}
to analyze distributed dual averaging (DDA) with communication delay;
the same authors \cite{RT2014} studied DDA with heterogeneous systems, i.e.,
distributed computing units with different query and computing speeds.
Langford et al.~\cite{LSZ2009} also studied problems with bounded communication delay.
}

\hide{
In a similar spirit to our analysis, Cheung, Cole and Rastogi~\cite{CCR2012} analyzed asynchronous tatonnement in certain Fisher markets.
\RJC{Remove the next sentence? Actually, maybe just remove the whole paragraph.}
This earlier work employed a potential function which drops continuously when there is no update and does not increase when an update is made.
}

\hide{
\RJC{Drop this paragraph?}
(Stochastic) stateless algorithms
have been studied in the area of distributed computing for a variety of problems,
e.g., packing (positive) linear programming~\cite{AwerbuchK09-SICOMP},
flow~\cite{GargY02,AwerbuchK08-LATIN}, load balancing~\cite{AwerbuchAK08}, and resource allocation~\cite{MarasevicSZ16}.
Most (if not all) of these algorithms presume no communication delay.
``Asynchrony'' refers to the uncoordinated update schedules for different variables,
while ``stochastic'' means the updating schedules are chosen via (independent) random processes.
But whenever an update is made in a round, it is always using the most up-to-date information available right before that round.
Using the terminology of the optimization community, this is spiritually closer to \emph{synchronous} block descent,
but where the block chosen in each round can be quite arbitrary.
}

\paragraph{Our Technical Contributions}
There are two key contributions in our work.
First, we identify an amortization approach for demonstrating convergence amid asynchrony.
Briefly, each update yields a progress term, modulo an error cost which occurs due to asynchrony.
A fraction of the progress per update is used to demonstrate overall progress, while in expectation
the remaining fraction of the total progress can be shown to compensate for the error costs of all the updates.
In short, it is the amortization of progress against errors that leads to our convergence analysis.
With this perspective, it is intuitively clear why we need the bounded asynchrony assumption and the Lipschitz parameter bounds:
the former to control how error blows up with the datedness of information being used,
and the latter to control how one update affects the gradient measurements of  other updates.
When we use the SCV assumption as was done by Liu and Wright~\cite{LiuW2015},
the amortization approach leads to a clean and fairly short analysis, and also improves the parallelism bound given in~\cite{LiuW2015};
see Section \ref{sect:simple}.

While there is no short answer as to why our approach improves the parallel bound
(partly because our analysis is substantially different from the one in~\cite{LiuW2015}),
we point out a notable difference between our analysis and those in~\cite{LiuW2015} and~\cite{MPPRRJ2017}.
In the two prior works, error bounds are \emph{global} in the sense that they involve distance terms
between the current point and the optimal point (see equation (A.18) in~\cite{LiuW2015}, and all the lemmas in Appendix A.1 of~\cite{MPPRRJ2017}).
In contrast, all our error bounds can be kept \emph{local}, i.e., they can be expressed only in terms of the magnitude of an update
and its range of variation, and also of gradient changes due to updates, but the optimal point is not involved in the error bounds at all.

The second key contribution is to provide a rigorous analysis that removes the UP and SCV assumptions.
We give a brief explanation of why this is technically challenging.
The standard stochastic analysis relies on showing an inequality of the following form:
$\expect{F(x^{t+1}) - F(x^*)\,|\, x^t} \le (1-\delta^t)\cdot [F(x^t) - F(x^*)]$ for some positive $\delta^t$. 
To remove the UP assumption, Mania et al.~\cite{MPPRRJ2017} used the ST order, while we use a slight twist (the SCC order);
but with either of these orders, a direct use of the standard stochastic analysis is not possible, since with these orders  the ``future'' can affect the ``past''.

Fundamentally, this apparent confusion in causality occurs because the standard choice of timing notation,
i.e., a single integer parameter for ordering all updates,
is inherently insufficient to represent the wide range of causality patterns in the asynchronous setting.
Consequently, we need to develop a more sophisticated notation which allows us to conveniently capture all possible causality patterns and derive useful error bounds.
The SCV assumption removes the possibility of the future affecting the past, and thus guarantees that $x^t$ is the same regardless the choice of coordinate at time $t$,
which is why it can lead to the aforementioned simple analysis.

One key idea is to judiciously overestimate the error terms affecting the $t$-th update so that they do not
depend on the choice of coordinate by the $t$-th update, which then allows averaging of the error over this choice.
A second observation is that these errors can be expressed in terms of a mutual recursion, which, with the right bounds on $q$, remain bounded.
Very briefly, the mutual recursion provides a way of capturing the maximum possible errors among all possible causality patterns.
We will explain how in Section~\ref{sec::gen-framework}.

\paragraph{Organization of the Paper}
In Section~\ref{sect:model}, we describe our model of asynchronous coordinate descent and state our results.
In Section~\ref{sect:keyidea}, we give a high-level sketch of the structure of our analysis, and
show that with the Strong Common Value assumption we can obtain
a simple analysis of \textsf{SACD}; 
this analysis achieves the maximum possible speedup (i.e.,  linear speedup with up to $\Theta(\sqrt n)$ cores).
Note that this is the same assumption as
in Mania et al.'s result~\cite{MPPRRJ2017}
and less restrictive
than the assumptions in Liu and Wright's analysis~\cite{LiuW2015}.
Then, in Section~\ref{sec::gen-framework}, we give the full analysis of \textsf{SACD}.
All omitted proofs can be found in the appendix.
Also, for the reader's convenience, at the end of this paper, we provide a table of the notation and parameters we use.

%% file: model.tex
\section{Model and Main Results}\label{sect:model}

\newcommand{\Ap}{A^+}
\newcommand{\Am}{A^-}
\newcommand{\pktaum}{x_k^{\tau-1}}
\newcommand{\ptauo}{x^{\tau-1}}
\newcommand{\ptauprm}{x^{\tau+n-1}}
\newcommand{\ptp}{x^{t+1}}

Recall that we are considering
convex functions $F:\rr^n\ra\rr$ of the form
$F(x) = f(x) + \sum_{k=1}^n \Psi_k(x_k)$, where $f:\rr^n\ra\rr$ is a smooth convex function,
and each $\Psi_k:\rr\ra\rr$ is a univariate and possibly non-smooth convex function.
We let $x^*$ denote a minimum point of $F$ and
$X^*$ denote the set of all minimum points of $F$.
Without loss of generality, we assume that $F^*$, the minimum value of $F$, is $0$.

We review some standard terminology.
Let $e_j$ denote the unit vector along coordinate $j$.

\begin{defn}\label{def:Lipschitz-parameters}
The function $f$ is $L$-Lipschitz-smooth if for any $x,\Dx\in\rr^n$, $\|\nabla f(x+\Dx) - \nabla f(x)\| \le L\cdot\|\Dx\|$.
For any coordinates $j,k$, the function $f$ is $\Ljk$-Lipschitz-smooth if for any $x\in\rr^n$ and $r\in\rr$,
$|\nabla_k f(x+r\cdot e_j) - \nabla_k f(x)| \le \Ljk\cdot |r|$;
as is conventional, we write $L_k \triangleq L_{kk}$.
$f$ is $\Lres$-Lipschitz-smooth if, for all $j$, $||\nabla f(x+r\cdot e_j) - \nabla f(x)|| \le \Lres\cdot |r|$.
Let $\Lmax \triangleq \max_{j,k} \Ljk$; we note that if $f$ is twice differentiable, then $\Lmax = \max_{j} L_{jj}$.
Let $\Lresbar \triangleq \max_k \left(\sum_{j=1}^n (\Lkj)^2\right)^{1/2}$.
\end{defn}

Note that if the convex function is $s$-sparse, meaning that each term $\nabla_k f(x)$ depends on at most $s$ variables,
then $\Lresbar \le \sqrt s \Lmax$. When $n$ is huge, it seems plausible that the only feasible problems are going to be sparse ones.

\hide{
\YKC{To Richard: Yixin noted that in optimization literature $\Lmax$ usually refers to $\max_j L_{jj}$. If $f$ is twice-differentiable,
it is not hard to show that $\Lmax ~=~ \max_{j} L_{jj}$, but if $f$ is only differentiable, then it is not so clear what is the relationship
between $\Lmax$ and $\max_{j} L_{jj}$. I suspect that it might be possible $\Lmax$ can be large when $\max_{j} L_{jj}$ remains small.}
}

\paragraph{The Difference Between $\Lres$ and $\Lresbar$}
In general, $\Lresbar\ge \Lres$.
$\Lres=\Lresbar$ when the rates of change of the gradient are constant, as for example in quadratic
functions such as $x^{\mathsf{T}}Ax + bx +c$.
We need $\Lresbar$ because we do not make the Common Value assumption, as we explain at the end of the simple analysis in Section~\ref{sect:keyidea}.
\hide{
We use $\Lresbar$ to bound terms of the
form $\sum_j |\nabla_j f(y^j) - \nabla_j f(x^j)|^2$, where $|y^j_k - x^j_k| \le |\Delta_k|$,
and for all $h,i$, $|y^i_k -y^h_k|, |x^i_k -x^h_k| \le |\Delta_k|$,
whereas in the analyses with the Common Value assumption, the term being bounded is
$\sum_j |\nabla_j f(y) - \nabla_j f(x)|^2$, where $|y_k - x_k| \le |\Delta_k|$;
i.e., our bound is over a sum of gradient differences along the coordinate axes for
pairs of points which are all nearby, whereas the other sum is over gradient differences along the coordinate axes
for the same pair of nearby points.}

By a suitable rescaling of variables, we may assume that $\Ljj$ is the same for all $j$ and equals $\Lmax$.
This is equivalent to using step sizes proportional to $\Ljj$ without rescaling, a common practice.

Next, we define strong convexity.
\begin{defn}
\label{def::str-conv}
Let $f:\rr^n\ra \rr$ be a convex function.
$f$ is strongly convex with parameter $\muf > 0$, if for all
$x,y$, $f(y) - f(x) \ge \langle \nabla f(x), y-x \rangle + \frac 12 \muf ||y-x||^2$.
\end{defn}

\paragraph{The Update Rule}
Recall that in a standard coordinate descent, be it sequential or parallel and synchronous,
the update rule, applied to coordinate $j$,
first computes the \emph{accurate} gradient $g^t_j \triangleq \nabla_j f(x^t)$,
and then performs the update given below.
\begin{align*}
W_j(d,g,x) &~~\triangleq~~ -gd ~-~ \G d^2 / 2 ~-~\Psij(x+d) ~+~ \Psij(x);\\
x_j^{t+1} &~~\leftarrow~~ x_j^t + \argmax_d W_j(d,g^t_j,x_j^t) ~\triangleq~x_j^t + \hdj(g^t_j, x_j^t),
\end{align*}
and $\forall k\neq j,~x_k^{t+1}\leftarrow x_k^t$, where $\G\ge \Lmax$ is a parameter controlling the step size.
As is well known, if $\Psij\equiv 0$, then $\hdj(g^t_j, x_j^t) = -g^t_j/\G$, i.e., it is simply an update in proportion to the gradient.

However, in an asynchronous environment, an updating core (or processor) might retrieve outdated information $\tx^t$ instead of $x^t$,
so the gradient the core computes will be $\tg_j^t 
= \nabla_j f(\tx^t)$, instead of the accurate value $\nabla_j f(x^t)$.
Our update rule, which is naturally motivated by its sequential 
counterpart, is
\begin{equation}\label{eq:update-rule}
x^{t+1}_j \leftarrow x_j^t + \hdj(\tg_j,x_j^t)~\equiv~x_j^t+\Delta \pt_j~~~~~~~~\text{and}~~~~~~~~\forall k\neq j,~x^{t+1}_k \leftarrow x^t_k.
\end{equation}
We call this the the $t$-th update (in the SCC order), and denote it by $\calU_t$. 
\[
\text{We let }\hspace*{0.6in}\hWj(g,x) ~\triangleq~ \max_d W(d,g,x)~\equiv~ \Wj(\hdj(g,x),g,x).\hspace*{0.62in}
\]
Note that $\Wj(0,g,x)=0$; thus $\hWj(g,x) \ge 0$ always.
It is well known that in the synchronous case, $\hWj(\nabla_j f(x^t),x^t_j)$ is a lower bound on the reduction in the value of $F$,
which we treat as the \emph{progress}.
Finally, we let $k_t$ denote the coordinate being updated at time $t$.

\begin{algorithm}
\caption{\textsf{SACD} Algorithm.}
\label{alg:SACD}
\textbf{Input: }The initial point $x^1 = (x^1_1,x^1_2,\cdots,x^1_n)$.\\
~\\
\mbox{Multiple processors use a shared memory. Each processor iteratively repeats the following}
\mbox{six-step procedure, without any global coordination}:
\begin{algorithmic}
\STATE \textbf{Step 1:} Choose a coordinate $j \in \{1,2,\cdots,n\}$ uniformly at random.\label{alg-line-random}
\STATE \textbf{Step 2:} Retrieve coordinate values $\tx^t$ from the shared memory.\label{alg-line-retrieve}
\STATE \textbf{Step 3:} Compute the gradient $\nabla_j f(\tilde{x}^t)$.\label{alg-line-gradient}
\STATE \textbf{Step 4:} Request a write lock on the memory that stores the (up-to-date) value of the\label{alg-line-lock-request}
\STATE \hspace*{0.56in}$j$-th coordinate.\footnotemark
\STATE \textbf{Step 5:} Retrieve the 
$j$-th coordinate value,
then update
it using rule \eqref{eq:update-rule}.\footnotemark \label{alg-update}
\STATE \textbf{Step 6:} Release the lock acquired in Step \ref{alg-line-lock-request}.
\end{algorithmic}
\end{algorithm}

\addtocounter{footnote}{-2}
\stepcounter{footnote}\footnotetext{Instead of having a lock in lines 4--6, a compare-and-swap operation can be used to perform the update in Line 5.
This has the effect of using the hardware lock that is part of the compare-and-swap operation.}
\stepcounter{footnote}\footnotetext{Even if the processor had retrieved the value of the $j$-th coordinate from the shared memory in Step \ref{alg-line-retrieve}, the processor needs to retrieve it again here,
	because it needs the most updated value when applying update rule \eqref{eq:update-rule}.}

\paragraph{The \textsf{SACD} Algorithm}
The coordinate descent process starts at an initial point
$x^1 = (x^1_1,x^1_2,\cdots,x^1_n)$.
Multiple cores then iteratively update the coordinate values.
We assume that at each time, there is exactly one coordinate update which is being written (in Step 5 of the \textsf{SACD} algorithm).
In practice, since there will be little coordination between cores,
it is possible that multiple coordinate values are updated at the same \emph{moment};
but by using an arbitrary tie-breaking rule, we can immediately extend our analyses to these scenarios.

In Algorithm~\ref{alg:SACD}, we provide the complete description of \textsf{SACD}.
The retrieval times for Step \ref{alg-line-retrieve} plus the gradient-computation time for Step \ref{alg-line-gradient} can be non-trivial,
and also in Step \ref{alg-line-lock-request} a core might need to wait if the coordinate it wants to update is locked by another core.
Thus, during this period of time other coordinates are likely to be updated.
For each update, we call the period of time spent performing the six-step procedure the \emph{span} of the update.
We say that update $A$ \emph{interferes with} update $B$ if the commit time of update $A$ lies in the span of update $B$.

Later in this section, we discuss why locking is needed and when it can be avoided;
we also explain why the random choice of coordinate should be made before retrieving coordinate values.

\paragraph{Managing the Undoing of Uniformity: The Single Coordinate Consistent Order}
Before stating our result formally, we need to disambiguate our timing scheme.
In every asynchronous iterative system, including our \textsf{SACD} algorithm,
each procedure runs over a span of time rather than atomically.
Generally, these spans are \emph{not consistent} ---
it is possible for one update to start later than another one but to commit earlier.
To create an analysis, we need a scheme that orders the updates in a consistent manner.

Using the commit times of the updates for the ordering seems the natural choice,
since this ensures that future updates do not interfere with the current update.
This is the choice made in many prior works.
However, as discussed by Mania et al.~\cite{MPPRRJ2017},
this causes uniformity to be undone, as shown in the following example.
\begin{example}
\label{expl::UoU}
Suppose there are three cores and four coordinates,
suppose that the workload for updating $x_1$ is 2.99 time units,
the workloads for updating $x_2, x_3, x_4$ are 1 time unit (the 2.99 is to avoid ties),
and suppose every update takes the same 0.5 time units for Steps 4--6.
Assuming the cores all start at the same time, then $\prob{k_1 = 1} =1/4^3$, which is the probability that all cores choose to update the first coordinate.
Contrariwise,
$\prob{k_2 = 1\,|\, k_1=1} =1$.
And, in general,
the probability distribution which the random variable $k_{t}$ follows
is strongly dependent on the recent history.
\end{example}

When there are many more cores and coordinates than the simple case we just considered,
and when the other asynchronous effects\footnote{E.g., communication delays,
interference from other computations (say due to mutual exclusion when multiple cores commit updates to the same coordinate),
interference from the operating system and CPU scheduling.
\label{fn:async-effects}} are taken into account,
it is highly uncertain what is the exact or even an approximate distribution for $k_{t+1}$
conditioned on knowledge of the history of $k_1,\cdots,k_t$.
\hide{
However, all prior analyses apart from~\cite{MPPRRJ2017} and \cite{Sun2017} proceeded by making the idealized assumption that
the conditional probability distribution remains uniform,
while in fact it may be far from uniform.
While it seems plausible that without conditioning, the $t$-th update to commit
is more or less uniformly distributed, many prior analyses needed this property with
the conditioning, and they needed it for every update without fail.
}

To bypass the above issue, we introduce the \emph{Single Coordinate Consistent
Order}, \emph{SCC} for short, defined as follows.
We begin from the updates ordered by start time.
Then, for each coordinate separately, we rearrange the updates
to this coordinate so that they are in commit order,
while collectively occupying the same places in the start ordering.
The next example illustrates all three orders.
The start times given by the ST order correspond to actual times;
but henceforth, the index $t$ will refer to the position of an update in the SCC order, and to the values computed by these updates.

\begin{example}
\label{ex::showing-orders}
In Figure~\ref{order-example}
we show six updates to two variables, $x_1$ and $x_2$,
starting at times $t=1$ to 6, and ending at times 7--12.
The updates are named $U_1$--$U_6$.
In the order listings below, to facilitate comparisons,
we give each update an argument comprising the variable it updates.

Update Orders:
\begin{align*}
\text{ST:~} &U_1(x_2),U_2(x_1),U_3(x_1),U_4(x_2),U_5(x_1),U_6(x_2) \hspace*{1in}\\
\text{CT:~} &U_1(x_2),U_3(x_1),U_6(x_2),U_2(x_1),U_5(x_1),U_4(x_2)\\
\text{SCC:~} &U_1(x_2),U_3(x_1),U_2(x_1),U_6(x_2),U_5(x_1),U_4(x_2)
\end{align*}
The updates to $x_1$ are in the same positions in the ST and SCC orders, in
the same order in the CT and SCC orders. Likewise for $x_2$.
\end{example}

\begin{figure}[htbp]

\begin{center}
\begin{tikzpicture}
\path
(6,1) node (m) {$t=1$} --
(4,1.5) node (n) {$t=2$} --
(2.5,2) node (o) {$t=3$} --
(7.5,2.5) node (p) {$t=4$} --
(1,3) node (q) {$t=5$} --
(9,3.5) node (r) {$t=6$} --
(6,4) node (s) {$t=7$} --
(2.5,4.5) node (t) {$t=8$} --
(9,5) node (u) {$t=9$} --
(4,5.5) node (v) {$t=10$} --
(1,6) node (w) {$t=11$}  --
(7.5,6.5) node (x) {$t=12$} --
(0.5,0) node (a) {}--
(4.5,0) node (b) {} --
(5.5,0) node (c) {}--
(9.5,0) node (d) {};
\begin{scope}[>=latex]
\draw[->] (m) to node[right] {$U_1$} (s);
\draw[->] (n) to node[left] {$U_2$} (v);
\draw[->] (o) to node[left] {$U_3$} (t);
\draw[->] (p) to node[right] {$U_4$} (x);
\draw[->] (q) to node[left] {$U_5$} (w);
\draw[->] (r) to node[right] {$U_6$} (u);
\draw[<->] (a) to node [below] {updates to $x_1$} (b);
\draw[<->] (c) to node [below] {updates to $x_2$} (d);
\end{scope}
\end{tikzpicture}

\caption{\label{order-example}Illustration of the ST, CT and SCC orders}
\end{center}
\end{figure}
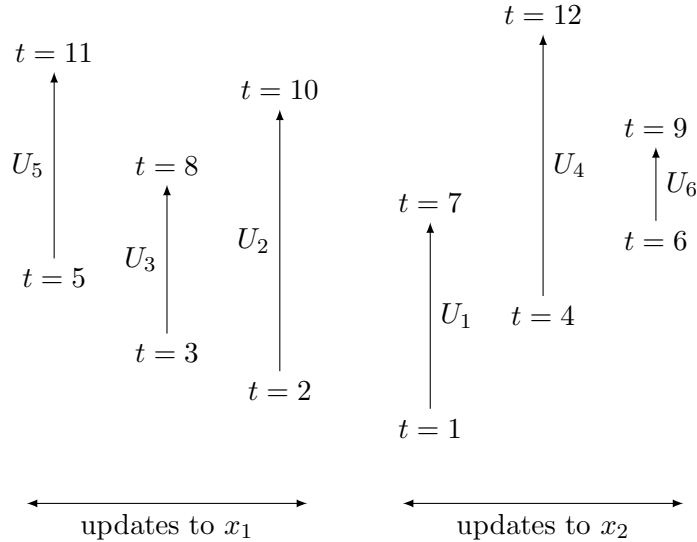

We can also understand the SCC order in terms of start times $1,2,\ldots,T$.
The $t$-th update in the ST order starts at time $t$ and commits
at some integer time in the range $[t+1,t+q+1]$ (this follows from our assumption that the asynchrony is $q$-bounded).
Remember that the ``times'' are simply providing an ordering; they are not measured in a common unit.
$\calU_t$, the $t$-th update in the SCC order, has an integer start time
in the range $[\max\{1,t-q+1\},t+q-1]$ and commits at an integer time in the range $[t+1,t+q+1]$ (see Lemma~\ref{lem::SCCrange}).
\hide{
The earlier it starts, the greater the variation in values in might read,
and so for our analysis, we will assume the start time is $\max\{1,t-q\}$.
We cannot set its commit time in a similar way, however, as its commit time
could affect which other updates might read its committed value.
}

Clearly the history has no influence on the choice of $k_{t+1}$.
However, there is 
a new issue: \emph{future} updates can interfere with the current update.
Here the term future is used w.r.t.~the SCC order;
recall that an update $U_a$ to one coordinate with an earlier starting time can commit later than another later starting update $U_b$
to a different coordinate, and therefore $U_b$ could interfere with $U_a$.

\hide{\paragraph{The Strong Common Value Assumption}
Since the retrievals of coordinate values are performed after choosing the coordinate $k_t$ to update,
and since the schedule of retrievals depends on the choice of $k_t$,
in general it is possible that the retrieved value $\tx$ in Step~\ref{alg-line-retrieve} of Algorithm~\ref{alg:SACD} varies with $k_t$.

Also, a later starting update (update $\calU_A$) can affect updates with earlier starts (updates in B)
if update A commits earlier than some of the updates in B.
One scenario in which this is likely occurs when the iteration lengths are unequal.
Suppose that at time $\tau$ a core $\cal B$ chose coordinate $k_\tau \in B$ to update,
and its update takes $2d$ time units to commit (where $d\ge 3$).
Also, suppose the update is scheduled to read the value of coordinate $j$ at some time after $\tau+d+2$.
At time $\tau+1$, core $\cal A$ \rjc{starts Update A by choosing} a random coordinate to update. If core $\cal A$ chooses coordinate $j$, and if its update takes $d$ time units to commit, then core $\cal B$ will read the value updated by core $\cal A$.
On the other hand, if coordinate $j$ was not chosen recently, then core $\cal B$ will surely read an earlier value of coordinate $j$.

More subtly, even if update $A$ commits after all the updates in $B$, it can still affect the updates in $B$
due to differential delays coming from the operating environment (see Footnote~\ref{fn:async-effects} for examples of such delays).

In~\cite{LiuW2015}, Liu and Wright implicitly used the Strong Common Value assumption,
namely that the choice of coordinate for update $t$ does not affect the value of $\tilde{x}^t$ read by update $t$
nor the values read by earlier updates.
This is the reason they can use the parameter $\Lres$ to bound gradient differences.
To avoid using the Strong Common Value assumption, we have introduced a new but similar parameter $\Lresbar$.}

\paragraph{Further Remarks about the \textsf{SACD} Algorithm}
%
%
In many optimization problems, e.g., those involving sparse matrices, the number of coordinate values needed
for computing the gradient in Step \ref{alg-line-gradient} of Algorithm~\ref{alg:SACD} is much smaller than $n$,
i.e., in Step \ref{alg-line-retrieve}, the core needs to retrieve only a tiny portion of the full set of coordinate values.
Also, the sets of coordinate values needed for computing the gradients along different coordinates can be very different.
Therefore, the random choice of coordinate (in Step \ref{alg-line-random})
should 
be made
ahead of the process of retrieving required information from the shared memory.

If the convex function $F$ does not have the univariate non-smooth components,
each update simply adds a number, which depends only on the computed gradient, to the current value in the memory.
Then the update can be done atomically (e.g., by fetch-and-add\footnote{The fetch-and-add CPU instruction atomically increments
	the contents of a memory location by a specified value.}), and no lock is required.

However, for general scenarios with univariate non-smooth components,
the update to $x_j$ must depend on the value of $x_j$ in memory right before the update
(see \eqref{eq:update-rule}). 
Then the update cannot be done atomically, and a lock is necessary.
We note that when the number of cores is far fewer than $n$, say when it is $\epsilon \sqrt{n}$ for some $\epsilon < 1$,
delays due to locking can occur, but are unlikely to be significant.\footnote{The standard \emph{birthday paradox} result states that
	if $\epsilon \sqrt{n}$ cores each chooses a random coordinate among $[n]$ uniformly, the probability of a collision is $\Theta(\epsilon^2)$.}
As already mentioned, even if the update is carried out using a Compare-and-Swap operation,
the lock is still present within the hardware implementation of this operation.

\paragraph{Justifying the SCC Order}
We begin by justifying why Step 5 in the update algorithm needs to use the most up-to-date value of $x_1$ (or more generally, of $x_j$), via the the following convex function example.
\begin{example}
\label{ex::SCC-need}
Let $\overline{F}$ be a convex function on $n-1$ variables. Then define the $n$ variable convex function $F$ as follows.
\begin{align*}
F(x) = \frac 12 x_1^2 + \Psi_1(x_1) + \overline{F}(x_2,\ldots,x_n) \\
\text{where } \Psi_1(x_1) = \left\{ \begin{array}{cl}0 & ~\text{ if } -1 \le x_1 \le \frac 12\\ \infty & ~\text{ otherwise} \end{array} \right. 
\end{align*}
Suppose $\G = 1$, and suppose $x_1^t= -1$.
Further suppose the $t$-th and $(t+1)$-st updates are both to $x_1$,
and suppose they both read the value $x_1^t$ for their gradient computation.
Then they compute increments $\argmax\{d-d^2/2 - \Psi(x_1+d) + \Psi(x_1)\}$.
If they both used the value $-1$ for $x_1$ they would both
increment $x_1$ by $+1$; the two updates would result in $x_1$ being set
to $1$, a value for which $F = \infty$.
\end{example}

Note that update rule \eqref{eq:update-rule} implies that
the sequence $x^1,x^2,\ldots,x^{T+1}$ is obtained by applying the
computed increments $\Delta x^1,\Delta x^2,\ldots \Delta x^T$,
one at a time, and in this order.
For this order to be consistent with Step 5 using
the most up-to-date value, we need that in this order, for each individual coordinate, the updates be in their up-to-date
order, i.e.\ in their commit order. This is why we use the SCC order for our analysis.
In the next example, we will show that an analysis based on the ST order need not work when $F$ has a non-smooth part $\Psi$.

\begin{example}
\label{ex::SCC-need-extend}
Let $\overline{F}$ be a convex function on $n-1$ variables. Then define the $n$ variable convex function $F$ as follows.
\begin{align*}
F(x) = \frac{1}{2}x_1^2 + \Psi_1(x_1) + \overline{F}(x_2,\ldots,x_n) \\
\text{where } \Psi_1(x_1) = \left\{ \begin{array}{cl}0 & ~\text{ if }  x_1 \ge -1 \\ 
\infty & ~\text{ otherwise} \end{array} \right. 
\end{align*}
Suppose $\G = 1$, and suppose $x_1^0= -1$. 
Further suppose there are three consecutive updates to $x_1$:
\begin{itemize}
\item Updates 1--3 start at times 1--3 respectively.
\item At time $4$, Updates 2 and 3 read the value of $x_1$ (which equals $-1$)
and calculate the gradient w.r.t.\ $x_1$ ($\nabla_{x_1} f = -1$).
\item At times 5 and 6 respectively, Updates 2 and 3 apply the 
update on $x_1$ ($x_1' \la x_1 - \argmax_d\{\frac{1}{\G} \nabla_{x_1} f\cdot d -\G d^2/2 -\Psi_1(x_1+d) +\Psi_1(x_1)\}$).  A simple calculation shows that both these updates increments
$x_1$ by $1$.
Therefore, after time $6$, 
the most up to date value of $x_1$ is $-1 + 1 + 1 = 1$. 
\item At time $7$, 
Update 1 reads the value of $x_1$ (which now equals $1$ after applying Updates 2 and 3) 
and calculates the gradient w.r.t.\ $x_1$ ($\nabla_{x_1} f = 1$);
\item Finally, at time $8$, Update 1 
applies the 
update on $x_1$. After this, the most up to date value of $x_1$ is $0$. 
\end{itemize}
In this example, the values of $\Delta x$ for Updates 1--3 are respectively $-1$, $1$, and $1$. 
If we use the ST order and apply Update 1 first, then after this update, the value of $x_1$ becomes $x_1^0 - 1 = -2$, which is less than $-1$ and thus $F(x) = \infty$.
In contrast, with the SCC order, as these updates are to the same coordinate, we will apply these $\Delta x$ based on the commit time order, and then $F(x)$ will never be $\infty$.
\end{example}

\subsection{Results}\label{subsect::results}
We assume that our algorithms are run until exactly $T$ coordinates are selected and then updated for some pre-specified $T$.
The initial value of $x$ is denoted by $x^1$, and the first update in each order is said to be at time $t=1$ w.r.t.~that order.
The commit times are constrained by the following assumption.

\begin{assume}\label{assume:SACD-q}
There exists a non-negative integer $q$ such that
the only updates that might interfere with the update at time $t$ in the ST order are those
that commit at times $t+1, t+2,\ldots,t+q$.
\end{assume}
When asynchronous effects are moderate, and if the various gradients have a similar computational cost,
the parameter $q$ will typically be bounded above by a small constant times the number of cores.

As we are using the SCC order, we need to express the constraint
in terms of the latter ordering.
\begin{lemma}\label{lem::SCCrange}
Let $\calU_t$ be the $t$-th update in the SCC order.
Its (integer) start time lies in the range $[\max\{1,t-q+1\},t+q-1]$
and its commit time is in the range $[t+1, t+q+1]$.
Also, update $\calU_s$ in the SCC order might 
interfere with $\calU_t$ only if $s \in [t-2q+1,t+q-1]$.
\end{lemma}

For simplicity, we relax the first range to $[t-2q,t+q]$.
Also the earlier an update starts, the greater the variation in values in might read, and so for our analysis, we will assume the start time is $\max\{1,t-q+1\}$.
We cannot set $\calU_t$'s commit time in a similar way, however, as its commit time
could affect which other updates might read its committed value.

\begin{theorem}[\textsf{SACD} Upper Bound]
\label{thm:main-SACD}
Given initial point $x^1$, Algorithm~\ref{alg:SACD} is run for exactly $T$ iterations by multiple cores.
Suppose that Assumption~\ref{assume:SACD-q} holds,
$\G \ge \Lmax$,
and $q \le \min\big\{\frac{\sqrt n}{270}, \frac{\G \sqrt n } {270\Lresbar}\big\}$.
\noindent
(i)~
If $F$ is strongly convex with parameter $\muF$,
and $f$ is strongly convex with parameter $\muf$, then
\begin{equation*}
\mathbb{E}\Big[F(x^{\barT})\Big] \le \Big[ 1 - \frac {1}{3n} \cdot \frac{\muF} {\muF + \G - \muf} \Big]^{T} \cdot F(x^1). 
\end{equation*}

\noindent
(ii)~Now suppose that $F$ is convex.
Let $R$ be the radius of the level set for $x^1$,
$\mathrm{Level}(x^1) = \{ x \,|\, f(x) \le f(x^1) \}$.
Then
\begin{equation*}
\expect{F(x^{\barT})} \le \frac{1}{1 + \min\big\{\frac 1{12n}, \frac{F(x^1)}{24n\G R^2}\big\} \cdot ~ T}\cdot F(x^1).
\end{equation*}
\end{theorem}

In a companion paper~\cite{CCT2018}, we show that the first bound is tight up to constant factors. 
Specifically, for any constant $c\ge 1$,
for $q \ge \frac{74 \G \sqrt n}{\Lresbar} + 96 c \ln n + 435$,
we give a family of convex functions for which, with probability at least
$1 - 1/n^c$, the first $n^c$ updates make essentially no progress toward the optimum.
This result holds even for smooth convex functions.
There remains a constant factor separating the upper and lower bounds,
and in this range we do not know how much if any parallel speed-up is possible.

\hide{
In the main body of the paper, we focus on the case of strongly convex $F$.
The full result (including the plain convex case) is in Appendix~\ref{sec:complete-SACD}.
}

\paragraph{Problem Instances with large $\Lres$ and $\Lresbar$}
Both $\Lres$ and $\Lresbar$ can be as large as $\sqrt n \cdot \Lmax$.
For problem instances of this type,
the bound on $q$ becomes $\Theta(1)$; i.e., it does not demonstrate any parallel speedup.
\hide{
\yxt{But possibly, in reality, the worst case in the theoretical analysis of asynchronous behavior may not happen.}
\yxt{Therefore,} it is still conceivable that parallel speedup \yxt{may} occur in practice, but to provide a confirming analysis would require new assumptions on the asynchronous behavior, and we leave the devising of such assumptions as an open problem.
\RJC{I suggest we take out the last two sentences. I am not sure they are useful.}\YKC{Do you mean starting from ``But possibly''? I agree.}
}

%% file: keyidea.tex
\section{The Basic Framework}\label{sect:keyidea}

\newcommand{\xktau}{x_{k_\tau}}
\newcommand{\pktautm}{x_{k_{\tau}}^{t-1}}
\newcommand{\gktaut}{g_{k_{\tau}}^t}


Recall that $\kt$ denotes the index of the coordinate that is updated at time $t$.
We let
$\gkt := \nabla_{k_t} f(x^t)$ denote the value of the gradient along coordinate $\kt$ computed at time $t$
using up-to-date values of the coordinates, and $\tgkt$ denote the actual value computed, which may use some out-of-date coordinate values.
\subsection{Classical Analysis of Stochastic Sequential 
 Coordinate Descent}

This classical analysis 
proceeds by first showing that
for any chosen $k_t$, $F(x^t) - F(x^{t+1}) ~\ge~ \hWkt(\gkt,x_{k_t}^{t})$.
Taking the expectation yields
\begin{align}
\nonumber
\expect{F(x^t) - F(x^{t+1})} & \ge \frac 1n~\sum_{j=1}^{n}~ \hWj(g_j^t,x_j^{t})\\
\nonumber
& \ge \frac 1n \cdot \frac{\mu_F}{\mu_F + \G - \mu_f}\cdot F(x^t) \comm{by~\cite[Lemmas 4,6]{richtarik2014iteration}}\\
& \triangleq \frac{\alpha}{n} \cdot F(x^t). \label{eqn::progress-basic-formula}
\end{align}
Note that in strongly convex case, we have defined $\alpha \triangleq \frac{\mu_F}{\mu_F + \G - \mu_f}$.
It follows that
$\expect{F(x^{t+1})}\le (1-\frac{\alpha}{n})\cdot \expect{F(x^t)}$; iterating this inequality yields
$\expect{F(x^{t+1})}\le (1-\frac{\alpha}{n})^t\cdot F(x^1)$.

\hide{To handle the case where \emph{inaccurate} gradients are used, we employ the following two lemmas.

\begin{lemma}\label{lem:F-prog-one}
If $\G \ge \Lmax$,
$F(x^t) - F(x^{t+1}) ~\ge~ \hWkt(\gkt,x_{k_t}^t) 
- \frac {1}{\G}\cdot (\gkt - \tgkt)^2.$
\end{lemma}

\begin{lemma}\label{lem:F-prog-two}
If $\G \ge \Lmax$,
$	F(x^t) - F(x^{t+1}) ~\ge~ \frac 14 \G \left( \Delta \pktt \right)^2 - \frac {1}{\G} \cdot (\gkt - \tgkt)^2$.
\end{lemma}
Proving these results for smooth functions is straightforward.
The version for non-smooth functions is less simple,
and makes use of the SCC ordering;
it follows from Lemma~\ref{lem:discrete-improvement-of-F} in Appendix~\ref{sect:appendix-general}.

Combining Lemmas~\ref{lem:F-prog-one} and~\ref{lem:F-prog-two} yields
\begin{equation}\label{eqn:F-prog}
F(x^t) - F(x^{t+1}) ~\ge~ \frac 12 \cdot \hWkt(\gkt,x^t_{k_t}) 
~+~ \frac \G 8 \cdot \left( \Delta \pktt \right)^2
~-~ \frac {1}{\G} \cdot (\gkt - \tgkt)^2.
\end{equation}
Taking expectations and using the bound on $\frac 1n~\sum_{j=1}^{n}~ \hWj(g_j^t,x_j^{t})$
from \eqref{eqn::progress-basic-formula} yields
\begin{align*}
\expect{ F(x^{t+1})} \le \left(1 - \frac{\alpha}{2n}\right)F(x^t)
- \expect{\frac \G 8 \cdot \left( \Delta \pktt \right)^2
~-~ \frac {1}{\G} \cdot (\gkt - \tgkt)^2}.
\end{align*}

To obtain $\expect{ F(x^{t+1})} \le \left(1 - \frac{\alpha}{2n}\right)F(x^0)$, it suffices to show that in expectation,
\hide{
The first term on the RHS of the above inequality, after taking the expectation, is more or less what is needed in order to demonstrate progress.
To complete the analysis we need to show that in expectation,}
\begin{align}
\label{eqn::basic-amort-inequality}
\sum_{t=1}^T \frac {\Gamma}{8} \cdot \expect{\left(\Delta \pktt \right)^2} \left(1 - \frac{\alpha}{2n}\right)^{T-t} \ge \sum_{t=1}^T\frac 1{\Gamma} \cdot \expect{(\gkt - \tgkt)^2} \left( 1 - \frac{\alpha}{2n}\right)^{T-t}.
\end{align}
\hide{
for then we can conclude that
$\expect{~F(x^{T+1})~} \le (1 - \frac{\alpha}{2n})^T \cdot F(x^{1})$.
}

\hide{
But we want $\expectsub{j}{\expectsub{k}{\hWk(g_{k}^{j,t},x_k^{j,t-1},\G,\Psik)}}$,
where $g_{k}^{j,t}$ indicates the value of $g_{k}$ at time $t$ had coordinate
$j$ rather than coordinate $k$ been selected, and similarly for $x_k^{j,t-1}$;
but we only have
$\expectsub{\kt}{ \hWkt(\gkt,\pkttm,\G,\Psikt)}= \expectsub{k}{\hWk(g_{k}^{k,t},x_{k}^{k,t-1},\G,\Psik)}$.
The two expectations would be the same if the Common Value assumptions held.
Without these assumptions, our remedy is to devise new shifting lemmas to bound the cost of changing the arguments in $\hWkt$.
To complete the analysis we need to compensate for the negative final term on the RHS which is due to the gradient difference.
The compensation comes from the middle term on the RHS --- this is where the amortization begins.
Precisely, we want to bound a gradient difference of the form $(\gjo - \gjtw)^2$ in terms of the $(\Delta x_{k_t}^t)^2$,
which we approach by using the Lipschitz parameters for the gradients as specified in Definition \ref{def:Lipschitz-parameters}.
Specifically, for $x^1,x^2\in \rr^n$,
for any $k$, let $\Dx_k ~:=~ x^1_k - x^2_k$, and
for $i=1,2$, let $g_j^i ~:=~ \nabla_j f(x^i)$. Then
\begin{equation}\label{eqn:g-diff-multi-L-bound}
\left( g_j^1 - g_j^2 \right)^2 \le~ \left[\sum_{k=1}^n L_{kj} \left|\Dx_k\right| \right]^2,
\end{equation}
where the RHS, the square of a sum, will in turn be bounded by a sum of squares via the Cauchy-Schwarz inequality.
The challenge is to combine these equations while minimizing $\G$.
}}

%% file: simple-analysis.tex
\subsection{Warm-up: A 
Simple Analysis for the Strongly Convex Case with the Strong Common Value Assumption}\label{sect:simple}

The following analysis already generalizes and improves the results shown in Liu et al.~\cite{LWRBS2015} and Liu and Wright~\cite{LiuW2015}.


Suppose there are a total of $T$ updates. We view the whole stochastic process as a branching tree of height $T$.
Each node in the tree corresponds to the \emph{moment} when some core randomly picks a coordinate to update,
and each edge corresponds to a possible choice of coordinate.
We use $\pi$ to denote a path from the root down to some leaf of this tree.
A superscript of $\pi$ on a variable will denote the instance of the variable on path $\pi$.
Note that for each path $\pi$ we reorder the coordinate instances so that they are in the SCC order.
For each path $\pi$ and for each coordinate $k$, this simply 
reorders the instances of $x_k$ on path $\pi$.

Contrary to intuition, in general we cannot associate a single value of $x$
with each node of the tree because future choices of coordinate to update
can affect the recent past; thus we need to specify the path
in order to know a coordinate value.
In contrast, 
the SCV assumption ensures there is a single value of $x$ for each node.
A double superscript of $(\pi,t)$ will denote the instance of the variable at time $t$ on path $\pi$, i.e., right before the $t$-th update.

As we will be computing expected values by averaging over the $n$ random coordinate choices for the $t$-th update $\calU_t$, 
we introduce a notation to capture this choice:
$\pi(k,t)$ will denote the path with the time $t$ coordinate on path $\pi$ replaced by coordinate $k$.
Note that $\pi(k_t,t) = \pi$. (Recall that $k_t$ is the coordinate chosen by  update $\calU_t$ on path $\pi$.)

Recall that $x^{\pi,t}$ denotes the value of $x$ 
on path $\pi$ when precisely the first $t-1$ updates in the SCC order have been applied;
however, $x^{\pi,t}$ may or may not actually be present in memory at any time.
Also, recall that $x_{k_t}^{\pi,t+1} = x_{k_t}^{\pi,t} + \Delta x^{\pi,t}_{k_t}$,
and $x_{k}^{\pi,t+1} = x_k^{\pi,t} $ for $k\ne k_t$, where $\Delta x^{\pi,t}_{k_t}$ is the increment computed by $\calU_t$.
So $x_{\kt}^{\pi(k,t),t}$ denotes the value of $x_{\kt}$ on path $\pi(k,t)$ immediately prior to update $\calU_t$,
and $x_{\ks}^{\pi(k,t),s}$ denotes the value of $x_{\ks}$ on path $\pi(k,t)$ immediately prior to update $\calU_s$.

Similarly, $g_{\kt}^{\pi,t} \triangleq \nabla_{k_t} f(x^{\pi,t})$ denotes the true (accurate) gradient on path $\pi$ immediately prior to update $\calU_t$,
$\tilde{g}_{\kt}^{\pi,t}$ the \emph{inaccurate} gradient used by $\calU_t$ on path $\pi$,
and $\tilde{g}_{k}^{\pi(k, t),t}$ the \emph{inaccurate} gradient used by $\calU_t$ on path $\pi(k, t)$.

To handle the case where \emph{inaccurate} gradients are used, we employ the following two lemmas.

\begin{lemma}\label{lem:F-prog-one}
If $\G \ge \Lmax$,
$F(x^{\pi, t}) - F(x^{\pi, t+1}) ~\ge~ \hWkt(g^{\pi, t}_{k_t}, x_{k_t}^{\pi, t}) 
- \frac {1}{\G}\cdot (g^{\pi, t}_{k_t} - \tilde{g}^{\pi, t}_{k_t})^2.$
\end{lemma}

\begin{lemma}\label{lem:F-prog-two}
If $\G \ge \Lmax$,
$	F(x^{\pi, t}) - F(x^{\pi, t+1})  ~\ge~ \frac \G 4 \left( \Delta  x^{\pi, t}_{k_t} \right)^2 - \frac {1}{\G} \cdot (g^{\pi, t}_{k_t} - \tilde{g}^{\pi, t}_{k_t})^2$,
where $\Delta  x^{\pi, t}_{k_t}$ denotes the increment computed by update $\calU_t$.
\end{lemma}

Proving these results for smooth functions is straightforward. The version for non-smooth functions is less simple, and makes use of the SCC order;
it follows from Lemma~\ref{lem:discrete-improvement-of-F} in Appendix~\ref{sect:appendix-general}.

Combining Lemmas~\ref{lem:F-prog-one} and~\ref{lem:F-prog-two} yields
\begin{equation}\label{eqn:F-prog}
F(x^{\pi, t}) - F(x^{\pi, t+1}) ~\ge~ \frac 12 \cdot \hWkt(g^{\pi, t}_{k_t}, x_{k_t}^{\pi, t}) 
~+~ \frac \G 8 \left( \Delta  x^{\pi, t}_{k_t} \right)^2
~-~ \frac {1}{\G} \cdot (g^{\pi, t}_{k_t} - \tilde{g}^{\pi, t}_{k_t})^2.
\end{equation}

As we will see, the following claim is one reason why this analysis, which uses the SCV assumption, is much simpler than that for the fully asynchronous setting.

\begin{clm}\label{clm:same-older-value}
With the SCV assumption, (i) for any $\tau\le t$, $\tilde{x}^{\pi(k, t), \tau}$ is the same for any coordinate $k$, and thus equals $\tilde{x}^{\pi, \tau}$;
(ii) for any $\tau\le t$, $x^{\pi(k, t),\tau}$ is the same for any coordinate $k$, and thus equals $x^{\pi, \tau}$.
\end{clm}

\begin{proof}
Part (i) follows directly from the SCV assumption.

For part (ii), we argue inductively on $\tau$ as follows.
Suppose the claim holds for earlier times.
By part (i), for any two coordinates $k,k'$, $\tilde{x}^{\pi(k', t),\tau-1} = \tilde{x}^{\pi(k, t), \tau-1}$ {for $\tau \leq t$}.
Thus the computed gradients for update $\calU_{\tau-1}$ are the same on  paths $\pi(k', t)$ and $\pi(k, t)$.
Also, as the claim holds for earlier times, the values read on both paths for Step 5 of update $\calU_{\tau-1}$ are the same,
meaning that these updates are identical and hence so are the outcome of  these updates;
i.e., $x^{\pi(k', t), \tau} =  x^{\pi(k, t),\tau}$.  As this is true for all $\tau \leq t$, the claim follows.
\end{proof}

We take expectations over all paths $\pi$ on both sides of inequality~\eqref{eqn:F-prog}.
We compute the expectation of $\frac 12 \cdot \hWkt(g^{\pi, t}_{k_t}, x_{k_t}^{\pi, t})$ as follows.
We group each collection of $n$ paths which differ only on their $t$-th coordinate choice; in other words, if $\pi$ is a path in a group,
then the $n$ paths in the group are $\pi(1,t),\pi(2,t),\ldots,\pi(n,t)$.
We first take the expectation within each group, which is the summation $\frac 1{2n}\cdot \sum_{k=1}^{n} \hWk(g^{\pi(k,t), t}_{k}, x_{k}^{\pi(k,t), t})$.
By Claim \ref{clm:same-older-value}(ii), $x^{\pi(k,t), t}$ $= x^{\pi,t}$ for any coordinate $k$, and hence $g^{\pi(k,t), t}_{k} = \nabla_{k} f(x^{\pi(k,t),t}) = \nabla_{k} f(x^{\pi,t}) = g^{\pi,t}_k$.
Thus the summation simplifies to
$\frac 1{2n}\cdot \sum_{k=1}^{n} \hWk(g^{\pi, t}_{k}, x_{k}^{\pi, t})$, which is at least $\frac \alpha{2n} \cdot F(x^{\pi,t})$ by inequality \eqref{eqn::progress-basic-formula}.
Then we take the expectation over all groups to obtain
\begin{align*}
\expect{ F(x^{\pi, t+1})} \le &\left(1 - \frac{\alpha}{2n}\right) \cdot \expect{F(x^{\pi, t})}\\
&\hspace*{0.3in}- \mathbb{E}\Big[\frac \G 8 \cdot \left( \Delta  x^{\pi, t}_{k_t} \right)^2
	~-~ \frac {1}{\G} \cdot (g^{\pi, t}_{k_t} - \tilde{g}^{\pi, t}_{k_t})^2\Big].
\numberthis\label{eqn::basic-progress-with-error}
\end{align*}

To obtain $\expect{ F(x^{T+1})} \le \left(1 - \frac{\alpha}{2n}\right)^T \cdot F(x^1)$, it suffices to show that 
\begin{align}
\label{eqn::basic-amort-inequality}
\sum_{t=1}^T \frac {\Gamma}{8} \cdot \expect{\left(\Delta  x^{\pi, t}_{k_t} \right)^2} \left(1 - \frac{\alpha}{2n}\right)^{T-t} \ge \sum_{t=1}^T\frac 1{\Gamma} \cdot \expect{(g^{\pi, t}_{k_t} - \tilde{g}^{\pi, t}_{k_t})^2} \left( 1 - \frac{\alpha}{2n}\right)^{T-t}.
\end{align}

In the remainder of this section we give a simple proof of the above inequality.
We first prove Lemma~\ref{lem::super::simple} below, which bounds the expectation, within each group of $n$ paths, of the gradient differences squared.

\begin{lemma} \label{lem::super::simple}
With the Strong Common Value assumption,
\begin{align*}
\mathbb{E}_{k} [(\tilde{g}_{k}^{\pi(k,t),t} - g_{k}^{\pi(k,t),t})^2]
~\le~ \frac{3 q \Lres^2}{n} \sum_{s \in [t - 2q, t + q] \setminus \{t\}}
\mathbb{E}_{k}  [ (\Delta x^{\pi(k,t),s}_{k_s})^2 ].
\end{align*}
\end{lemma}
\begin{proof}
By definition, $g_{k}^{\pi(k,t),t} = \nabla_{k} f(x^{\pi(k,t),t})$, the gradient of up-to-date point $x^{\pi(k,t),t}$,
and $\tilde{g}_{k}^{\pi(k,t),t} = \nabla_{k} f(\tilde{x}^{\pi(k,t),t})$, the gradient of the point actually read from main memory, out-of-date point $\tilde{x}^{\pi(k,t),t}$.
By Lemma~\ref{lem::SCCrange}, the updates $\calU_s$, for $s <t-2q$, have been written into memory before update $\calU_t$ starts.
Thus, the difference between $x^{\pi(k,t),t}$ and $\tilde{x}^{\pi(k,t),t}$ is due to a subset $U$ of the updates $\calU_s$ with $s\in [t - 2q, t + q] \setminus \{t\}$.
Let
\begin{align*}
U = \{  t_1, t_2, ... , t_{|U|}\}.
\end{align*}

Viewing $\Delta {x}^{\pi(k,t),t_i}_{k_{t_i}}$ as the $n$-vector with a non-zero entry for coordinate $k_{t_i}$ and zero elsewhere, we have:

\begin{align*}
\hspace*{0.9in}x^{\pi(k,t),t} = \tilde{x}^{\pi(k,t),t} + \sum_{i = 1}^{|U|}
\left\{ \begin{array}{ll}
         \Delta {x}^{\pi(k,t),t_i}_{k_{t_i}} & \mbox{if $t_i < t$};\\
        -\Delta {x}^{\pi(k,t),t_i}_{k_{t_i}} & \mbox{if $t_i > t$}.\end{array}
         \right.
\end{align*}

\begin{align*}
\text{We define:}~~~~x^{\pi(k,t),t}[j] = \tilde{x}^{\pi(k,t),t} + \sum_{i = 1}^{j} \left\{ \begin{array}{ll}
         \Delta {x}^{\pi(k,t),t_i}_{k_{t_i}} & \mbox{if $t_i < t$};\\
        -\Delta {x}^{\pi(k,t),t_i}_{k_{t_i}} & \mbox{if $t_i > t$}.\end{array} \right.
\end{align*}

Then, $x^{\pi(k,t),t}[0] = \tilde{x}^{\pi(k,t),t}$ and $x^{\pi(k,t),t}[|U|] = x^{\pi(k,t),t}$.
By the definition of $\Lres$ and the triangle inequality, we obtain

\begin{align}
\nonumber
&\left\|\nabla f(\tilde{x}^{\pi(k,t),t}) - \nabla f(x^{\pi(k,t),t}) \right\|^2\\
\nonumber
&\hspace*{0.2in} \leq \bigg( \sum_{j = 0}^{|U| - 1} \left\|\nabla f(x^{\pi(k,t),t}[j+1]) - \nabla f(x^{\pi(k,t),t}[j])  \right\|\bigg)^2 \\
&\hspace*{0.2in}  \leq  \bigg(\sum_{i = 1}^{|U|} L_{\res} \left|\Delta {x}^{\pi(k,t),t_i}_{k_{t_i}}\right| \bigg)^2
\le  3q \sum_{s\in [t-2q,t+q] \setminus \{t\} } L_{\res}^2 \left(\Delta {x}^{\pi(k,t),s}_{k_{s}} \right)^2.
\label{eqn::simple-bound}
\end{align}

The last inequality followed from applying the Cauchy-Schwarz inequality to the RHS, and relaxing $U$ to $[t - 2q, t + q] \setminus \{t\}$.

By Claim~\ref{clm:same-older-value}(i), 
$\tilde{x}^{\pi(k', t),t} = \tilde{x}^{\pi(k, t), t}$.
By Claim~\ref{clm:same-older-value}(ii), $x^{\pi(k', t),t} =  x^{\pi(k, t), t}$.
Thus,

\begin{align*}
\mathbb{E}_{k}\Big[(\tilde{g}_{k}^{\pi(k,t),t} - g_{k}^{\pi(k,t),t})^2\Big]
& = \mathbb{E}_{k}\Big[|\nabla_{k} f(\tilde{x}^{\pi(k, t), t}) - \nabla_{k} f(x^{\pi(k, t), t}) |^2\Big]  \\
& = \frac 1n \sum_{k'}|\nabla_{k'} f(\tilde{x}^{\pi(k', t), t}) - \nabla_{k'} f(x^{\pi(k', t), t}) |^2 \\
& = \frac 1n \sum_{k'}|\nabla_{k'} f(\tilde{x}^{\pi(k, t), t}) - \nabla_{k'} f(x^{\pi(k, t), t}) |^2 \\
&= \frac{1}{n} 
\cdot \|\nabla f(\tilde{x}^{\pi(k,t),t}) - \nabla f(x^{\pi(k,t),t}) \|^2\\
\numberthis\label{eqn::basic-error-bound}
& \le \frac{3 q \Lres^2}{n} \sum_{s \in [t - 2q, t + q] \setminus \{t\}} (\Delta x^{\pi(k,t),s}_{k_s})^2~~\text{(by~\ref{eqn::simple-bound})}. 
\end{align*}
\end{proof}

To obtain the bound in~\eqref{eqn::basic-amort-inequality}, it suffices to have
\begin{align*}
\nonumber
& \sum_{t=1}^T \frac {\Gamma}{8} \cdot \mathbb{E}_{k}\Big[\Big(\Delta x^{\pi(k,t),t}_{k} \Big)^2\Big] \left(1 - \frac{\alpha}{2n}\right)^{T-t} \\
\label{eqn::basic-amort-inequality-var}
&~~~~\ge  \frac{3 q \Lres^2}{n\Gamma}\sum_{t=1}^T  \left( 1 - \frac{\alpha}{2n}\right)^{T-t}\sum_{s \in [t - 2q, t + q] \setminus \{t\}}
\mathbb{E}_{k}  \big[ (\Delta x^{\pi(k,t),s}_{k_s})^2 \big]
\end{align*}
and in turn it suffices that
$\frac{9q^2\Lres^2}{n}/(1 - \frac{\alpha}{2n})^{2q} \le \frac{\Gamma^2}{8}$.
Since $\frac{\alpha}{2n} \le \frac{1}{2n}$ and $q \ll n$, it suffices that $\frac{9q^2\Lres^2}{n}/ \frac 12 \le \frac{\Gamma^2}{8}$,
or $q\le \frac{\sqrt n \G}{12 \Lres}$. 
The bound in Theorem~\ref{thm::informal} then follows readily
(with 
$\Lresbar$ replaced by $\Lres$, and setting $\G = \Lmax$).

\smallskip

\paragraph{Why $\Lresbar$ is needed in general}
In \eqref{eqn::basic-error-bound},
we are seeking to bound $\mathsf{Diff}=\sum_{k'}\|\nabla_{k'} f(\tilde{x}^{\pi(k', t), t}) - \nabla_{k'} f(x^{\pi(k', t), t}) \|^2$.
The SCV assumption ensures that $\tilde{x}^{\pi(k', t), t}$ and $x^{\pi(k', t), t}$ are independent of $k'$,
but this need not hold in the fully asynchronous setting.
As it happens, in the fully asynchronous setting,
we will be able to obtain bounds of the form $|\tilde{x}^{\pi(k', t), t} - x^{\pi(k', t), t}| \le \sum_s\Delta_s$, i.e., independent of $k'$,
where the sum is over $s$ with $ t-2q \leq s \leq t+q$ and
$s\ne t$ (the bounds $\Delta_s$ are larger than analogous terms
$|\Delta x^{\pi(k,t),s}_{k_s}|$ when the SCV assumption holds).  On using the Lipschitz parameters, this gives a bound of the form
$\mathsf{Diff} \le 3q\sum_{s,k'} L_{sk'}^2\Delta_s^2 \le 3q\sum_{s}\Lresbar^2 \Delta_s^2$,
which is weaker than the corresponding bound of $3q\sum_s \Lres^2 (\Delta x^{\pi(k,t),s}_{k_s})^2$ when the SCV assumption holds.

%% file: app-macros.tex
\newcommand{\Dxktau}{\Delta x_k^\tau}
\newcommand{\Dxktautau}{\Delta x_{k_\tau}^\tau}
\newcommand{\Dxktt}{\Delta x_{k_t}^t}

\newcommand{\xstar}{x^*}
\newcommand{\xt}{x^t}
\newcommand{\xto}{x^{t-1}}
\newcommand{\xktto}{x_{k_t}^{t-1}}

\newcommand{\Dmaxs}{\Delta_{\max}^s}
\newcommand{\Dmaxt}{\Delta_{\max}^t}
\newcommand{\Dmaxtp}{\Delta_{\max}^{t'}}
\newcommand{\Dmaxup}{\Delta_{\max}^{u'}}
\newcommand{\Dmaxbar}{\overline{\Delta}_{\max}}
\newcommand{\Dmaxu}{\Delta_{\max}^u}

\newcommand{\Dmins}{\Delta_{\min}^s}
\newcommand{\Dmint}{\Delta_{\min}^t}
\newcommand{\Dmintp}{\Delta_{\min}^{t'}}
\newcommand{\Dminbar}{\overline{\Delta}_{\min}}
\newcommand{\Dminu}{\Delta_{\min}^u}

\newcommand{\dminus}{d^{-}}
\newcommand{\dplus}{d^{+}}
\newcommand{\dpi}{d\hspace*{0.01in}'}
\newcommand{\hhd}{\breve{d}}

\newcommand{\xo}{x^\circ}

\newcommand{\node}{\mathcal{N}}
\newcommand{\ideal}{\Delta^{\text{I}}~}
\newcommand{\idealt}[1]{\Delta^{\text{I}}_{#1}~}
\newcommand{\differ}{\Delta^{\text{D}}~}
\newcommand{\FE}{\Delta^{\text{FE}}~}
\newcommand{\FEI}{\Delta^{\text{FEI}}~}
\newcommand{\FEIt}[1]{\Delta^{\text{FEI}}_{#1}~}
\newcommand{\FEs}{\Delta^{\text{FE}}}
\newcommand{\Lksk}{L_{k_s,k}}
\newcommand{\Lksksq}{(L_{k_s,k})^2}
\newcommand{\Dxks}{\Delta~x_{k_s}^{\pi,s}}
\newcommand{\Dxku}{\Delta~x_{k_u}^{\pi,u}}
\newcommand{\Dxkt}{\Delta~x_k^{\pi,t}}
\newcommand{\Dxktnopi}{\Delta~x_k^{t}}
\newcommand{\xkpit}{x_k^{\pi,t}}
\newcommand{\xkspis}{x_{k_s}^{\pi,s}}
\newcommand{\xktaupitau}{x_{k_\tau}^{\pi,\tau}}
\newcommand{\DRIS}[2]{\Delta^{\text{RI},#1}_{#2}~}
\newcommand{\DRES}[2]{\Delta^{\text{RE},#1}_{#2}~}
\newcommand{\emp}{\emptyset}
\renewcommand{\xs}{x_s}
\newcommand{\xu}{x_u}

\newcommand{\DI}[1]{\Delta^I_{#1}}
\newcommand{\DFEIt}{\Delta_t^{\text{FEI}}~}
\newcommand{\DRI}[1]{\Delta^{\text{RI},#1}~}
\newcommand{\DRIs}[1]{\Delta_s^{\text{RI},#1}~}
\newcommand{\DRIt}[1]{\Delta_t^{\text{RI},#1}~}
\newcommand{\DRE}[1]{\Delta^{\text{RE},#1}~}
\newcommand{\DREs}[1]{\Delta_s^{\text{RE},#1}~}
\newcommand{\DREt}[1]{\Delta_t^{\text{RE},#1}~}
\newcommand{\DRs}[1]{\Delta_s^{#1}~}
\newcommand{\DRt}[1]{\Delta_t^{#1}~}
\newcommand{\singlet}{\{t\}}
\newcommand{\singles}{\{s\}}
\newcommand{\Lkskt}{L_{k_sk_t}}
\newcommand{\Lksku}{L_{k_sk_u}}
\newcommand{\Lktks}{L_{k_tk_s}}
\newcommand{\Lkukt}{L_{k_uk_t}}
\newcommand{\Lkuks}{L_{k_uk_s}}
\newcommand{\Lkukv}{L_{k_uk_v}}
\newcommand{\Lkvks}{L_{k_vk_s}}
\newcommand{\Lkvkt}{L_{k_vk_t}}
\newcommand{\Lkvku}{L_{k_vk_u}}
\newcommand{\Lkvkw}{L_{k_vk_w}}
\newcommand{\pit}{\pi^t}
\newcommand{\xkspia}{x_{k_s}^{\pi,s}}
\newcommand{\xktpit}{x_{k_t}^{\pi,t}}
\newcommand{\xkupiu}{x_{k_u}^{\pi,u}}
\newcommand{\xkvpiv}{x_{k_v}^{\pi,v}}
\newcommand{\xkwpiw}{x_{k_w}^{\pi,w}}
\newcommand{\Rbar}{\overline{R}}
\newcommand{\Rsbar}{\overline{R\cup\{s\}}}
\newcommand{\Rubar}{\overline{R\cup\{u\}}}
\newcommand{\Lkuktsq}{\left(L_{k_u,k_t}\right)^2}

\newcommand{\Dmax}{\Delta_{\max}}
\newcommand{\Dmin}{\Delta_{\min}}
\newcommand{\Es}{E_s^x}
\newcommand{\Et}{E_t^x}
\newcommand{\Eu}{E_u^x}
\renewcommand{\FEs}{\Delta^{\text{FE}}_s}
\newcommand{\FEt}{\Delta^{\text{FE}}_t}
\newcommand{\gmaxktt}{g_{\max,k_t}^t}
\newcommand{\gminktt}{g_{\min,k_t}^t}
\newcommand{\exkss}{x_{k_s}^s}
\newcommand{\xktt}{x_{k_t}^t}

\newcommand{\calHlessthant}{\mathcal{H}_{<t}}

\newcommand{\barsett}{\overline{\{t\}}}
\newcommand{\barsetu}{\overline{\{u\}}}

\renewcommand{\Ap}{A^+}
\renewcommand{\Am}{A^-}

\newcommand{\Del}{\Delta}

\newcommand{\gam}{\gamma}
\newcommand{\gamp}{\gamma'}

\newcommand{\xv}{x_v}
\newcommand{\gmaxktpiuRt}{g_{\max,k_t}^{\pi,u,R,t}}
\newcommand{\gmaxkspiuRs}{g_{\max,k_s}^{\pi,u,R,s}}

\newcommand{\gmaxktpiRt}{g_{\max,k_t}^{\pi,R,t}}
\newcommand{\gmaxkspiRs}{g_{\max,k_s}^{\pi,R,s}}

\newcommand{\gmaxktpit}{g_{\max,k_t}^{\pi,t}}
\newcommand{\gmaxkspis}{g_{\max,k_s}^{\pi,s}}

\newcommand{\gDelmaxkspis}{g_{\overline{\Delta}_{\max},k_s}^{\pi,s}}
\newcommand{\gDelmaxktpit}{g_{\overline{\Delta}_{\max},k_t}^{\pi,t}}

\newcommand{\gmaxktpiut}{g_{\max,k_t}^{\pi,u,t}}
\newcommand{\gmaxkspius}{g_{\max,k_s}^{\pi,u,s}}

\newcommand{\gDelmaxkspius}{g_{\overline{\Delta}_{\max},k_s}^{\pi,u,s}}
\newcommand{\gDelmaxktpiut}{g_{\overline{\Delta}_{\max},k_t}^{\pi,u,t}}

\newcommand{\gmaxktpiphit}{g_{\max,k_t}^{\pi,\phi,t}}
\newcommand{\gmaxkspiphis}{g_{\max,k_s}^{\pi,\phi,s}}

\newcommand{\gmaxktpiupphit}{g_{\max,k_t}^{\pi,u',\phi,t}}

\newcommand{\gmaxktpiupt}{g_{\max,k_{t}}^{\pi,u',t}}
\newcommand{\gmaxkupiu}{g_{\max,k_u}^{\pi,u}}
\newcommand{\gmaxkupipiu}{g_{\max,k_u}^{\pi,\pi',u}}
\newcommand{\gmaxkupitRu}{g_{\max,k_u}^{\pi,t,R,u}}
\newcommand{\gmaxkupiRu}{g_{\max,k_u}^{\pi,R,u}}
\newcommand{\gmaxkupitphiu}{g_{\max,k_u}^{\pi,t,\phi,u}}
\newcommand{\gmaxkupitpphiu}{g_{\max,k_u}^{\pi,t',\phi,u}}
\newcommand{\gmaxkupiphiu}{g_{\max,k_u}^{\pi,\phi,u}}
\newcommand{\gminkupiu}{g_{\min,k_u}^{\pi,u}}
\newcommand{\gminkupipiu}{g_{\min,k_u}^{\pi,\pi',u}}
\newcommand{\gminkupitRu}{g_{\min,k_u}^{\pi,t,R,u}}
\newcommand{\gminkupiRu}{g_{\min,k_u}^{\pi,R,u}}
\newcommand{\gminkupitphiu}{g_{\min,k_u}^{\pi,t,\phi,u}}
\newcommand{\gminkupiphiu}{g_{\min,k_u}^{\pi,\phi,u}}

\newcommand{\gminktpit}{g_{\min,k_t}^{\pi,t}}
\newcommand{\gminkspis}{g_{\min,k_s}^{\pi,s}}

\newcommand{\gDelminkspis}{g_{\overline{\Delta}_{\min},k_s}^{\pi,s}}
\newcommand{\gDelminktpit}{g_{\overline{\Delta}_{\min},k_t}^{\pi,t}}

\newcommand{\gminktpiut}{g_{\min,k_t}^{\pi,u,t}}
\newcommand{\gminkspius}{g_{\min,k_s}^{\pi,u,s}}

\newcommand{\gDelminktpiut}{g_{\overline{\Delta}_{\min},k_t}^{\pi,u,t}}
\newcommand{\gminktpiuRt}{g_{\min,k_t}^{\pi,u,R,t}}
\newcommand{\gminktpiRt}{g_{\min,k_t}^{\pi,R,t}}
\newcommand{\gminktpiuphit}{g_{\min,k_t}^{\pi,u,\phi,t}}
\newcommand{\gminktpiphit}{g_{\min,k_t}^{\pi,\phi,t}}
\newcommand{\gkspis}{ {g}_{k_s}^{\pi,s}}
\newcommand{\gkspipis}{ {g}_{k_s}^{\pi,\pi',s}}
\newcommand{\tgkspis}{ \tilde{g}_{k_s}^{\pi,s}}
\newcommand{\tgkspipis}{ \tilde{g}_{k_s}^{\pi,\pi',s}}
\newcommand{\gkupiu}{ {g}_{k_u}^{\pi,u}}
\newcommand{\gkupipiu}{ {g}_{k_u}^{\pi,\pi',u}}
\newcommand{\tgkupiu}{ \tilde{g}_{k_u}^{\pi,u}}
\newcommand{\tgkupipiu}{ \tilde{g}_{k_u}^{\pi,\pi',u}}
\newcommand{\pip}{\pi'}
\newcommand{\pku}{p_{k_u}}
\newcommand{\pkv}{p_{k_v}}
\newcommand{\sett}{\{t\}}
\newcommand{\ubar}{\overline{\{u\}}}
\newcommand{\tbar}{\overline{\{t\}}}
\newcommand{\xkss}{x_{k_s}^s}
\newcommand{\xkuu}{x_{k_u}^u}
\newcommand{\xkupu}{x_{k'_u}^u}
\newcommand{\xktpt}{x_{k'_t}^t}
\newcommand{\xkspipis}{x_{k_s}^{\pi,\pi',s}}
\newcommand{\xkspiRs}{x_{k_s}^{\pi,s, R}}
\newcommand{\xkspists}{x_{k_s}^{\pi,s,\{t\}}}
\newcommand{\xktpistt}{x_{k_t}^{\pi,t,\{t\}}}
\newcommand{\xkspiess}{x_{k_s}^{\pi,s,\emptyset}}
\newcommand{\xktpiest}{x_{k_t}^{\pi,t,\emptyset}}
\newcommand{\xkspius}{x_{k_s}^{\pi,\{u\},s}}
\newcommand{\xkspiphis}{x_{k_s}^{\pi,\phi,s}}
\newcommand{\Ru}{R\cup \{u\}}
\newcommand{\Tsu}{T_{su}}
\newcommand{\Tst}{T_{st}}
\newcommand{\xkspiRus}{x_{k_s}^{\pi,\Ru,s}}
\newcommand{\xkupiRu}{x_{k_u}^{\pi,R,u}}
\newcommand{\xkupipiu}{x_{k_u}^{\pi,\pi',u}}
\newcommand{\xktpiRt}{x_{k_t}^{\pi,t,R}}
\newcommand{\xktpipit}{x_{k_t}^{\pi,\pi',t}}
\newcommand{\xkpupiRpu}{x_{k'_u}^{\pi,R',u}}
\newcommand{\xkupipphiu}{x_{k_u}^{\pi',\phi,u}}
\newcommand{\xkpupiRu}{x_{k'_u}^{\pi,R,u}}
\newcommand{\xkupiRpu}{x_{k_u}^{\pi,R',u}}
\newcommand{\xkpupiphiu}{x_{k'_u}^{\pi,\phi,u}}
\newcommand{\xkptpipphit}{x_{k'_t}^{\pi',\phi,t}}
\newcommand{\xkpupipRu}{x_{k'_u}^{\pi',R,u}}
\newcommand{\xkupiphiu}{x_{k_u}^{\pi,\phi,u}}
\newcommand{\xktpiphit}{x_{k_t}^{\pi,\phi,t}}
\newcommand{\xkvpiphiv}{x_{k_v}^{\pi,\phi,v}}

\newcommand{\FEst}{\Delta^{\text{FE},t}_s~}
\newcommand{\FEu}{\Delta^{\text{FE}}_u}
\newcommand{\FEut}{\Delta^{\text{FE},t}_u~}
\newcommand{\FEuRt}{\Delta^{\text{FE},R,t}_u~}
\newcommand{\FEsRut}{\Delta^{\text{FE},\Ru,s}_u~}
\newcommand{\FEtpi}{\Delta^{\text{FE}}_{t+i}}
\newcommand{\FEtpip}{\Delta^{\text{FE}}_{t+i+1}}
\newcommand{\tone}{t-1}

%% file: general-framework.tex
\section{The Framework for the General Analysis}
\label{sec::gen-framework}

\newcommand{\Err}{\operatorname{Err}}
\newcommand{\prev}{\operatorname{prev}}
\newcommand{\gktpit}{ {g}_{k_t}^{\pi,t}}
\newcommand{\gkpikt}{g_k^{\pi(k),t}}
\newcommand{\gkpiktt}{g_k^{\pi(k_t),t}}
\newcommand{\gkspiktt}{g_{\ks}^{\pi(\kt),t}}
\newcommand{\tgktpit}{\tilde{g}_{k_t}^{\pi,t}}

\newcommand{\xkspikts}{x_{\ks}^{\pi(\kt),s}}

At a high level, the new framework has the same general structure as the basic
framework described in Section~\ref{sect:simple}.
It consists of three parts.
In the first part, we obtain the following variant of \eqref{eqn::basic-progress-with-error}
without using the SCV assumption.
\[
\expect{F(x^{t+1})} \le \left(1 - \frac {\alpha}{3n}\right) \cdot \expect{F(x^{t})} +
\mathbb{E}\Big[\frac \G 8\left(\Delta x_{\kt}^{\pi,t}\right)^2 -\Err_t\Big].
\]
$\Err_t$ will be specified in Lemma~\ref{lem::F-prog-general} below.

The second part, which is the heart of the analysis, bounds $\expect{\Err_t}$ in terms
of $\mathbb{E}\left(\Delta x_{\ks}^{\pi,s}\right)^2$, for a suitable range of $s$ values,
and other terms $(\calD_s)^2$, which we will define later, and which 
are themselves bounded in terms of $\expect{\left(\Delta x_{k_u}^{\pi,u}\right)^2}$ and
$(\calD_u)^2$ for a suitable range of $u$ values.

The third part deduces the bounds in Theorem~\ref{thm:main-SACD}, by means of a suitable potential function (a.k.a.\ a Lyapunov function) and an amortized analysis.
\subsection{Part 1: Demonstrating Substantial Progress}
Recall that $\pi(k,t)$ denotes the path in which coordinate $\kt$ at time $t$ is replaced by coordinate $k$;
to reduce clutter we now abbreviate this as $\pi(k)$. 
Note that $\pi(\kt) = \pi$.
We let
$\prev(t,k)$ denote the time of the most recent update to coordinate $k$, if any,
in the time range $[t-2q,t-1]$; otherwise, we set it to $t$.
\begin{lemma}
\label{lem::F-prog-general}
\begin{align*}
& \expect{F(x^{t}) -F(x^{t+1})}\\
&~~~~~~~~\ge \frac 1{3n^2} \mathbb{E}\Big[ \sum_{k=1}^n \sum_{k_t=1}^n
\hWk(\gkpiktt,x^{\pi(k_t), t}_k) \Big] +
\mathbb{E}\Big[\frac \G 8\left(\Delta x_{\kt}^{\pi,t}\right)^2 -\Err_t\Big],
\end{align*}
\begin{align*}
&\text{where}~~\Err_t  =  \frac{1}{3n^2} 
\bigg[ \sum_{\substack{t-2q \le s < t\\ 
 \& s = \prev(t,\ks)}}\sum_{k_t=1}^n   
 \bigg(\frac 3{2\G} \underbrace{\left(\tilde{g}_{k_s}^{\pi(k_t), s} - \gkspiktt
 \right)^2}_{A} 
 \\
& \hspace*{1.5in}  +  
2\G \underbrace{\left(x^{\pi(k_s), t}_{k_s} - x^{\pi(k_t), t}_{k_s} \right)^2}_{B} 
+   
\frac{3\G}{2} \underbrace{\left(\Delta \xkspikts \right)^2}_{C} 
\bigg)\bigg]\\
&\hspace*{0.6in}+ 
\frac 1{n^2}\sum_{k=1}^n \sum_{k_t=1}^n\frac 2{3\G} \underbrace{\left(\gkpikt - \gkpiktt\right)^2}_{D} 
+ \underbrace{\frac 1\G 
\left(\gktpit -  \tgktpit \right)^2}_{E}.
\end{align*}
\end{lemma}

\begin{align*}
&\text{By \eqref{eqn::progress-basic-formula},}~~
\sum_{k=1}^n \sum_{k_t=1}^n
\hWk(\gkpiktt,x^{\pi(k_t), t}_k)  \ge \sum_{k_t=1}^n \alpha F(x^{\pi(\kt),t}),
\\
%
& \text{which gives}~~
\expect{F(x^{t+1})} \le \left(1 - \frac {\alpha}{3n}\right)\expect{F(x^{t})} +
\mathbb{E}\Big[\frac \G 8\left(\Delta x_{\kt}^{\pi,t}\right)^2 -\Err_t\Big].
\hspace*{0.5in}
\end{align*}

\hide{
Our remaining task, as in \eqref{eqn::basic-amort-inequality},
is to show that over the course of $T$ updates, the terms $\frac \G 8\mathbb{E}\Big[\left(x_{\kt}^{\pi,t}\right)^2\Big]$ outweigh the terms $\Err_t$. Now we will do this for both strongly convex and plain convex functions.
This is considerably more challenging, due to the involved form for $\Err_t$, and to the issues raised in Comments~\ref{com::x-tilde-depends-on-coord-choice} and~\ref{com::x-depends-on-coord-choice} in Section~\ref{sec:comments-on-full-results}.
}

To prove Lemma~\ref{lem::F-prog-general}, we start from~\eqref{eqn:F-prog},
and then apply the following two lemmas regarding shifting the parameters in $\hW$. (See Appendix~\ref{sect:appendix-general} for proofs.)

\begin{lemma}[$\hW$ Shifting on the $g$ parameter]\label{lem:W-shift-re}
For any $\gej$, $\gejp$,
\[
\hWj(\gej,x_j) ~\ge~ \frac 23 \cdot \hWj(\gejp,x_j) ~-~ \frac{4}{3\G} \cdot (\gej - \gejp)^2.
\]
\end{lemma}

\hide{
\RJC{The next paragraph does not seem useful any more.}\YKC{Agree.}
Before we state the next lemma, we need to introduce some further notation.
Recall that at the beginning of the $t$-th update, $x_k^{\pi,t-2q}$ is already fixed when a core chooses $k_t$.
However, if there are some updates to coordinate $k$ over the time interval $[t-2q,t-1]$, the value of $x_k$ will be modified during this time interval.
A subtle observation is that how $x_k$ is modified might depend on the choice of $k_t$
(and also of $k_{t+1},k_{t+2},\cdots,k_{t+q - 1}$).
}

\begin{lemma}[$\hW$ Shifting on the $x$ parameter]
\label{lem::What-on-two-paths}
Suppose there are $\ell$ updates to coordinate $k$
over the time interval $[t - 2q, t - 1]$. Then
\begin{align*}
&\text{if $\ell = 0$, }~~~~~\hW(g^{\pi,t}_k, x_{k}^{\pi(k),t}) = \hW(g^{\pi,t}_k,x_{k}^{\pi,t}) \\
&\text{if $\ell > 0$, }~~~~~
 \hW(g_k^{\pi, t}, x_k^{\pi(k),t}) \ge \hW(g_k^{\pi, t},x_k^{\pi,t})
- \frac 3{2\G} \cdot (\tilde{g}_k^{\pi, {\prev}(t, k)} - g_k^{\pi, t })^2 \\
&\hspace*{1.8in}- 2 \G (x_k^{\pi,t} - x_k^{\pi(k),t})^2
- \frac{3\G}{2} \cdot (\Delta x_k^{\pi, {\prev}(t, k)})^2.
\end{align*}
\end{lemma}

%% file: SACD-notation.tex
\subsection{Part 2: Bounding $\Err_t$, the Error Term}

We begin by stating the following lemma, which bounds the difference in the increments computed by two updates to a coordinate $x_j$
when the inputs to Step 5 vary.
To avoid notational clutter, we write $\Psi$ and $\hd$ in lieu of $\Psi_j$
and $\hd_j$
(to review its definition see the update rule in Section~\ref{sect:model});  also, by $x_1$ and $x_2$ we mean two
possible values of $x_j$, and by $g_1$ and $g_2$ two
possible values of $g_j$.

\begin{lemma}
\label{lem:change-of-Dp-vs-change-of-g-gen}
For any $g_1,g_2, x_1, x_2 \in\rr$ and $\G\in\rrplus$,
$| \hd(g_1,x_1) - \hd(g_2,x_2) | \leq |x_1 - x_2| + \frac{1}{\G}\cdot \left|g_1 - g_2\right|$, and hence
\[
\Big( \hd(g_1,x_1) - \hd(g_2,x_2) \Big)^2 ~\le~ 2 (x_1 - x_2)^2 ~+~ \frac{2}{\G^2}\cdot \left(g_1 - g_2\right)^2.
\]
If $\Psi$ is the zero function, then $| \hd(g_1,x_1) - \hd(g_2,x_2) |= \frac{1}{\G}\cdot \left|g_1 - g_2\right|$.
\end{lemma}

\subsubsection{Additional Notation}
\label{sec::addnl-notation}

In this subsection, we will be defining notation of the form $\Delta^{\bullet}_{\max} x^{\pi,s}_{k_s}$, where $\bullet$ refers to various parameters we will specify as needed, and the $\max$ refers to taking a suitable maximum.
Without spelling it out, we will assume the analogous notation with
$\Delta^{\bullet}_{\min}$ is also being defined.
In addition, we will define
$\Delta_{\mathsf{span}}^{\bullet} x^{\pi,s}_{k_s} \triangleq \Delta^{\bullet}_{\max} x^{\pi,s}_{k_s} - \Delta^{\bullet}_{\min} x^{\pi,s}_{k_s}$,
and $\Delta_{\mathsf{var}}^{\bullet} x^{\pi,s}_{k_s} \triangleq
\max\{|\Delta^{\bullet}_{\min} x^{\pi,s}_{k_s}|, |\Delta^{\bullet}_{\max} x^{\pi,s}_{k_s}|,$ $\Delta_{\mathsf{span}}^{\bullet} x^{\pi,s}_{k_s}\}$.

The next step in our analysis is to generalize Lemma~\ref{lem::super::simple}
to settings in which the SCV Assumption needs not hold, so as to bound the
``error'' terms in $\Err_t$.
We seek to carry out an analysis analogous to \eqref{eqn::basic-error-bound}.
The first difficulty we face is that the bound we obtain is going to depend
on the span of possible values of $\Delta x^{\pi,s}_{k_s}$, which we denote
by $\Delta_{\mathsf{span}} x^{\pi,s}_{k_s}$,
where $\Delta_{\max} x^{\pi,s}_{k_s}$ is the maximum possible value
for this increment over all asynchronous schedules on path $\pi$,
assuming the first $t-2q - 1$ updates are already fixed.
Thus, in addition to bounds on the various gradient differences,
we will need to bound $\Delta_{\mathsf{span}} x^{\pi,s}_{k_s}$.
We begin with this task.

Notice that we have assumed that the first $t-2q - 1$ updates are known rather than the first $s-2q - 1$. To reflect this, we denote the
maximum possible value of the update by $\Delta_{\max}^t x^{\pi,s}_{k_s}$,
and analogously, we write $\Delta_{\mathsf{span}}^t x^{\pi,s}_{k_s}$.
We are interested in $(s,t)$ pairs with $t-2q \leq s \leq t+q$,
or equivalently, $s-q \leq  t \leq s+2q$;
these are the updates $\calU_s$ whose value may not be determined at the start of update $\calU_t$ and which may affect update $\calU_t$.
We call $t$ the \emph{reference time} for update $\calU_s$.

For notational convenience, rather than give a bound on $\Delta_{\mathsf{span}}^t x^{\pi,s}_{k_s}$, we will bound
$\Delta_{\mathsf{span}}^u x^{\pi,t}_{k_t}$ instead.
So, suppose that the first $u-2q -1$ updates have been fixed, for some
$u$ with $t-q \leq u \leq t+2q$.
Let $x_{\max,k_t}^{\pi,u,t}$  and $x_{\min,k_t}^{\pi,u,t}$, resp.,
be the largest and smallest values that $x_{\kt}^{\pi,t}$
could attain with the first $u-2q - 1$ updates already fixed;
similarly, let $\tilde{g}_{\max,k_t}^{\pi,u,t}$ and $\tilde{g}_{\min,k_t}^{\pi,u,t}$ be the largest and smallest gradient values
that could be computed by update $\calU_t$ with the first $u-2q - 1$ updates already fixed.

Lemma~\ref{lem:change-of-Dp-vs-change-of-g-gen} implies
\begin{align*}
\left(\Delta_{\mathsf{span}}^u x^{\pi,t}_{k_t}\right)^2
\le
  2\big(x_{\max,k_t}^{\pi,u,t} - x_{\min,k_t}^{\pi,u,t} \big)^2
 +
 \frac {2} {\G^2} \big(
 \tilde{g}_{\max,k_t}^{\pi,u,t} - \tilde{g}_{\min,k_t}^{\pi,u,t}\big)^2.
\end{align*}
We use a Lipschitz bound to obtain

\begin{align*}
\nonumber
\big(\tilde{g}_{\max,k_t}^{\pi,u,t} - \tilde{g}_{\min,k_t}^{\pi,u,t}\big)^2
\le& \Big[\sum_{\substack{t-2q\leq s \leq t+q\\ \text{and } s\ne t}} \Lkskt \max\Big\{\big|\Delta_{\max}^u x^{\pi,s}_{k_s}\big|,\big|\Delta_{\min}^u x^{\pi,s}_{k_s}\big|,\\[-0.2in]
&\hspace*{2.2in}  \Delta_{\mathsf{span}}^u x^{\pi,s}_{k_s}\Big\}\Big]^2.
\end{align*}
The reason for the three terms is that in determining each gradient,
for each $s$ in the given range, the relevant update could be read or not
read; so the difference due to this coordinate could stem from its maximum value, its minimum value, or their difference.

By the Cauchy-Schwartz inequality,
\begin{align*}
\left(\Delta_{\mathsf{span}}^u x^{\pi,t}_{k_t}\right)^2
&\le
 2 \big(x_{\max,k_t}^{\pi,u,t} - x_{\min,k_t}^{\pi,u,t}\big)^2\\
 & ~~~~+ \frac{6q}{\G^2} \sum_{\substack{t-2q\leq s\leq t+q\\ \text{and } s\ne t}} \Lkskt^2 \max \Big\{ \big|\Delta_{\max}^u x^{\pi,s}_{k_s}\big|^2,
                 \big|\Delta_{\min}^u x^{\pi,s}_{k_s}\big|^2,\\[-0.2in]
& \hspace*{2.7in}\big(\Delta_{\mathsf{span}}^u x^{\pi,s}_{k_s}\big)^2\Big\}.
\numberthis\label{eqn::span-bound-without-exclusions}
\end{align*}

\paragraph{Legitimate Averaging via Exclusion}
Recall that in Section~\ref{sect:simple}, a crucial step for obtaining a good parallelism bound was to perform averaging over the $n$ paths in each group, i.e.,
to replace the terms $\Lkskt^2$ by $\Lresbar^2$ by averaging over $\kt$.
To do this here we would need $\Delta_{\max}^u x^{\pi,s}_{k_s}$ and $\Delta_{\min}^u x^{\pi,s}_{k_s}$ to have the same value on every path $\pi(k)$.
But this need not be the case, because the computation of $\Delta_{\max}^u x^{\pi,s}_{k_s}$ could read the result of update $\calU_t$,
and therefore depend on the choice of $\kt$.
To address this, we will create terms which upper bound $\Delta_{\max}^u x^{\pi,s}_{k_s}$ and
which have the same value on every path $\pi(k)$, and
similarly for $\Delta_{\min}^u x^{\pi,s}_{k_s}$.

Our first key observation, concerns $\calU_s$ and $\calU_t$: one of them commits first.
Suppose $\calU_s$ commits first; then the value computed by $\calU_s$ does not use the value computed by $\calU_t$,
either directly as an input, or indirectly because its inputs do not use this value either.
Otherwise, $\calU_s$ has no impact on $\Delta^u_{\mathsf{span}} x_{\kt}^{\pi,t}$.

We introduce new terminology to capture this observation.
If update $\calU_s$ commits before $\calU_t$, we will say that $\calU_t$ is \emph{excluded} from the computation of $\calU_s$.
To  capture the exclusion of $\calU_t$, we define
$\Delta_{\max}^{u,\{t\}} x^{\pi,s}_{k_s}$ to be the maximum value  $\calU_s$ can compute on path $\pi$, assuming that $\calU_s$ commits before $\calU_t$, and the first $u-2q-1$ updates are fixed.

The updates that cause the output of $\calU_s$ to vary are those that might commit before or after $\calU_s$;
these are always a subset of the $\calU_v$ with $v\in [s-2q,s+q]$.
We also observe that if update $\calU_s$ commits before $\calU_t$, then $\calU_s$ commits before any update $\calU_v$ with $v > t+q$.
We can safely incorporate this constraint in the notation $\Delta_{\max}^u$ and $\Delta_{\max}^{u,\{t\}}$.
The notation extends to $\Delta_{\mathsf{span}}$ and $\Delta_{\mathsf{var}}$ in the natural way.

Note that $\pi(k)$ and $\pi(k')$ are identical paths apart from the coordinate chosen at time $t$.
The phrase ``$\calU_t$ is excluded from the computation of $\calU_s$'' could cause us to conjecture that 
$\Delta_{\max}^{u, \{t\}} x^{\pi(k),s}_{k_s}$ are identical for all $k$, and similarly for the $\Delta_{\min}^{u, \{t\}} x^{\pi(k),s}_{k_s}$ .
If it were so, we could rewrite \eqref{eqn::span-bound-without-exclusions} as follows, and average over $\kt$:

\begin{align*}
\left(\Delta_{\mathsf{span}}^u x^{\pi,t}_{k_t}\right)^2  \le
 \ldots + \frac{6q}{\G^2}
\sum_{\substack{t-2q \leq s \leq t+q\\ \text{and } s\ne t}} \Lresbar^2 \cdot \max\Big\{\ldots, \big(\Delta_{\mathsf{span}}^{u, \{t\}} x^{\pi,s}_{k_s}\big)^2\Big\}.
\end{align*}

However, this conjecture needs not be  true with the current definition, so the above averaging is not yet valid.
For, as explained in the next paragraph, a problem can arise if there is a coordinate $k_v = k_t$ with $t<v \leq t+q$;
when averaging over all $k_t$ we are certain to encounter paths with this property.
(As an aside, we note that the conjecture is true when $\Psi\equiv 0$ and the ST order is used.)

Suppose $k_v = k_t$, $v\ne t$, and suppose $\calU_t$ is excluded.
Given the SCC order, it would appear $\calU_v$ should also be excluded.
But then suppose there is some other update $\calU_s$ with $t-2q\leq s \leq t+q$, $k_s\ne k_t$,
and on some path $\pi(k)$ where $k\ne \kt$, in computing
$\Delta_{\max}^{u,\{t\}} x^{\pi,s}_{k_s}$,
$\calU_s$ reads the result of update $\calU_v$.
Then to be sure the same maximum value were computed by
$\calU_s$ on path $\pi$, we would need it to read this excluded value.
So this value cannot be excluded.
Instead, we define the computation to act as if it were in the SCC order,
but with update $\calU_t$ simply not present, i.e., as if there were a total of
$T-1$ updates over the whole computation. (This only pertains
to updates $\calU_s$ with $t-2q \leq s\leq t+q$ and $s\ne t$.)

This looks promising, as we would anticipate that
$\Delta_{\max}^{u,\{t\}} x^{\pi,s}_{k_s} \le \Delta_{\max}^u x^{\pi,s}_{k_s}$,
and so it would seem the last term on the RHS can be bounded recursively.
But unfortunately, this property need not hold if there are updates to the same coordinate
on $\pi$ in the range $[t+1,\min\{s,u\}+q]$.
To understand the issue
we revisit Example~\ref{ex::SCC-need}. Suppose updates $\calU_t$ and $\calU_{t+1}$ are both to coordinate $x_1$, with $x_1^t = -1$ as before.
If both updates are present, and both read value $x_1=-1$ for their gradient computation, then $\calU_t$ computes an increment of $1$ and $\calU_{t+1}$ an increment of $\frac 12$. However, if  $\calU_t$ is excluded,
then $\calU_{t+1}$ computes an increment of $1$; i.e.,
$\Delta_{\max}^{u,\{t\}} x^{\pi,t+1}_{k_{t+1}}>
\Delta_{\max}^{u} x^{\pi,t+1}_{k_{t+1}}$, contrary to the desired property.

To avoid the difficulty illustrated by the above example, we will need to modify the definition of $\Delta_{\max}^{u, \{t\}} x^{\pi,s}_{k_s}$.
These modifications enable Lemma~\ref{Dmax-defn-works} below,
which ensures $\Delta_{\max}^{u, \{t\}} x^{\pi,s}_{k_s}$ has the properties
we need to carry out the averaging in the analysis.
%
Recall that depending on the choice of asynchronous schedule, $\calU_s$
may or may not read values computed by updates in the range $[s-2q,\min\{s,u\}+q]\setminus\{t\}$. We will want to pretend that $\calU_s$ can make this
choice for an expanded range of updates, namely a subset
of $[s-4q,\min\{s,u\}+q]\setminus\{t\}$.
In effect, this enlarges the set of possible asynchronous schedules.
To be very precise: let $l$ be the maximum commit time for updates $\calU_1,\calU_2,\ldots,\calU_{u-4q-1}$; $\calU_s$ may read or not read any values
computed by updates $\calU_r$ that commit after time $l$ for $r\in [s-4q,\min\{s,u\}+q]\setminus\{t\}$.
We call this $\calU_r$'s \emph{extended computation}.

So as to identify the updates $\calU_r$ with $r\ge u-4q$ that commit by time $l$,
we introduce the following set $A^{\pi,u}$ of updates: $A^{\pi,u}= \{ r \,|\, u-4q \leq r <  u-2q$  and  $\calU_r$ has committed before some $\calU_{p} \text{ for } p < u-4q\}$,
which means the updates in $A^{\pi,u}$ have committed before $\calU_u$ starts its extended computation.
Also, rather than just excluding $\calU_t$,
we allow any additional subset $S= \{v \,|\,   u-4q \leq v \leq u+q\}\setminus A^{\pi,u}$ of updates to be excluded; we follow this by maximizing over all such $S$.
This also implies that the range of $s$ in which we are interested becomes
$u-4q \leq s \leq u+q$.
Note that if $s \in A^{\pi,u}$ then $\Delta_{\mathsf{span}}^{u, \{t\}} x^{\pi,s}_{k_s} =0$.


Later on, we will be allowing the exclusion of sets $R$ other than just $\{t\}$, so we also incorporate this in our definitions.
For $R,S$ disjoint from $A^{\pi,u}$, we define:

\smallskip

\begin{tabular}{lll}
\multirow{5}{*}{$\Delta_{\max}^{u,R,S} x_{k_s}^{\pi, s}$}
&\multirow{5}{*}{~$\triangleq$~}
& the maximum value that $\Delta \xkspis$ can assume when the \\
&&  first $(u-4q - 1)$ updates and all updates in $A^{\pi,u}$ on\\
&&  path $\pi$ have been fixed, and update $\calU_v$ is excluded\\
&&  from the computation of $\calU_s$  for $v \in R\cup S$  and for \\
&&  $v > u+ q$;\\[0.1in] 
$\Delta_{\max}^{u,R} x_{k_s}^{\pi, s}$&~$\triangleq$~&
$\max_{S\subseteq	 [u-4q,u+q]\setminus A^{\pi,u}\cup\{s\}} \Delta_{\max}^{u,R,S} x_{k_s}^{\pi, s}$.\\[0.1in]
$\Delta_{\max}^u x_{k_s}^{\pi, s}$&~$\triangleq$~& $\Delta_{\max}^{u,\emptyset} x_{k_s}^{\pi, s}$\\[0.1in]
\end{tabular}


Now we have the desired properties:

\begin{lemma}
\label{Dmax-defn-works}
i. If $t \in R$ and $u \leq t + 2q$, then
$\Delta_{\max}^{u,R} x_{k_s}^{\pi(k), s}$ is identical on every path $\pi(k)$.\\[0.05in]
ii. If $s\le t$ then $\Delta_{\max}^{t, R} x_{k_s}^{\pi,s} \le \Delta_{\max}^{s, R} x_{k_s}^{\pi,s}$.\\[0.05in]
iii. If $R\subset R'$ then $\Delta_{\max}^{u,R'} x_{k_s}^{\pi, s} \le \Delta_{\max}^{u,R} x_{k_s}^{\pi, s}$.\\[0.05in]
iv. $\Delta_{\max}^{u, R} x_{k_s}^{\pi,s} ~\le~  \Delta_{\max}^{u, \emptyset} x_{k_s}^{\pi,s}$.
\end{lemma}

\begin{proof}
\hide{
\RJC{Should we move this proof to the appendix?}
\YKC{I think we should keep the proof of part i here, as part i is a main target
and a main reason for introducing such non-trivial definition, while the proof can give the audience some idea why such definition is needed.
For other parts I agree we can move their proofs to the appendix.}
\RJC{Removing parts ii-iv seems to have no effect on the page length of the main part, so I am a bit disinclined to move these parts.}
}
i. We begin by showing that the set $A^{\pi(k),u}$ is the same for all $k$.
Recall that $u \le t+2q$.
By Lemma~\ref{lem::SCCrange}, the start time for $\calU_t$ is at least $t-q+1 \ge u -3q+1$.
For $p < u-4q$, again by Lemma~\ref{lem::SCCrange},
$\calU_p$ has commit time at most $u-4q-1+q+1= u-3q$.
So if $\calU_r$ commits before $\calU_p$ for some $p\le u -4q$,
$\calU_r$ has commit time at most $u-3q-1$.
Thus $\calU_r$'s commit time is unaffected by the choice of coordinate by $\calU_t$,
and therefore $A^{\pi(k),u}$ is the same for all $k$.

Now, recall that $\pi(k)$ is the path in which coordinate $\kt$ at time $t$ on path $\pi$ is replaced by coordinate $k$.
By definition, if $t \in R$, then $\calU_t$ is excluded
from the computation of $\Delta_{\max}^{u,R} x_{k_s}^{\pi, s}$. As the paths $\pi(k)$ are identical apart from their $t$-th coordinate, and as the
sets $A^{\pi(k),u}$ are the same for all these paths, it follows that $\Delta_{\max}^{u,R} x_{k_s}^{\pi(k), s}$ is exactly the same for every $k$.

ii. This follows from the next two observations:
(a) the update does not read any of the variable values computed by updates $\calU_v$ for $v > s+q$ and therefore reducing the top end of the range in going from reference time $t$ to $s$ does not change the possible updates
computed by $\calU_s$; (b) $A^{\pi,s}\cap [t-4q,t-2q) \subseteq A^{\pi,t}$,
and therefore there are at least as many updates that are not yet fixed
with reference time $s$ compared to reference time $t$;
furthermore, the fixed values in $A^{\pi,t}\setminus A^{\pi,s}$
are all values that could be computed in the computation with
reference time $s$.

iii. Recall the definition of $\Delta_{\max}^{u,R} x_{k_s}^{\pi, s}$, and let $S'$ be the set for which $\Delta_{\max}^{u,R'} x_{k_s}^{\pi, s} = \Delta_{\max}^{u,R',S'} x_{k_s}^{\pi, s}$. Now, we let $S= (R'\cup S') \setminus R$; thus $S\cup R = S'\cup R'$.
Clearly, $\Delta_{\max}^{u,R',S'} x_{k_s}^{\pi, s}
=\Delta_{\max}^{u,R,S} x_{k_s}^{\pi, s} \le \Delta_{\max}^{u,R} x_{k_s}^{\pi, s}$.

iv. This follows immediately from iii.
\end{proof}

Recall that we defined $\Dspan{\pi, s}{k_s}{u, R} = \Delta_{\max}^{u, R} x_{k_s}^{\pi, s} - \Delta_{\min}^{u, R} x_{k_s}^{\pi, s}$ and $\Dvar{\pi, s}{k_s}{u, R} = \max\big\{\big| \Delta_{\max}^{u, R} x_{k_s}^{\pi, s} \big|, \big|\Delta_{\min}^{u, R} x_{k_s}^{\pi, s} \big| , \Dspan{\pi, s}{k_s}{u, R}  \big\}$.
We want to have one term to cover every update in which
$x_{k_s}^{\pi,s}$ is involved.
Accordingly, we define 
\[
\Dmaxbar^R \xkspis := \max_{\substack{u: u-4q \leq s\\ \hspace{0.3in}\le u+q}} \Delta_{\max}^{u, R} \xkspis
= \max_{\substack{u:s-q\le u\\ \hspace{0.2in} \leq s+4q}} \Delta_{\max}^{u, R} \xkspis
=\max_{s-q\le u \le s} \Delta_{\max}^{u, R} \xkspis,
\]
where we use Lemma~\ref{Dmax-defn-works}(ii) for the final equality.

We let
$\Dspan{\pi, s}{k_s}{u, R} = \Delta^{u, R}_{\max} x_{k_s}^{\pi, s} - \Delta^{u, R}_{\min} x_{k_s}^{\pi, s}$,
$\Dbarspanalt{R} x^{\pi, s}_{k_s} = \overline{\Delta}_{\max}^{R} x^{\pi, s}_{k_s} - \overline{\Delta}_{\min}^{R} x^{\pi, s}_{k_s}$ and
$\Dbarvaralt{R} x^{\pi, s}_{k_s} = \max\big\{\big|\overline{\Delta}_{\max}^{R} x^{\pi, s}_{k_s}\big|,$ $\big|\overline{\Delta}_{\min}^{R} x^{\pi, s}_{k_s}\big|,\Dbarspanalt{R} x^{\pi, s}_{k_s}\big\}$.
We simplify the notation when $R = \emptyset$, defining
$\Dmaxbar \xkspis \triangleq\Dmaxbar^{\emptyset} \xkspis$,
$\Dbarspan{\pi, s}{k_s}  \triangleq \Dbarspanalt{\emptyset} x^{\pi, s}_{k_s} $
and $\Dbarvar{\pi, s}{k_s} \triangleq \Dbarvaralt{\emptyset} x^{\pi, s}_{k_s}$.
Clearly, if $u  \subseteq [s - q, s]$, then
\begin{align*} 
\Dspan{\pi, s}{k_s}{u, R} \le \Dbarspanalt{R} x^{\pi, s}_{k_s}  \leq \Dbarspan{\pi, s}{k_s}
~~~\text{and}~~~\Dvar{\pi, s}{k_s}{u, R} \le \Dbarvaralt{R} x^{\pi, s}_{k_s} \leq \Dbarvar{\pi, s}{k_s}. \numberthis\label{eqn::dropRinDelta-var}
\end{align*}

Finally, we will want to know the expected effect of update $\calU_t$.
Thus, we define\\
\[\left(\Df_t \right)^2 ~\triangleq~ \expect{\left( \Dmaxbar \xktpit - \Dminbar \xktpit \right)^2}
~~~\text{and}~~~\left( \DE_t \right)^2 ~\triangleq~ \expect{\left( \Delta  \xktpit \right)^2}.
\]

Since exactly $T$ updates are made, we assume that $\left(\Df_t \right)^2,\left( \DE_t \right)^2 \equiv 0$
for  $t= 0$ and $t\ge T+1$ throughout the analysis.

\smallskip

Next, we introduce analogous notation for the gradients.

\smallskip

\begin{tabular}{lcl}
\multirow{5}{*}{$\tilde{g}_{\max,\ks}^{u,R,S,\pi,s}$}
&\multirow{5}{*}{~$\triangleq$~}
& the maximum value of $\tilde{g}_{\ks}^{\pi,s}$ can assume when the \\
&&  first $(u-4q - 1)$ updates on path $\pi$ have been fixed, and\\
&&  update $\calU_v$ is excluded from the computation of $\calU_s$ for\\
&&  $v \in R\cup S$ and for $v > u+ q$; for all $r \in A^{\pi,u}$, the\\
&& value of the update $\calU_r$ is already fixed;\\[0.1in]
$\tilde{g}_{\max,k_s}^{u,R,\pi,s}$
&~$\triangleq$~
& $\max_{S\subseteq	 [u - 4q, u + q]\}\setminus A^{\pi,u}\cup\{s\}} \tilde{g}_{\max,k_s}^{u,R,S,\pi,s}$;\\[0.1in]
$\overline{g}_{\max,k_s}^{R,\pi,s}$
&~$\triangleq$~
& $\max_{s-q\le u \le s} \tilde{g}_{\max,k_s}^{u,R,\pi,s}$;\\[0.1in]
$\overline{g}_{\max,k_s}^{\pi,s}$
&~$\triangleq$~
& $\overline{g}_{\max,k_s}^{\emptyset,\pi,s}$ ~~and~~~
$\overline{g}_{\mathsf{span},k_s}^{\pi,s}$
~$\triangleq$~
$\overline{g}_{\max,k_s}^{\pi,s} - \overline{g}_{\min,k_s}^{\pi,s}$.
\end{tabular}

%% file: bound-on-delta-x-span.tex
\subsubsection{Bounding $\left(\Df_t\right)^2 = \mathbb{E}\Big[\left(\Dbarspan{\pi,t}{k_t}\right)^2\Big]$}
\label{sec::span-bound}

We are now ready to bound $\left(\Df_t\right)^2$. 
Let $\nuo :=  \frac{20 q^2}{n}$ and $\nut = \frac  {24 q^2\Lresbar^2} {n\G^2} $.

\begin{lemma}
\label{lem::key-rec-bound-full}
\begin{align*}
\G \cdot \left(\Df_t\right)^2 ~\le~
\Big(\frac{\nuo}{q} + \frac {\nut}{q}\Big)\G \sum_{s\in [t- 5q, t+q]\setminus \{t\} } \Big[\big(\Df_s\big)^2 + \left( \DE_s \right)^2 \Big].
\hide{
\\
&\text{ii.}~~ \mathbb{E}\Big[\frac 2{\Gamma}\Big( \overline{g}^{\pi, t}_{\mathsf{span},k_t} \Big)^2\Big]
~\le~ \frac {\nut}{q} \sum_{s\in [t- 5q, t+q]\setminus \{t\} } \G\cdot \Big[ \big(\Df_s\big)^2 + \left(\DE_s \right)^2 \Big].
}
\end{align*}
\end{lemma}

\begin{proof}
Recall that $\Dbarspan{\pi,t}{k_t} = \max_{t-q\leq r\le t} \Dmax^{r} x_{\kt}^{\pi,t} - \min_{t-q \leq r'\le t}  \Dmin^{r'} x_{\kt}^{\pi,t}$.
We call $r$ and $r'$ the reference parameters.
By Lemma~\ref{lem:change-of-Dp-vs-change-of-g-gen},
we obtain a bound of
\begin{align}
\label{eqn::apply-lemma-dhat-change}
\max_{t-q \le r,r'\le t} \Big[ 2\big(x^{r,\pi,t}_{\kt} -x^{r',\pi,t}_{\kt} \big)^2
+ \frac 2{\G^2} \cdot 
\big(\tilde{g}^{r,\emptyset,\pi,t}_{\kt} - \tilde{g}^{r',\emptyset,\pi,t}_{\kt}\big)^2\Big],
\end{align}
where $x_{k_t}^{r,\pi,t}$ is the value of $x_{\kt}^{\pi,t}$ right before
Step 5 of update $\calU_t$ when $r$ is the reference parameter;
the maximum is also over the maximum and minimum possible values of the four terms in the above expression.

The first difference on the RHS of the above expression is going to involve updates to coordinate $x_{\kt}$,
i.e., updates $\calU_{k_s}$ with $k_s=\kt$ and $\min\{r-4q,r'-4q\} \leq s<t$.
We will consider the maximum and minimum possible values for these updates.
This suggests a bound of 
$\big( \Dmax^{r} x_{k_s}^{\pi,s} -  \Dmin^{r'} x_{k_s}^{\pi,s} \big)^2$
 for each such $s$.
But recall that the definition of $\Dmax^{r} x_{k_s}^{\pi,s}$ allows the exclusion of updates $\calU_v$ for a worst case set of $v$ when computing
$\calU_s$ (this is the effect of the maximization over $S$ in the definition of $\Dmax^r$), and therefore the first difference in \eqref{eqn::apply-lemma-dhat-change} may be as large as
$\big(\Dmax^{r} x_{\ks}^{\pi,s}\big)^2$ 
or $\big( \Dmin^{r'} x_{\ks}^{\pi,s}\big)^2$,
 meaning that the actual bound is
$\big(\max\big\{\big|\Dmax^{r} x_{k_s}^{\pi,s}\big|,
\big| \Dmin^{r'} x_{k_s}^{\pi,s}\big|,$
$\Dmax^{r} x_{k_s}^{\pi,s} -  \Dmin^{r'} x_{k_s}^{\pi,s}\big\}\big)^2
$.

By Lemma~\ref{Dmax-defn-works}(ii), we can modify the range of $r,r'$ from $[t-q,t]$ to $[\min\{s,t-q\},s] \subset [s-q,s]$ as $s<t$.
Thus the first term in \eqref{eqn::apply-lemma-dhat-change} is bounded by
\begin{align*}
\Big( \sum_{\substack{t-5q \leq s < t\\k_s=k_t}}
\Dbarvaralt{\{t\}} x^{\pi,s}_{k_s}
\Big)^2.
\end{align*}

We use a similar argument to bound the second term in \eqref{eqn::apply-lemma-dhat-change}.
The value of $g^{r,\emptyset,\pi,t}_{\kt}$ can vary due to the updates
$\Delta^r x_{k_s}^{\pi,s}$ for $r-4q\leq s \leq r+q$ and $s\ne t$;
equivalently, $s-q \leq r \leq s+4q$. Recall also that $t-q\le r \le t$.
Thus, for $s>t$, the relevant range of $r$ is $[s-q,t] \subset [s-q,s]$,
and for $s<t$, as before, the relevant range is also at most $[s-q,s]$.
Using the Lipschitz bound for the gradients, we obtain:
\begin{align}
\max_{t-q\le r,r'\le t} \big(
\tilde{g}^{r,\emptyset,\pi,t}_{\kt} - \tilde{g}^{r',\emptyset,\pi,t}_{\kt}
\big)^2
\le \Big( 
\sum_{\substack{t-5q \leq s \leq t+q\\s\ne t}}
\Lkskt \Dbarvaralt{\{t\}} x^{\pi,s}_{k_s} \Big)^2.
\label{eqn::gspan-bound}
\end{align}

\hide{
Note that the term on the LHS is exactly $\big( \overline{g}^{\pi, t}_{\mathsf{span},k_t} \big)^2$, which we need to bound to obtain result (ii) of the lemma.
Below, it will be the case that the second term on the RHS is
always a bound for $\big( \overline{g}^{\pi, t}_{\mathsf{span},k_t} \big)^2$, and we will thereby obtain result (ii) along with result (i).
}

Thus, using the Cauchy-Schwartz inequality for the second inequality below, we obtain

\begin{align*}
\left(\Dbarspan{\pi, t}{k_t}\right)^2
&\le
 2 
 \Big( \sum_{\substack{t-5q \leq s < t\\k_s=k_t}} \Dbarvaralt{\{t\}} x^{\pi,s}_{k_s}  \Big)^2
 +
 \frac {2} {\G^2}  \Big( 
 \sum_{\substack{t-5q \leq s \leq t+q\\s\ne t}} \Lkskt \Dbarvaralt{\{t\}} x^{\pi,s}_{k_s} \Big)^2 \\
&\le
{10}q \sum_{\substack{t-5q \leq s < t\\k_s=k_t}}  \big(\Dbarvaralt{\{t\}} x^{\pi,s}_{k_s}\big)^2
+\frac {{12}q} {\G^2}
\sum_{\substack{ t-5q\leq s \leq t+q\\ s\ne t}}
\Lkskt^2  \big(\Dbarvaralt{\{t\}} x^{\pi,s}_{k_s}\big)^2.
\end{align*}

Now we average over all $n$ choices of $\kt$; consequently, $\pi$ is now being viewed as a random variable
where $k_t$ on $\pi$ is being chosen uniformly at random, while the coordinates at times other than $t$ are fixed.
Notice that on all the paths $\pi$ being considered in the averaging, the value of each $\Delta_{\max}^{r, \{t\}} x_{k_s}^{\pi, s} $
is the same as their computation does not involve the update to $\xktt$,
and because at most the first $(t-4q)$ updates have been fixed in any of these terms,
none of the updates that could affect the update to $\xktt$ in its extended computation have been fixed;
similarly for $\Delta_{\min}^{r, \{t\}} x_{k_s}^{\pi, s}$.
Consequently, for each $s$, the term $\Dbarvar{k_s}{\pi,s}$ is the
same on each path. Thus the averaging is simply averaging the
values $\Lkskt$ as $\kt$ varies.
Recall that by Definition~\ref{def:Lipschitz-parameters},
$\Lresbar^2=\max_k \sum_{j=1}^n (\Lkj)^2$; this yields
\begin{align*}
&\mathbb{E}_{\kt}\Big[ \left(\Dbarspan{\pi, t}{k_t}\right)^2\big] \\
& \hspace*{0.2in} \le
 \frac{{10} q}{n} \sum_{ t- 5q \leq s < t} \big(\Dbarvaralt{\{t\}} x^{\pi,s}_{k_s}\big)^2
+\frac {{12}q} {n\G^2}
\sum_{\substack{ t- 5q \leq s \leq t+q\\ s\ne t}}
\Lresbar^2  \cdot
\big(\Dbarvaralt{\{t\}} x^{\pi,s}_{k_s}\big)^2\\
&\hspace*{0.2in} \le
 \frac{{10} q}{n} \sum_{ t-5q \leq s < t} \big(\Dbarvar{\pi,s}{k_s}\big)^2
+\frac {{12}q} {n\G^2}
\sum_{\substack{ t-5q \leq s \leq t+q\\ s\ne t}}
\Lresbar^2  \cdot
\big(\Dbarvar{\pi,s}{k_s}\big)^2~~\text{(by \eqref{eqn::dropRinDelta-var}).}
\numberthis\label{eqn:bound-on-delbar-diff}
\end{align*}

Now, $ \Delta  \xkspis \in \left[\Dmin^s \xkspis ,~\Dmax^s  \xkspis \right] \subseteq \left[\Dminbar \xkspis,~\Dmaxbar  \xkspis \right]$; thus,\\
$\left| \Dminbar \xkspis \right|$, $\left| \Dmaxbar  \xkspis \right|
\le$  $\left|\Delta  \xkspis \right|$ $+  \left(  \Dbarspan{\pi, s}{k_s} \right)$.
Also, $\left(\Dminbar \xkspis \right)^2$,\\ $\left(\Dmaxbar \xkspis \right)^2 \le 2\left( \Delta  \xkspis \right)^2 + 2 \left(  \Dbarspan{\pi, s}{k_s} \right)^2$.
So, $\left(\Dbarvar{\pi, s}{k_s}\right)^2 \leq 2\left( \Delta  \xkspis \right)^2$\\ $+~2 \left(  \Dbarspan{\pi, s}{k_s} \right)^2$.
Consequently,
\begin{align*}
\mathbb{E}_{\kt}\Big[ \left(\Dbarspan{\pi, t}{k_t}\right)^2\Big]
& \le
\mathbb{E}_{\kt}\bigg[
\frac{{20} q}{n} \sum_{ t-5q \leq s < t} \left( \Dbarspan{\pi, s}{k_s} \right)^2 + \left( \Delta  \xkspis \right)^2\\
& \hspace*{0.2in}+\frac {{24}q\Lresbar^2} {n\G^2}
\sum_{\substack{ t-5q \leq s \leq t+q\\ s\ne t}}
 \left( \Dbarspan{\pi, s}{k_s} \right)^2 + \left( \Delta  \xkspis \right)^2\bigg].
\end{align*}
Taking the expectation over every group, which on the RHS amounts to taking the expectation over every path $\pi$, yields

\begin{align*}
&\left(\Df_t\right)^2 = \expect{ \left( \Dbarspan{\pi, t}{k_t} \right)^2  }
\le
 \frac{{20}q}{n} \sum_{t-5q \leq s \leq t} \left(\expect{\left( \Dbarspan{\pi, s}{k_s} \right)^2} +  \expect{\left( \Delta  \xkspis \right)^2} \right)\\
 &\hspace*{1.2in} +
\frac {{24}q\Lresbar^2} {n\G^2}
\sum_{\substack{ t-5q \leq s \leq t+q\\ s\ne t}}
\left( \expect{\left( \Dbarspan{\pi, s}{k_s} \right)^2}
+ \expect{\left( \Delta  \xkspis \right)^2}  \right).
\end{align*}
Lemma~\ref{lem::key-rec-bound-full} follows.
\hide{
As already noted, on tracing back the calculations starting from~\eqref{eqn::gspan-bound}, looking just at the second terms on the RHS of each subsequent expression, we have also shown part (ii) as 
\begin{align*}
\mathbb{E}\Big[\frac 2{\Gamma}\left(\overline{g}^{\pi, t}_{\mathsf{span},k_t} \right)^2\Big]
&=~\mathbb{E}\Big[\frac 2{\Gamma}\cdot \max_{t-q\le r,r'\le t}\big(\tilde{g}^{r,\emptyset,\pi,t}_{\kt} - \tilde{g}^{r',\emptyset,\pi,t}_{\kt}\big)^2\Big].

\end{align*}
}
\end{proof}

%% file: SACD-gradient-bounds.tex
\newcommand{\nutr}{\nu_3}
\newcommand{\nuf}{\nu_4}
\newcommand{\Lam}{\Lambda}

\subsubsection{Gradient Bounds}
\label{sec::grad-bounds}

In the previous subsection, one of the terms being bounded
was $(\tilde{g}^{r,\emptyset,\pi,t}_{\kt} - \tilde{g}^{r',\emptyset,\pi,t}_{\kt})^2 = (\overline{g}_{\max, k_t}^{\pi, t} - \overline{g}_{\min, k_t}^{\pi, t})^2$.

Here, we will bound term $E$, $\left(g_{k_t}^{\pi,t}  - \tilde{g}_{k_t}^{\pi,t}\right)^2$.
Unfortunately, $g_{k_t}^{\pi,t}$ might not be in $\big[\overline{g}_{\min, k_t}^{\pi, t}, \overline{g}_{\max, k_t}^{\pi, t}\big]$
and so we cannot simply apply Lemma~\ref{lem::key-rec-bound-full}.
The reason is that $g_{k_t}^{\pi,t}$ is a function of $x^{\pi(\kt),t}$,
and this up-to-date $x$ value could depend on the
value of update $\calU_t$; for recall that $\calU_t$ might finish before some updates $\calU_s$ with $s<t$,
and the latter updates could then read the updated value of the coordinate being updated by $\calU_t$.
As already explained, there are also other ways that the choice of $\kt$ by update $\calU_t$ could affect earlier updates.

We start by upper bounding
term $E$ by
$\big(\sum_{l_0 \in [t - 2q, t - 1]}  L_{k_{l_0}, k_t} \Dvar{\pi, l_0}{k_{l_0}}{t, \emptyset}\big)^2$.
The challenge is that changing either coordinate $k_{l_0}$ or $k_t$ may change the value of $\Delta^{t, \emptyset}_{\max} x^{\pi, l_0}_{k_{l_{0}}}$ or of $\Delta^{t, \emptyset}_{\min} x^{\pi, l_{0}}_{k_{l_{0}}}$
which implies that a simple averaging of the terms $L_{k_{l_0}, k_t}^2$ to obtain a term $\Lresbar^2$ as was done to obtain \eqref{eqn:bound-on-delbar-diff} is not possible.
Instead, we will bound this term recursively.
The somewhat more general result we will need is stated in the following lemma.
Its proof, which is quite involved, is deferred to Appendix~\ref{app::rec-bound-gen}.
\hide{
The difficulty we face is that
we cannot exclude any non-empty $R$
in performing the analysis, while in the previous subsection we could clearly exclude the set $\{t\}$.
It turns out that deriving an upper bound under these conditions
is considerably more challenging.
}

Let $\Lam^2  =  \frac{\Lresbar^2}{\Gamma^2} + 1$, $r = \frac{160 q^2 \Lam^2 }{n}$, $\nu_3 =  \frac{3}{16}(\frac{r^2}{1 - r} + r)$, and $\nu_4 = \frac{6r}{1 - r}$.
\begin{lemma}
\label{lem::grad::sync::bound::gen}
For any $u \in [t - 2q, t]$, if $r<1$, then
\begin{align*}
&\mathbb{E}\bigg[\Big(\sum_{l_0 \in [t - 4q, t + q] \setminus \{u\}}  L_{k_{l_0}, k_u} \Dvar{\pi, l_0}{k_{l_0}}{t, \emptyset}\Big)^2 \bigg] \\
&\hspace*{0.2in} \leq \frac{\nu_3 \Gamma^2}{q}
\sum_{s \in [t - 7q, t+q] \setminus \{u\}}
\left[ \left(\Df_s \right)^2 + \left( \DE_s \right)^2 \right]
+ \nu_4 \Gamma^2 \left[\left(\Df_u \right)^2 + \left( \DE_u \right)^2 \right].
\end{align*}
\end{lemma}

Then, on substituting for $\Df_t$ from Lemma~\ref{lem::key-rec-bound-full}, 
we obtain the following
bound on term $E$.

\begin{clm}\label{clm::bound::E}[Bounding Term $E$]
\label{lem::g-gtilde-diff}
If $r< 1$,
term $E$ is bounded by:
\begin{align*}
&\frac{\nutr \Gamma}{q}
\sum_{s \in [t - 7q, t+q] \setminus \{t\}}
\left[ \left(\Df_s \right)^2 + \left( \DE_s \right)^2 \right]
\\
&~~+\nuf \Gamma \Big( \frac{\nuo}{q} + \frac {\nut}{q}\Big)
\sum_{s\in[t-5q,t+q]\setminus \{t\}}
\left[ \left(\Df_s \right)^2 + \left( \DE_s \right)^2 \right]
+ \nuf \Gamma\left( \DE_t \right)^2.
\end{align*}
\label{lem::key-bound}
\end{clm}

With more effort one can also bound the gradient difference terms $A$ and $D$
as follows, as shown in the appendix.

\begin{clm}\label{clm::bound::A}[Bounding Term $A$]
If $r< 1$,
term $A$ is bounded by:
\begin{align*}
& \frac{2(\nutr + \nuf)\Gamma}{n}  
\sum_{s \in [t - 7q, t+q] \setminus \{t\}}
\left[ \big(\Df_s \big)^2 + \left( \DE_s \right)^2 \right] + \frac{\Gamma}{n} \sum_{s \in [t - 2q, t - 1]}\left( \DE_s\right)^2
\\
&~~~~~~+\frac{2\nutr\Gamma}{n} \Big(\frac{\nuo}{q} +  \frac{\nut}{q}\Big)
\sum_{s\in[t-5q,t+q]\setminus \{t\}}
\left[ \big(\Df_s \big)^2 + \left( \DE_s \right)^2 \right]
+ \frac{ 2\nutr\Gamma }{n} \left( \DE_t \right)^2.
\end{align*}
\end{clm}

\begin{clm}\label{clm::bound::D}[Bounding Term $D$]
If $r< 1$,
term $D$ is bounded by:
\begin{align*}
&\frac{ 2\nut \Gamma}{3q} \sum_{\substack{s\in[t-2q,\\
~~~~~~t-1]}}
\left[ \big(\Df_s \big)^2 + \left( \DE_s \right)^2 \right]
+\frac{ 4\nutr\Gamma}{3q}
\sum_{\substack{s \in [t - 7q,\\
~~~~~~~~ t+q] \setminus \{t\}}}
\left[ \big(\Df_s \big)^2 + \left( \DE_s \right)^2 \right]\\
&~~+\frac{4\nuf \Gamma}{3} \left( \frac{\nuo}{q} + \frac {\nut}{q}\right)
\sum_{s\in[t-5q,t+q]\setminus \{t\}}
\left[ \big(\Df_s \big)^2 + \left( \DE_s \right)^2 \right]
+ \frac{4\nuf \Gamma}{3}\left( \DE_t \right)^2.
\end{align*}
\end{clm}

In Section~\ref{sec::span-bound}, to obtain a bound on $\mathbb{E}\left[\Dbarspan{\pi,t}{k_t}\right]$, we bound $\Dbarspan{\pi,t}{k_t}$ in terms of $\Dbarspan{\pi, s}{k_s}$ for $s \in [t - 5q, t + q]$. In contrast, in the proof of Lemma~\ref{lem::grad::sync::bound::gen}, we start by bounding $\Delta_{\mathsf{var}} x^{\pi, l_0}_{k_{l_0}}$ in terms of various $\Delta_{\mathsf{span}} x^{\pi, l_1}_{k_{l_1}}$;
but then, for each $l_1$, we bound  $\Delta_{\mathsf{span}} x^{\pi, l_1}_{k_{l_1}}$  in terms of various $\Delta_{\mathsf{span}} x^{\pi, l_2}_{k_{l_2}}$; we continue recursively until one of the following two cases occurs.
\\
1. When the recursion reaches a term $\Delta_{\mathsf{span}} x_{k_s}^{\pi,s}$
with $s < u-q$, as this term does not depend on $\calU_u$, one can safely average over $k_u$, thereby replacing the multiplier $L_{k_0k_u}^2$ by $\frac 1n \Lresbar^2$.
\\
2. When the recursion reaches a term $\Delta_{\mathsf{span}} x_{k_u}^{\pi,u}$ it needs to stop. Now to remove the multiplier $L_{k_0k_u}^2$ we simply upper bound it by $\Lmax^2$.
Unfortunately, there is no averaging and so we ``lose'' a factor of $\frac 1n$, which could otherwise more than compensate for a growth by a factor of $q^2$, as in the non-recursive analysis.
We save
one $\Theta(q)$ factor by using an \emph{unbalanced} Cauchy-Schwartz inequality described in the appendix.
(This might not appear to be sufficient, but in fact it is.
The reason is that for each time $u$, this error term originates from $u$ itself.
The other error terms originate from $\Theta(q)$ times in a range $u\pm\Theta(q)$, and so are $\Theta(q)$ times as numerous, so in fact we are saving a $\Theta(q^2)$ factor.)

%% file: Errt-bound.tex
\newcommand{\Apim}{A_-^{\pi}}
\newcommand{\Apip}{A_+^{\pi}}
\newcommand{\Cpi}{C^{\pi}}
\newcommand{\gamo}{\gamma_1}
\newcommand{\gamt}{\gamma_2}
\newcommand{\gkpit}{g_k^{\pi,t}}
\newcommand{\gktSpikt}{\tilde{g}_k^{S,\pi(k),t}}
\newcommand{\gkSpikt}{g_k^{S,\pi(k),t}}
\newcommand{\gktSpiktt}{\tilde{g}_k^{S,\pi(\kt),t}}
\newcommand{\gkSpiktt}{g_k^{S,\pi(\kt),t}}
\newcommand{\gkspikts}{g_{\ks}^{\pi(\kt),s}}
\newcommand{\gkspikt}{g_{\ks}^{\pi(\kt),t}}
\newcommand{\gktpiktt}{g_{k_t}^{\pi(k_t),t}}

\newcommand{\gkspiktprevtks}{g_{\ks}^{\pi(\kt),\prev(t,\ks)}}
\newcommand{\gkpiktprevtk}{g_k^{\pi(k_t),\prev(t,k)}}
\newcommand{\gkpikprevtk}{g_k^{\pi(k),\prev(t,k)}}
\newcommand{\gmaxkspit}{g_{\max,k_s}^{\pi,t}}
\newcommand{\gminkspit}{g_{\min,k_s}^{\pi,t}}
\newcommand{\nuth}{\nu_3}
\newcommand{\pik}{\pi(k)}
\newcommand{\piks}{\pi(\ks)}
\newcommand{\pikt}{\pi(\kt)}
\newcommand{\Ppi}{P^{\pi}}
\newcommand{\Rpi}{R^{\pi}}
\newcommand{\Spi}{S^{\pi}}
\newcommand{\tgktpiktt}{\tilde{g}_{k_t}^{\pi(\kt),t}}
\newcommand{\xkspiks}{x_{\ks}^{\pi(k),s}}
\newcommand{\xkspikss}{x_{\ks}^{\pi(\ks),s}}
\newcommand{\xkspiktprevtks}{x_{\ks}^{\pi(\kt),\prev(t,\ks)}}
\newcommand{\xkspiktone}{x_{\ks}^{\pi(\kt),t-1}}
\newcommand{\xkspikstone}{x_{\ks}^{\pi(\ks),t-1}}
\newcommand{\xkspikttone}{x_{\ks}^{\pi(\kt),t-1}}
\newcommand{\xkpitone}{x_{k}^{\pi,t-1}}
\newcommand{\xktpitone}{x_{k_t}^{\pi,t-1}}
\newcommand{\xkpiktone}{x_k^{\pi(k),t-1}}
\newcommand{\xkpikttone}{x_k^{\pi(k_t),t-1}}
\newcommand{\xktpikttone}{x_{k_t}^{\pi(k_t),t-1}}

\subsubsection{Finalizing the bound on $\Err_t$}
\label{sec:Errt-final-bdd}

We finish Part 2 of the analysis by expressing the bound on $\expect{\Err_t}$ in terms of
$\left(\Df_s \right)^2$ and $\left(\DE_s \right)^2$. Recall that $\Lam^2  =  \frac{\Lresbar^2}{\Gamma^2} + 1$ and $r = \frac{160 q^2 \Lam^2 }{n}$,
and we make $q$ sufficiently small to ensure that $r<1$.

\begin{lemma}
\label{lem::final-bound-on-Errt}
\begin{align*}
&\expect{\Err_t} \le  \Big(\hide{\varpi + }\frac{15r}{1-r}\Big)\G \left(\DE_t\right)^2 +\varpi \Gamma  \sum_{s\in[t-7q , t+q] \setminus \{t\}} \left[\big(\Df_s \big)^2
+ \left(\DE_s \right)^2\right] \\
\hide{& \hspace*{1in}+ \varpi \Gamma \big(\Df_t\big)^2,\\}
&\text{where }\varpi = \frac{1}{q} \left[ \frac{2r}{3} + \frac{3r^2}{1280} + \frac{9r^3}{25600}  + \frac{3r^2}{1 - r}  + \frac{r^3}{426(1 - r)} + \frac{r^4}{2844(1-r)} \right].
\end{align*}
\end{lemma}
\begin{proof}
We begin with the bound in Lemma~\ref{lem::F-prog-general}.
\hide{
 Terms $A$, $D$ and $E$ are expectations of gradient differences;
they can be bounded using Lemmas~\ref{lem::key-rec-bound-full}--\ref{lem::g-gtilde-diff}, along
with Lemma~\ref{lem::grad::sync::bound::gen}, which provides a bound stemming from the recursive formulation mentioned earlier.
The upper bounds on the three terms are all of the form
\[
\frac{ O(\nu_1,\nu_2,\nu_3,\nu_4) }{q} \cdot \Gamma \cdot \sum_s \left[\big(D_s\big)^2 + (\Delta_s^X)^2\right] ~+~ O(\nu_4) \cdot \Gamma (\Delta_t^X)^2,
\]
where $s$ runs over a neighborhood of $t$ of length $O(q)$. In the appendix, we compute these upper bounds precisely (see Claims~\ref{clm::bound::A},
\ref{clm::bound::D}, and~\ref{clm::bound::E}).
}
We have already bounded terms $A$, $D$, and $E$.
Terms $B$ and $C$ are bounded as follows.

\begin{clm} \label{clm::bound::B}[Bounding Term $B$]
Term $B$ is bounded by
\begin{align*}
\mathbb{E}\Bigg[\frac{2}{3n^2} \sum_{\substack{t-2q \le s < t\\ \& s = \prev(t, \ks)}}\sum_{k_t=1}^n
\Gamma \cdot  \left(x^{\pi(k_s), t}_{k_s} - x^{\pi(k_t), t}_{k_s} \right)^2\Bigg] \leq \frac{\nuo}{15q}  \sum_{s\in [t-2q, t-1] } \G\cdot \big(\Df_s\big)^2.
\end{align*}

\end{clm}

\begin{clm}
\label{clm::bound::C}[Bounding Term $C$]
Term $C$ is bounded by
\begin{align*}
\mathbb{E}\Bigg[ \frac {\G}{2n^2}  \sum_{\substack{t-2q \le s < t\\ \& s = \prev(t,\ks)}}\sum_{k_t=1}^n
 \left(\Delta \xkspikts \right)^2\Bigg]
&\le \frac {\G} {2n} \sum_{t-2q \le s \le t-1}\left(\DE_s\right)^2.
\end{align*}
\end{clm}

Summing up these bounds yields
\begin{align*}
&\expect{\Err_t}
\le\hide{ \frac 1{3n^2} \expectpi{ \sum_{k=1}^n \sum_{k_t=1}^n \hWk(\gkpiktt,x^{\pi(k_t), t}_k)  }
+} \left( \frac{2\nu_3}{n} + \frac{4\nu_4}{3} +  \nu_4 \right)\G \left(\DE_t\right)^2 \\
& \hspace*{0.1in}+\underbrace{\left[ \max\left\{\frac{\nuo}{15q}, \frac{3}{2n}\right\} + \frac{2\nu_2}{3q} +  \frac{2(\nutr + \nu_4)}{n} + \frac{7 \nu_3}{3q} + \left(\frac{2\nu_3}{nq} +  \frac{7 \nu_4}{3q} \right)\left(\nu_1 + \nu_2\right)\right]}_{G} \\
\numberthis\label{eqn:nu-coeff-bound}
& \hspace*{1.5in} \cdot \Gamma \sum_{s\in[t-7q , t+q]} \left( \left(\Df_s\right)^2 + \left(\DE_s \right)^2\right).
\end{align*}

Via some elementary calculations, presented in Claim~\ref{clm::convert-nu-bound-to-r-bound} in Appendix~\ref{sec::claims},
we show that the quantity $G$ above can be bounded by
\[
 \frac{1}{q} \left[ \frac{2r}{3} + \frac{3r^2}{1280} + \frac{9r^3}{25600}  + \frac{3r^2}{1 - r}  + \frac{r^3}{426(1 - r)} + \frac{r^4}{2844(1-r)} \right]\]
and that $\frac{2\nu_3 }{n} + \frac{7 \nu_4}{3} \le \frac{15r}{1-r}$ as $1 \leq q < n$ and $r \leq 1$, which implies the bounds stated in the lemma.

\hide{
\begin{align*}
\Err_t
\le &
&+\varpi \G \left(\DE_t\right)^2
+\varpi \Gamma  \sum_{s\in[t-5q , t+q] } \left[\left(\Df_s \right)^2 + \left(\DE_s \right)^2\right].
\end{align*}
The bound stated in the lemma follows immediately.
}

\end{proof}

%% file: new-amortization.tex
\subsection{Part 3: The Amortization}
\label{sec::amortization}

We let $\varrho = \frac{1}{8} - \frac{15r}{1 - r}$.
From Lemmas~\ref{lem::F-prog-general} and~\ref{lem::final-bound-on-Errt},
we obtain:

\begin{align*}
 \expectpi{F(x^{t}) -F(x^{t+1})}
\ge &\frac 1{3n^2}\cdot \mathbb{E}\bigg[ \sum_{k=1}^n \sum_{k_t=1}^n \hWk(g^{\pi(k_t, t), t}_k,x^{\pi(k_t, t), t}_k)  \bigg]\\
&+\varrho \G \left(\DE_t\right)^2
-\varpi \Gamma  \sum_{s\in[t-7q , t+q] } \left[\big(\Df_s \big)^2 + \left(\DE_s \right)^2\right].
\numberthis\label{eqn::expected-F-progress}
\end{align*}

The term $\left( \DE_s \right)^2$ and $\left(\Df_s \right)^2$ in \eqref{eqn::expected-F-progress} will be paid for by the progress terms from time $s$ by means of an amortization. Also, we will account for the term $\left( \Df_t \right)^2$ using the bound from Lemma~\ref{lem::key-rec-bound-full}: 
\begin{align*}
\G \cdot \left(\Df_t\right)^2 & \le
\left(\frac{\nuo}{q} + \frac {\nut}{q}\right)\G \sum_{s\in [t-5q, t+q]\setminus \{t\} } \left(\big(\Df_s\big)^2
+  \left(\DE_s \right)^2\right). \numberthis \label{ineq::new::2}
\end{align*}
For the purpose of amortizing the $\left(\Df_s\right)^2$ terms, for some constant
$\gamma > 0$ which we will specify later,
we add terms $+ \gamma \G \left( \Df_t \right)^2 -  \gamma \G \left( \Df_t \right)^2$ to the bound from \eqref{eqn::expected-F-progress},
and then we use \eqref{ineq::new::2} to bound $( \gamma + \varpi ) \G \left( \Df_t \right)^2$, which yields

\begin{align*}
&\expectpi{F(x^{t}) -F(x^{t+1})}
\ge \frac 1{3n^2} \cdot \mathbb{E}\bigg[ \sum_{k=1}^n \sum_{k_t=1}^n \hWk(g^{\pi(k_t, t), t}_k,x^{\pi(k_t, t), t}_k)  \bigg]\\
& ~~\hspace*{0.1in}+\left( \varrho - \varpi \right) \G \left(\DE_t\right)^2 + \gamma \G \big( \Df_t \big)^2
-\varpi \Gamma  \sum_{s\in[t-7q , t+q] \setminus \{t\}} \left[\big(\Df_s \big)^2 + \left(\DE_s \right)^2\right]\\
& ~~\hspace*{0.1in}- (\varpi + \gamma) \bigg[ \left(\frac{\nuo}{q} + \frac {\nut}{q}\right)\G \sum_{s\in [t-5q, t+q]\setminus \{t\} } \left(\big(\Df_s\big)^2
+ \left(\DE_s \right)^2\right)\bigg].  \numberthis \label{ineq::new::3}
\end{align*}

In the standard convergence analysis for a sequential stochastic coordinate descent,
one shows that
\begin{align*}
\expect{F(x^{t}) -F(x^{t+1})}
&\ge \frac 1{3n} \cdot \mathbb{E}\bigg[ \sum_{k=1}^n \hWk(g^{\pi(k_t, t), t}_k,x^{\pi(k_t, t), t}_k)  \bigg]\\
&= \frac 1{3n^2} \cdot \mathbb{E}\bigg[ \sum_{k=1}^n \sum_{k_t=1}^n \hWk(g^{\pi(k_t, t), t}_k,x^{\pi(k_t, t), t}_k)  \bigg].
\end{align*}

To obtain an analogous bound, we have to show all the additional terms
in \eqref{ineq::new::3} make a non-positive contribution over the
$T$ steps of the algorithm, analogous to the use of \eqref{eqn::basic-amort-inequality} 
in the basic framework.
To this end, we apply the following theorem regarding rates of convergence.
This theorem uses amortization terms $\Ap$ and $\Am$,
to define a potential function $H(t)$,
which could also be viewed as a Lyapunov function.
We note that the same result, but without the amortization terms
$\Ap$ and $\Am$, can be found in~\cite{richtarik2014iteration}.
$\Ap$ represents progress that has occurred, but which is being saved to pay for future errors; $\Am$, in contrast, represents the effect of errors in the past, which will be paid for by future progress.

\begin{theorem}\label{thm:meta-new::re}
Suppose that $\G \ge \Lmax$.
Let $q$ be a fixed integer parameter.
	Let $\Ap(t)$, $\Am(t)$ be non-negative functions with $\Ap(1) = 0$, $\Am(T+1) = 0$,
	and let $H(t):=\expect{F(\pt)} + \Ap(t) - \Am(t)$.
	Suppose that
	\begin{enumerate}[label=\alph*.]
		\item $H(t) \ge 0$ for all $t\ge 1$;
		\item for all $t\ge 1$, $H(t+1) \le H(t)$, i.e., $H(t)$ is a decreasing function of $t$;
		\item there exist constants $\alpha,\beta > 0$ such that for any $t\ge 1$,
\[ H(t) - H(t+1) ~\ge~  \frac \alpha n ~ \mathbb{E}\bigg[\sum_{k=1}^n ~ \hWk(\nabla_k f(x^{t}),x_k^{t}) \bigg] + \frac{\beta}{n} \cdot \Ap(t).\]
	\end{enumerate}

	\noindent
	\emph{(i)} If $F$ is strongly convex with parameter $\muF$, 
	and $f$ has strongly convex parameter $\muf$, then for all $T\ge 0$,
\[
\expect{F(x^{T+1})} ~\le~	H(T+1) ~\le~ \left[ 1 - \min\left\{\frac {\alpha}{n} \cdot \frac{\muF} {\muF + \G - \muf}~,~ \frac{\beta}{n} \right\} \right]^{T} \cdot F(x^1).\]

	\noindent
	\emph{(ii)} Now suppose that $F$ is convex.
	Let $\calR$ be the radius of the level set for $x^1$. Formally,
	let $X = \{x \,|\, F(x) \le F(x^1)\}$;
	then
	$\calR = \sup_{x\in X} \inf_{x^* \in \Xs} \|x - x^*\|$.
	Then, for all $T\ge 0$,
	\[
\expect{F(x^{T+1})} \le	H(T+1) \le \frac{F(x^1)}{1 + \min \left\{ \frac{\beta}{2n\cdot F(x^1)}, \frac{\alpha}{4n\cdot F(x^1)}, \frac{\alpha}{8n\G \calR ^2} \right\}
	\cdot F(x^1) \cdot T}.
	\]
\end{theorem}

\hide{
In the standard convergence analysis, in order to show a convergence for function $F$, one needs to show
\begin{align*}
\expectpi{F(x^{t}) -F(x^{t+1})}
&\ge \frac 1{3n^2} \cdot \mathbb{E}_{\pi}\bigg[ \sum_{k=1}^n \sum_{k_t=1}^n \hWk(g^{\pi(k_t, t), t}_k,x^{\pi(k_t, t), t}_k)  \bigg].
\end{align*}

Here, in order to deduce a rate of convergence,
we will be applying Theorem~\ref{thm:meta-new::re}.
To this end, we need to create a potential function
\begin{align*}
H(t) = \mathbb{E}\left[F(x^t)\right] + A^+(t) - A^-(t),
\end{align*}
where $A^+$ and $A^-$ are non-negative.
To obtain a good convergence rate for $H(t)$, since $H(t)$ is upper bounded by $F(x^t) + A^+(t)$, we want to show

\begin{align*}
H(t) - H(t+1) \geq \frac 1{3n^2}  \mathbb{E}_{\pi}\bigg[ \sum_{k=1}^n \sum_{k_t=1}^n \hWk(g^{\pi(k_t, t), t}_k,x^{\pi(k_t, t), t}_k)  \bigg] + \frac{1}{3n} A^+(t), \numberthis \label{ineq::desire}
\end{align*}
where, intuitively, $\frac{1}{3n} A^+(t)$ provides the needed progress for $H - F$.
This inequality will allow us to apply Theorem~\ref{thm:meta-new::re} with $\beta = \frac{1}{3n}$.
}

We will be applying Theorem~\ref{thm:meta-new::re}
with $\alpha= \beta = \frac {1}{3n}$.

In order to obtain condition (c) of Thereom~\ref{thm:meta-new::re} from \eqref{ineq::new::3}, it suffices to show
\begin{align*}
&\left[\left(1 - \frac{1}{3n} \right) A^+(t) - A^-(t)\right] - \left[ A^+(t+1) - A^-(t+1) \right] \\
&\hspace*{0.1in} \geq - \left( \varrho - \varpi \right) \G \left(\DE_t\right)^2 - \gamma \G \big( \Df_t \big)^2 + \varpi \Gamma  \sum_{s\in[t-7q , t+q] \setminus \{t\}} \left[\big(\Df_s \big)^2 + \left(\DE_s \right)^2\right]\\
& \hspace*{0.2in}+ (\varpi + \gamma) \bigg[ \left(\frac{\nuo}{q} + \frac {\nut}{q}\right) \G\sum_{s\in [t-5q, t+q]\setminus \{t\} } \left(\big(\Df_s\big)^2
+  \left(\DE_s \right)^2\right)\bigg].  \numberthis \label{ineq::final::H}
\end{align*}

\begin{lemma}
\label{lem::ApAmProg}
Inequality~\eqref{ineq::final::H} holds
if $7q < n$, $c = \varpi +  (\gamma + \varpi) \left(\frac{\nuo}{q} + \frac{\nut}{q}\right)\Gamma$,  $\gamma = \varrho - \varpi$,
$\Lam =\frac{\Lresbar^2}{\Gamma^2} +  1$,
 $ r= \frac{160 q^2 \Lam
 }{n} \leq \frac{1}{225}$, and
\begin{align*}
&A^+(t)= \sum_{s=t-7q}^{t-1}~\sum_{v=t}^{s+7q} \frac {1}{\left(1 - \frac 1{3n}\right)^{v-t+1}} \left[ c \big( \Df_s \big)^2 + c \left( \DE_s\right)^2\right],\\
&A^-(t) = \sum_{s=t-q}^{t-1}~\sum_{v=t}^{s+q}  \left[ c \left( \Df_v \right)^2 + c \left( \DE_v\right)^2\right].
\end{align*}
\end{lemma}
We note that as $\Df_t=0$ and $\DE_t=0$ for $t=0$ and for $t\ge T+1$,
with the above definition,
$\Ap(1)=0$ and $\Am(T+1)=0$.

We are now ready to conclude the proof of our main result.

\begin{pfof}{Theorem \ref{thm:main-SACD}}
We set $\alpha = \beta = \frac 13$.

By Lemma~\ref{lem::ApAmProg}, if $r \le \frac 1{225}$,
the conditions for applying Theorem~\ref{thm:meta-new::re} hold:
(c) holds by construction; this implies that (b) holds as the RHS of (c) is
non-negative; finally, as $\Am(T+1) = 0$, $H(T+1) \ge 0$, and together
with (b) this implies (a).
As $A^-(T+1) = 0$, we conclude that $\expect{F(x^{T+1})} \le H(T+1)$; this inequality also holds in expectation, thus we are done.

We now apply Theorem~\ref{thm:meta-new::re}, which yields the stated results.
Recall that
$ r= \frac{160 q^2 \left( \frac{\Lresbar^2}{\Gamma^2} + 1 \right)}{n}$.
Thus, to achieve $r \le \frac 1 {225}$
it suffices to have $q \le \min \left\{\frac{\G \sqrt n } {270 \Lresbar}, \frac{\sqrt{n}}{270} \right\}$.
\end{pfof}

Note that we have not sought to fully optimize the constants.

%% file: appendix_general.tex
\section{Omitted Proofs and Subsidiary Lemmas}\label{sect:appendix-general}

\hide{
\newcommand{\Dxktau}{\Delta x_k^\tau}
\newcommand{\Dxktautau}{\Delta x_{k_\tau}^\tau}
\newcommand{\Dxktt}{\Delta x_{k_t}^t}

\newcommand{\xstar}{x^*}
\newcommand{\xt}{x^t}
\newcommand{\xto}{x^{t-1}}
\newcommand{\xktto}{x_{k_t}^{t-1}}

\newcommand{\Dmaxs}{\Delta_{\max}^s}
\newcommand{\Dmaxt}{\Delta_{\max}^t}
\newcommand{\Dmaxtp}{\Delta_{\max}^{t'}}
\newcommand{\Dmaxup}{\Delta_{\max}^{u'}}
\newcommand{\Dmaxbar}{\overline{\Delta}_{\max}}
\newcommand{\Dmaxu}{\Delta_{\max}^u}

\newcommand{\Dmins}{\Delta_{\min}^s}
\newcommand{\Dmint}{\Delta_{\min}^t}
\newcommand{\Dmintp}{\Delta_{\min}^{t'}}
\newcommand{\Dminbar}{\overline{\Delta}_{\min}}
\newcommand{\Dminu}{\Delta_{\min}^u}

\newcommand{\dminus}{d^{-}}
\newcommand{\dplus}{d^{+}}
\newcommand{\dpi}{d\hspace*{0.01in}'}
\newcommand{\hhd}{\breve{d}}
}

We begin with the proof of Lemma~\ref{lem::SCCrange}.
Next, in Appendix~\ref{subsec::progress lemmas}, we prove Lemmas~\ref{lem:F-prog-one} and~\ref{lem:F-prog-two},
the basic progress lemmas.
In Appendix~\ref{app-sec::exptd-prog}, we prove Lemma~\ref{lem::F-prog-general}, the progress lemma for the general analysis.
Then, in Appendix~\ref{subsect:Wd},
we give several bounds on how much $\hW$ can change when
one of its arguments is altered, leading to proofs of Lemmas~\ref{lem:W-shift-re}--\ref{lem:change-of-Dp-vs-change-of-g-gen}.
We follow this, in Appendix~\ref{app::rec-bound-gen},
with the proof of the recursive bound given in Lemma~\ref{basic::recursion::gen},
which is used to show Lemma~\ref{lem::g-gtilde-diff}.
We continue, in Appendix~\ref{sec::claims}, with the proofs of the claims from Section~\ref{sec:Errt-final-bdd}.
Finally, in Appendix~\ref{app:good-progress}, we prove Theorem~\ref{thm:meta-new::re} and Lemma~\ref{lem::ApAmProg}.

\begin{pfof}{Lemma~\ref{lem::SCCrange}} %
Let $\calU_b$ be an update to coordinate $x_j$ with start and finish times
$s_b$ and $f_b$, resp.
Let $\calU_a$ be the update to coordinate $x_j$ with the earliest start time before $s_b$ that commits later than $\calU_b$, if any.
Let $\calU_c$ be the update to coordinate $x_j$ with the latest start time
after $s_b$ that commits
earlier than $\calU_b$, if any.
Let the start and finish times for $\calU_a$ and $\calU_c$ be
$(s_b,f_b)$, and $(s_c,f_c)$, resp.
Note that the start and commit times of an update differ by at most $q+1$
as there are at most $q$ interfering updates for each update.

Suppose the SCC order for $\calU_b$ is $s$.
If $\calU_a$ does not exist then $s\ge s_b$; it is convenient to set $(s_a,f_a)=(s_b,f_b)$ in this case.
Similarly, If $\calU_c$ does not exist then $s\le s_b$; it is convenient to set $(s_c,f_c)=(s_b,f_b)$ in this case.
Then it is always the case that $s_a \le s \le s_c$.

\smallskip
\noindent
Case i. $s<s_b$.\\
Then $s_a \le s < s_b <f_b < f_a \le s_a + q +1$.
It follows that $s_b \le s_a+q-1$ and hence $s\ge s_b -q +1$.
Also $f_b \le s_b+q+1 \le s_a+q-1 \le s+q-1$.

\smallskip
\noindent
Case ii. $s > s_b$.\\
Then $s_b <s \le s_c < f_c < f_b \le s_b+q+1$.
It follows that $s  \le s_b+q-1$ and $f_b \le s+q$.

\smallskip
\noindent
Case iii. $s = s_b$.\\
Then $f_b\le s +q+1$.

Therefore, for any updates, their SCC order $s$ and their commit time $f_b$ satisfy $s < f_b \leq s + q + 1$.
\hide{\YKC{Since the commit time satisfies $s+1\le  f_b \leq s + q + 1$, while we also have $1\le f_b - s_b\le q+1$, so we should have $s+q \ge s_b \ge s-q$?
But the first part of Lemma 1 says the start time of $\calU_t$ lies in $[\max\{1,t-q-1\},t+q]$, which is slightly weaker.}\YXT{I think the range $[\max\{1,t-q-1\},t+q]$ is a typo. It should be $[\max\{1, t - q + 1\}, t+q - 1]$. This is not hard to see. In the first two cases, other than the strict inequality between $s$ and $s_b$, there are other two strict inequality. The entire range is $q+1$, so the difference between $s$ and $s_b$ is at most $q - 1$.}
\RJC{It is simpler than this. Cases ii and iii have $s \ge s_b$. Case i shows
$s \ge s_b -q+1$. I have corrected the typo.}}
Now we determine the SCC range for updates that might interfere
$\calU_b$. Note that they have commit time in the range $[s_b+1,f_b-1]$.

\smallskip
\noindent
Case i. $s<s_b$.\\
The commit time for possibly interfering updates
is in the range $[s_b+1,f_b-1]\subseteq [s,s+q-2]$ and hence their SCC rank
is in the range $[s-(q+1),s+q-3] = [s-q-1,s+q-3]$.

\smallskip
\noindent
Case ii. $s > s_b$.\\
The commit time for possibly interfering updates
is in the range $[s_b+1,f_b-1]\subseteq [s-q+2, s+q]$ and hence their SCC rank
is in the range $[s-q+2-(q+1),s+q-1]=[s-2q+1,s+q-1]$.

\smallskip
\noindent
Case iii. $s = s_b$.\\
The commit time for possibly interfering updates
is in the range $[s_b+1,f_b-1]\subseteq[s+1,s+q]$ and hence their SCC rank
is in the range $[s+1-(q+1),s+q-1]=[s-q,s+q-1]$.

For all $q\ge 1$, this range is contained in $[s-2q+1,s+q-1]$.
\hide{
Now we show that $f\le s+q+1$.
If $s=t$, then clearly $f \le s+q+1$.
Now suppose that $U$ acquires CT time $t'$ and $t'<t$.
Then $U'$ commits after $U$, so $f< f' < f\le t'+q+1 = s+q+1$.
Clearly, if $t'>t$, $f\le t+q+1 < t'+q+1 = s+q+1$ in this case too.
\YXT{Is following argument true? Only for the first $q+1$ updates will not come with the commit time. After $q+1$ time, the commit time of one update is the start time
of the next update. Then, it's easy to see that $f \leq s + q + 1$. However, may I ask why do we need to prove this?}
\YKC{The fact that $f\le s+q+1$ is used in the final paragraph of this proof. (See the phrase after my ``Agree''.)}

Next, we observe that for each update,
the difference between the SCC time and the start time lies in the range
$[-(q-1),(q-1)]$. For the commit time of one update is the start time
of the next update chosen by the core executing the update.
As there are at most $q$ interfering coordinates, these two times
are at most $q+1$ apart. For an update to receive an SCC time
differing from its start time, either its interval from start to commit
time contains the interval for another update to the same coordinate,
or its interval is contained in another such interval. Either way, the
two start times differ by at most $q-1$.

This immediately implies that if an update has SCC time $t$, it has
start time at least $\max\{1, t -(q-1)\} = \max\{1,t-q+1\}$.

Finally, we prove the bound on the range for interfering updates.

Consider an update $U'$ with start time $t'$ which receives SCC time $t\ge t'$.
It's commit time does not change; it is still at most $t'+q+1$.
Also, $t \le t'+q-1$. Likewise, any update that interferes with $U'$
must still have commit time in the range $[t'+1,t'+q]$, and hence SCC time
in the range  $[t'-q,t'+q-1] \subseteq [t-2q+1,t+q-1]$.

While if $t < t'$, this means that the update $U$ with start time $t$
had a later commit time than $U'$, \YXT{I think this might not be true. Consider the following three updates $[t_1 = 1,f_1 = 7]$, $[t_2 = 2,f_2 = 5]$, $[t_3 = 3,f_3 = 6]$. The update $[t_3 = 3, f_3 = 6]$ will finally have SCC order $2$, but the update with start time $2$ had a earlier commit time than $U'$. I think the right argument is \emph{this means that there exist one update $U''$ with start time $t'' < t'$ had a later commit time than $U'$.}}\YKC{Agree.}
but this commit time was at most
$t+q+1$, and therefore $U'$ had commit time at most $t+q$.
Again, the interfering updates had commit times
in the range $[t'+1,t+q-1]$,
and hence SCC times in the range $[t'-q,t+q-2] \subseteq [t-q+1,t+q-2]$.
}
\end{pfof}
\hide{
\YKC{I suggest the following rewrite of the above proof, which is much shorter and perhaps easier to understand, and fixes the concern raised by Yixin. See if you like it.
The range is slightly weaker, but it doesn't matter.}

\begin{pfof}{Lemma~\ref{lem::SCCrange}}
Recall that by the $q$-bounded assumption, an update starts at time $t$ must have commit time in the range $[t+1,t+q+1]$;
equivalently, any update with commit time $\tau$ must have start time in the range $[\tau-q-1,\tau-1]$.
Also, recall that $\calU_s$ denote the $s$-th update in the SCC order.

First of all, we show the following fact: the commit time of $\calU_s$ lies in the interval $[s+1,s+q+1]$.
Let $\overline{U}$ denote the set of all updates to the coordinate $k_s$.
\begin{itemize}
\item Suppose $\ell$ of the updates in $\overline{U}$ have start time larger than or equal to $s$;
by the $q$-bounded assumption, these $\ell$ updates must have commit times of at least $s+1$.
Now, suppose the commit time of $\calU_s$ is strictly less than $s+1$.
Then by the definition of SCC order, only those updates in $\overline{U}$ with SCC order strictly larger than $s$ can have commit times of at least $s+1$,
so there can only be at most $(\ell-1)$ such updates, a contradiction. \YXT{I don't understand where this $\ell - 1$ comes from and also the following $\gamma - 1$.}
\item Suppose $\gamma$ of the updates in $\overline{U}$ have start time less than or equal to $s$;
by the $q$-bounded assumption, these $\gamma$ updates must have commit times of at most $s+q+1$.
Now, suppose the commit time of $\calU_s$ is strictly larger than $s+q+1$.
Then by the definition of SCC order, only those updates in $\overline{U}$ with SCC order strictly less than $s$ can have commit times of at most $s+q+1$,
so there can only be at most $(\gamma-1)$ such updates, a contradiction.
\end{itemize}

Next, we show that all updates $\calU_s$ with $s\ge t+q+1$ or $s\le t-2q-1$ cannot interfere $\calU_t$. So the updates $\calU_s$ which interfere $\calU_t$
must satisfy $s\in [t-2q,t+q]$; Lemma~\ref{lem::SCCrange} follows.
\begin{itemize}
\item Due to the above fact, all updates $\calU_s$ with $s\ge t+q+1$ must have commit times at least $t+q+2$.
The fact also implies that the commit time of $\calU_t$ is at most $t+q+1$. Therefore, the updates $\calU_s$ with $s\ge t+q+1$ cannot interfere with $\calU_t$.
\item Due to the above fact, all updates $\calU_s$ with $s\le t-2q-1$ must have commit times at most $t-q$.
The fact also implies that the commit time of $\calU_t$ is at least $t+1$, and hence its start time is at least $t-q$ by the $q$-bounded assumption.
Therefore, the updates $\calU_s$ with $s\le t-2q-1$ cannot interfere with $\calU_t$.
\end{itemize}
\end{pfof}
}

\hide{
\RJC{Should the following be stated with an actual lemma?
i.e.\ where is it used? I moved it to right before
Lemma~\ref{lem:discrete-improvement-of-F}.}
Let $[n]$ denote the set of coordinates $\{1,2,\cdots,n\}$:
it is well-known that for any $k\in [n]$, $x\in\rr^n$ and $r\in\rr$,
\begin{equation}\label{eq:upper-sandwich}
f(x+r \ve_k) \leq f(x) + \nabla_k f(x) \cdot r + \frac{L_k}{2} \cdot r^2.
\end{equation}
}

\subsection{The Basic Progress Lemmas, Lemmas~\ref{lem:F-prog-one} and~\ref{lem:F-prog-two}}
\label{subsec::progress lemmas}

We recall two known results.

\begin{lemma}[{Three-Point Property, \cite[Lemma 3.2]{CT1993}}]\label{lm:three-pp}
For any proper, convex and lower semi-continuous function $Y:\rr\ra\rr$ and for any $\dminus\in\rr$, let\\
$\dplus := \argmax_{d\in \rr} \left\{-Y(d) - \G (d-\dminus)^2 \right\}$.
Then for any $\dpi\in\rr$,
$$Y(\dpi) + \G (\dpi-\dminus)^2 ~~\ge~~ Y(\dplus)+ \G (\dplus - \dminus)^2 + \G (\dpi-\dplus)^2.$$
\end{lemma}

\begin{lemma}[{\cite[Lemma 4]{TsengY2009}}]\label{lem:change-of-Dp-vs-change-of-g}
For any $g_1,g_2,x\in\rr$ and $\G\in\rrplus$,\\
$\left| \hd(g_1,x) - \hd(g_2,x) \right| \leq \frac{1}{\G}\cdot \left|g_1 - g_2\right|.$
\end{lemma}

We can now lower bound $\hW_j(g,x)$ in terms of $\hd_j(g,x)$.
\begin{lemma}\label{lem:prog-better-than-quadratic}
For any $g,x\in\rr$ and $\G\in\rrplus$, $\hW_j(g,x) \geq \frac{\G}{2} \left(\hd_j(g,x)\right)^2$.
\end{lemma}
\begin{proof}
We apply Lemma \ref{lm:three-pp} with $\dminus = \dpi = 0$, with $\G$ replaced by $\frac 12 \G$, and $Y(d) = gd - \Psi(x) + \Psi(x+d)$.
Then $W_j(d,g,x) = -Y(d) - \G d^2/2$, and hence $\dplus$, as defined in Lemma \ref{lm:three-pp},  equals $\hd_j(g,x)$.
These yield
$$Y(0) \geq Y(\hd_j(g,x)) + \G\cdot \left(\hd_j(g,x)\right)^2.$$
Since $Y(0) = 0$ and $-Y(\hd_j(g,x)) = \hW_j(g,x) + \frac{\G}{2} \left(\hd_j(g,x)\right)^2$, we are done.
\end{proof}

We are now ready to show
Lemmas \ref{lem:F-prog-one} and \ref{lem:F-prog-two};
they
follow directly from Lemma \ref{lem:discrete-improvement-of-F} below.
We will use the following well-known observation:
for any $1\le k \le n$, $x\in\rr^n$ and $r\in\rr$,
\begin{equation}\label{eq:upper-sandwich}
f(x+r \cdot e_k) \leq f(x) + \nabla_k f(x) \cdot r + \frac{L_k}{2} \cdot r^2,
\end{equation}
where $e_k$ is unit vector along coordinate $k$.

\begin{lemma}\label{lem:discrete-improvement-of-F}
Suppose there is an update to coordinate $j$ at time $t$ according to rule \eqref{eq:update-rule},
and suppose that $\G \ge \Lmax$.
Let $g_j = \nabla_j f(x^t)$ and $\tg_j = \nabla_j f(\tx)$.
Then

\begin{align*}
F(x^t) - F(x^{t+1}) & \geq \frac{\G}{4} (\hdj(\tg_j,x^t_j))^2 - \frac{1}{\G}\cdot (g_j - \tg_j)^2\\
\text{and}~~~~~~~~~~~~~~
F(x^t) - F(x^{t+1}) & \geq \hWj(g_j,x_j^t) - \frac{1}{\G}\cdot (g_j - \tg_j)^2.~~~~~~~~~~~~~~~~~~~~~~
\end{align*}
\end{lemma}

\begin{proof}
To avoid clutter, we use the shorthand $d_j := \hdj(g_j,x_j^t)$ and $\td_j := \hdj(\tg_j,x_j^t)$.
By update rule \eqref{eq:update-rule}, $\td_j = \Delta x_j^t$.
\begin{align*}
F(x^{t+1}) &= f(x^{t+1}) + \Psi_j (x_j^{t+1}) + \sum_{k\neq j} \Psi_k(x_k^{t+1})\\
&~\leq~ f(x^t) + g_j \td_j + \frac{\G}{2} (\td_j)^2 + \Psi_j (x^t_j + \td_j) + \sum_{k\neq j} \Psi_k(x^t_k)\\ & \hspace*{0.75in}\comm{By \eqref{eq:upper-sandwich}, \eqref{eq:update-rule}, and the assumption $\G\ge \Lmax \ge L_j$}\\
&= F(x^t) + \tg_j \td_j  + \frac{\G}{2} (\td_j)^2 - \Psi_j(x^t_j) + \Psi_j (x^t_j + \td_j) + (g_j - \tg_j)\td_j\\
&= F(x^t) - \hWj(\tg_j,x^t_j) + (g_j - \tg_j)\td_j.
\end{align*}

$$\text{Hence,}~~~~F(x^t) - F(x^{t+1}) ~\geq~ \hWj(\tg_j,x^t_j) - (g_j - \tg_j) \td_j.\hspace*{2in}$$

Then we can apply Lemma \ref{lem:prog-better-than-quadratic} to prove the first inequality in Lemma \ref{lem:discrete-improvement-of-F}:
\begin{align*}
 F(x^t) - F(x^{t+1})
&\ge \hW_j(\tg_j,x^t_j)- (g_j - \tg_j) \td_j
 \ge \frac{\G}{2} (\td_j)^2 - |g_j - \tg_j|\cdot |\td_j|\\
& \ge \frac{\G}{2} (\td_j)^2 - \frac 12 \left[ \frac{2}{\G}\cdot (g_j - \tg_j)^2 + \frac {\G}2 (\td_j)^2 \right]\comm{by the AM-GM ineq.}\\
& = \frac{\G}{4} (\td_j)^2 - \frac{1}{\G}\cdot (g_j - \tg_j)^2.
\end{align*}

We prove the second inequality in Lemma \ref{lem:discrete-improvement-of-F} as follows:
\begin{align*}
F(x^t) - F(x^{t+1}) & \geq \hW_j(\tg_j,x^t_j) - (g_j - \tg_j) \td_j
\geq W_j(d_j,\tg_j,x^t_j) - (g_j - \tg_j) \td_j\\
& = W_j(d_j,g_j,x^t_j) + (g_j - \tg_j) d_j - (g_j - \tg_j) \td_j\\
& = \hWj(g_j,x^t_j) + (g_j - \tg_j) (d_j - \td_j)
 \geq \hWj(g_j,x^t_j) - |g_j - \tg_j| \cdot |d_j - \td_j|\\
& \geq \hWj(g_j,x^t_j) - \frac{1}{\G}(g_j - \tg_j)^2. \comm{By Lemma \ref{lem:change-of-Dp-vs-change-of-g}.}
\end{align*}
\hide{
\YKC{As Reviewer 1 pointed out, with the old definition of $\hW$, we forgot sign switch twice, so the final inequality is correct.
By using the new definition of $\hW$, there is no sign switch needed.}
}
\end{proof}

\hide{
\subsection{The Proof of Theorem~\ref{thm:meta-new::re}}\label{app:good-progress}

The following lemma is key to the demonstration of progress in both the strongly convex and convex cases.

For any $t \ge 1$, we define:
$$
\PRG(t) ~:=~ \sum_{k=1}^n ~\hWk(\nabla_k f(x^{t}),x_{k}^{t}).
$$

We will use the following lemma from~\cite[Lemmas 4,6]{richtarik2014iteration}.
The version we present here is slightly different from the one in~\cite{richtarik2014iteration}, but the proofs are essentially the same.

\begin{lemma}[{\cite[Lemmas 4,6]{richtarik2014iteration}}]\label{lem:good-progress}
	~\\
	(a) Suppose that $f,F$ are strongly convex with parameters $\mu_f,\mu_F > 0$ respectively,
	and suppose that $\G \ge \muf$. Then
	\[\PRGe(t) ~\geq~ \frac{\mu_F}{\mu_F + \G - \mu_f}\cdot F(x^t).\]
	(b) Suppose that $f,F$ are convex functions. Suppose that $\calR := \min_{x^*\in X^*} \|\ptone - x^*\| < \infty$. Then
	\[\PRGe(t) ~\geq~ \min\left\{\frac 12~,~\frac{F(x^t)}{2\G \calR^2}\right\}\cdot F(x^t).\]
\end{lemma}
\begin{pfof}{Theorem~\ref{thm:meta-new::re}} %
We begin by showing (i).
By assumption \rjc{(c)} and Lemma \ref{lem:good-progress}\rjc{(a)},
\begin{align*}
 H(t) - H(t+1)
\ge
 \left[\frac \alpha n \cdot \PRG(t) + \frac{\beta}{n} \cdot \Ap(t)~\right]
&\ge
 \left[\frac \alpha n \cdot \frac{\mu_F}{\mu_F + \G - \mu_f} \cdot  F(x^{t})+ \frac{\beta}{n} \cdot \Ap(t) \right]\\
&\ge~\delta \cdot H(t),
\end{align*}
where $\delta :=~ \min\left\{ \frac \alpha n \cdot \frac{\mu_F}{\mu_F + \G - \mu_f} ~,~ \frac{\beta}{n}\right\}$.

Thus $H(t+1) \le \left( 1-\delta  \right) H(t)$ for all $t\ge 1$.
Iterating the above inequality $T$ times yields
$H(T+1) \le \left( 1-\delta  \right)^{T} H(1).$

To finish the proof note that since $\Ap(1) = 0$  and $\Am(1)\ge 0$, $H(1) \le F(x^1)$.

\smallskip

Now we show (ii).
By the second assumption and Lemma \ref{lem:good-progress},
\begin{align*}
H(t) - H(t+1)
&\ge
 \left[ \frac \alpha n \cdot \PRG(t) + \frac{\beta}{n} \cdot \Ap(t)~\right]
\ge
 \left[ \frac \alpha n \cdot \min\left\{\frac 12, \frac{F(x^{t})}{2\G\calR ^2}\right\}\cdot F(x^{t}) + \frac{\beta}{n} \cdot \Ap(t) \right].
\end{align*}

We consider two cases:
\begin{itemize}
	\item If $F(x^{t}) \le \Ap(t)$, then $\Ap(t) \ge \frac{H(t)}{2}$, thus
	\begin{align*}
	\frac \alpha n \cdot \min\left\{\frac 12, \frac{F(x^{t})}{2\calR ^2}\right\}\cdot F(x^{t}) ~+~ \frac{\beta}{n} \cdot \Ap(t)
	&~\ge~ \frac{\beta}{2n}\cdot H(t).
	\end{align*}
	\item If $F(x^{t}) > \Ap(t)$, then $F(x^{t}) > \frac{H(t)}{2}$, thus
\[
\frac \alpha n \cdot \min\left\{\frac 12, \frac{F(x^{t})}{2\G\calR ^2}\right\}\cdot F(x^{t}) + \frac{\beta}{n} \cdot \Ap(t)
~>~ \frac \alpha {2n} \cdot \min\left\{\frac 12, \frac{H(t)}{4\G~\calR ^2}\right\}\cdot H(t).
\]
\end{itemize}
Since $H$ is a decreasing function, $H(t)\le H(1) \le F(x^1)$.
Thus, unconditionally,
\begin{align*}
 \frac \alpha n \cdot \min\left\{\frac 12, \frac{F(x^{t})}{2\G\calR ^2}\right\}\cdot F(x^{t}) + \frac{\beta}{n} \cdot \Ap(t)
&\ge \min \left\{ \frac{\beta}{2n}, \frac{\alpha}{4n}, \frac{\alpha\cdot H(t)}{8n\G \calR ^2} \right\} \cdot H(t)\\
&  \ge \min \left\{ \frac{\beta}{2n\cdot F(x^1)}, \frac{\alpha}{4n\cdot F(x^1)}, \frac{\alpha}{8n\G\calR ^2} \right\} \cdot H(t)^2.
\end{align*}
Note that the term $\min \left\{\frac{\beta}{2n~F(x^1)}, \frac{\alpha}{4n F(x^1)}, \frac{\alpha}{8n\G \calR ^2} \right\}$
is independent of $t$. We denote it by $\varepsilon$.
Thus,
$H(t) - H(t+1) \ge \varepsilon~ H(t)^2$.
Dividing both sides by $H(t) \cdot H(t+1)$ yields
\[
\frac{1}{H(t+1)} - \frac{1}{H(t)} ~\ge~ \varepsilon \frac{H(t)}{H(t+1)} ~\ge~ \varepsilon.
\]
\[
\text{Iterating the above inequality $T$ times yields}~~~~\frac{1}{H(T+1)} - \frac{1}{H(1)} ~\ge~ \varepsilon T,\hspace*{1in}
\]
\[
\text{and hence}~~~~\frac{1}{H(T+1)} ~\ge~ \varepsilon T + \frac{1}{H(1)} ~\ge~ \varepsilon T + \frac{1}{F(x^1)}. \comm{since $H(1) \leq F(x^1)$}\hspace*{1in}
\]
(ii) follows by taking reciprocal on both sides of the above inequality.

It is straightforward to see that taking expectations leaves the proof unchanged.
\end{pfof}

\subsection{Bounding How Much $\hW$ and $\hd$ Vary as a Function of Their Arguments}\label{subsect:Wd}

We first present the proofs of Lemma~\ref{lem:W-shift-re} and \ref{lem:change-of-Dp-vs-change-of-g-gen}.
To prove Lemma~\ref{lem::What-on-two-paths}, we will need two additional lemmas, to be presented below.
We finish with the proof of Lemma~\ref{lem::F-prog-general}.

\begin{pfof}{Lemma~\ref{lem:W-shift-re}} %
\rjc{For notational simplicity, we let $g_1$ and $g_2$ denote $g_j$ and $g'_j$,
 resp.}
\begin{align*}
\hW(g_1,x)
&~= \max_{d\in \rr} W(d,g_1,x)
\ge W(\hd(g_2),g_1,x)\\
&~= -g_1 \cdot \hd(g_2) - \G \cdot \hd(g_2)^2 / 2 + \Psi(x) - \Psi(x + \hd(g_2))\\
&~= -g_2 \cdot \hd(g_2) - \G \cdot \hd(g_2)^2 / 2 + \Psi(x) - \Psi(x + \hd(g_2))\\
&\hspace*{1in}+ (g_2 - g_1) \cdot \big[\hd(g_1)+ (\hd(g_2) - \hd(g_1))\big]\\
&~\ge \hW(g_2,x) - |g_1 - g_2| \cdot \big|\hd(g_1)\big| - |g_1 - g_2| \cdot \big|\hd(g_2) - \hd(g_1)\big|\\
&~\ge \hW(g_2,x) - |g_1 - g_2| \cdot \big|\hd(g_1)\big| - \frac{1}{\G} (g_1 - g_2)^2\comm{By Lemma \ref{lem:change-of-Dp-vs-change-of-g}}\\
&~\ge \hW(g_2,x) - \frac{1}{\G} (g_1 - g_2)^2  - \frac{\G}{4} (\hd(g_1))^2  - \frac{1}{\G} (g_1 - g_2)^2\comm{AM-GM ineq.}\\
&~\ge \hW(g_2,x) - \frac{2}{\G} (g_1 - g_2)^2 - \frac 12 \hW(g_1,x).\comm{By Lemma \ref{lem:prog-better-than-quadratic}}
\end{align*}
\end{pfof}

\rjc{The next four lemmas concern an arbitrary coordinate $x_j$.
To avoid notational clutter, we write $W$, $\hW$, $\hd$, and $\Psi$ in lieu of $W_j$
$\hW_j$, $\hd$, and $\Psi_j$, resp. Also, by $x_1$ and $x_2$ we will mean two
possible values of $x_j$.}

\rjc{
Next, we demonstrate Lemma~\ref{lem:change-of-Dp-vs-change-of-g-gen};
it is a simple corollary of Lemma~\ref{lem:change-of-Dp-vs-change-of-g} and the following lemma.
}

\begin{lemma}\label{lem:change-of-Dp-vs-change-of-x}
For any $g,x_1,x_2\in\rr$, 
$\big| \hd(g,x_1) - \hd(g,x_2) \big| \leq \left|x_1 - x_2\right|.$
\end{lemma}
\begin{proof}
For $i=1,2$, let $d_i := \hd(g,x_i)$.
By the definition of $\hd$, for $i=1,2$,
there exists a subgradient $\Psi'(x_i+d_i)$ such that
$$
g + \G \cdot d_i + \Psi'(x_i+d_i) ~=~ 0.
$$

If $d_1 = d_2$, we are done. If $d_1 > d_2$, then $\Psi'(x_1 + d_1) < \Psi'(x_2 + d_2)$.
Since $\Psi$ is convex, $x_1 + d_1 \le x_2 + d_2$ and hence $0 < d_1 - d_2 \le x_2 - x_1$.

If $d_2 > d_1$, by the same argument as above we have $0 < d_2 - d_1 \le x_1 -x_2$.
\end{proof}

\rjc{Lemma~\ref{lem:change-of-Dp-vs-change-of-g-gen} is a simple corollary of Lemmas~\ref{lem:change-of-Dp-vs-change-of-g} and~\ref{lem:change-of-Dp-vs-change-of-x}.
}

\rjc{
\begin{pfof}{Lemma~\ref{lem:change-of-Dp-vs-change-of-g-gen}}
\begin{align*}
\left(\hd(g_1,x_1) - \hd(g_2,x_2)\right)^2
&= \left(\hd(g_1,x_1) - \hd(g_1,x_2) + \hd(g_1,x_2) - \hd(g_2,x_2)\right)^2\\
&\le  2 \left(\hd(g_1,x_1) - \hd(g_1,x_2)\right)^2
+ 2\left(\hd(g_1,x_2) - \hd(g_2,x_2)\right)^2\\
& \le 2 (x_1 - x_2)^2 + \frac {2}{\Gamma^2} (g_1 - g_2)^2.
\end{align*}
\end{pfof}
}

The next two lemmas will be needed to prove Lemma~\ref{lem::What-on-two-paths}.

\begin{lemma}[$\hW$ Shifting on $x$ parameter]\label{lm:W-x-shift::re}
Let $\hW(g,x_1)  = W (\hhd_1,g,x_1)$ and $\hW(g,x_2) = W(\hhd_2,g,x_2)$. Then
$$
\hW(g,x_1) + \Psi(x_2) - \Psi(x_1)
~~\ge~~ \hW(g,x_2) - g(x_2 - x_1) - \G \hhd_2 (x_2-x_1) - \frac{\G}{2}\cdot (x_2 - x_1)^2.
$$
\end{lemma}
\begin{proof}
We use Lemma~\ref{lm:three-pp} with $\dminus = 0$, \rjc{$\dplus=\hhd_1$},
and $Y(d) = gd - \Psi(x_1) + \Psi(x_1+d)$.
\rjc{We note that $Y(\dpi) + \frac{\G}{2}\cdot  \dpi^2 = - W(\dpi,g,x_1)$.
Also, we observe that $-W(\dpi,g,x_1)$ is strongly convex with strong convexity parameter $\Gamma$. As $W(\dpi,g,x_1)$ is minimized at $\hhd_1$,
we conclude that $-W(\dpi,g,x_1) \ge -W(\hhd_1,g,x_1) + \frac {\Gamma}{2} (\hhd_1 - \dpi)^2 = -\hW(g,x_1)+ \frac {\Gamma}{2} (\hhd_1 - \dpi)^2$.
Thus}
$$
Y(\dpi) + \frac{\G}{2}\cdot (\dpi)^2 ~\ge~ -\hW(g,x_1) + \frac{\G}{2}\cdot  (\dpi-\hhd_1)^2.
$$
The above inequality holds for any $\dpi$. In particular, we pick $\dpi = x_2 - x_1 + \hhd_2$, yielding
\begin{align*}
\hW(g,x_1) &\ge -g(x_2 - x_1 + \hhd_2) + \Psi(x_1) ~-~ \Psi(x_2+\hhd_2)
- \frac{\G}{2}\cdot (x_2 - x_1 + \hhd_2)^2 \\
&~~~~~~~~~~~~~+ \frac{\G}{2}\cdot (x_2 - x_1 + \hhd_2 - \hhd_1)^2.
\end{align*}
By adding $\Psi(x_2) - \Psi(x_1)$ to both sides, we obtain
\begin{align*}
& \hW(g,x_1) + \Psi(x_2) - \Psi(x_1)\\
&~~\ge -g(x_2 - x_1 + \hhd_2) + \Psi(x_2) - \Psi(x_2+\hhd_2) - \frac{\G}{2}\cdot (x_2 - x_1 + \hhd_2)^2 \\
&~~~~~~~~~~~~~+ \frac{\G}{2}\cdot (x_2 - x_1 + \hhd_2 - \hhd_1)^2\\
&~~=~
\hW(g,x_2)- g(x_2 - x_1) - \G \hhd_2 (x_2-x_1) - \frac{\G}{2}\cdot (x_2 - x_1)^2 
+ \frac{\G}{2}\cdot (x_2 - x_1 + \hhd_2 - \hhd_1)^2\\
&~~\ge~
\hW(g,x_2)  - g(x_2 - x_1) - \G \hhd_2 (x_2-x_1) - \frac{\G}{2}\cdot (x_2 - x_1)^2.
\end{align*}
\end{proof}

\begin{lemma}[$\Psi$ Shifting]\label{lm:psi-shift}
Let
$\hW(g_1,x_1)  = W(\hd_1,g_1,x_1)$ and
$\hW(g_2,x_2) =$\\ $W(\hd_2,g_2,x_2)$. Then
$$
\Psi(x_2 + \hd_2) - \Psi(x_1 + \hd_1) ~~\le~~ g_2 (x_1 - x_2 + \hd_1 - \hd_2) + \frac{\G}{2} \cdot (x_1 - x_2 + \hd_1)^2.
$$
\end{lemma}
\begin{proof}
By the definition of $\hd_2$, we have the following inequality, which directly implies the one stated in the lemma.
$$
-g_2 \hd_2 - \frac{\G}{2} \cdot (\hd_2)^2 - \Psi(x_2 + \hd_2) ~~\ge~~ -g_2 (x_1 - x_2 + \hd_1) - \frac{\G}{2} \cdot (x_1 - x_2 + \hd_1)^2 - \Psi(x_1 + \hd_1).
$$
\end{proof}


\begin{proof}{of Lemma~\ref{lem::What-on-two-paths}.}
Suppose the latest update to coordinate $k$ occurred at time $\breve{t}$.
Also suppose that
\begin{itemize}
\item the changes to $x_k$ from $x_k^{t-2q+1}$ to $x_k^{\pi(k),t}$ are $d_{11},d_{12},\cdots,d_{1\ell}$;
\item the changes to $x_k$ from $x_k^{t-2q+1}$ to $x_k^{\pi,t}$ are $d_{21},d_{22},\cdots,d_{2\ell}$.
\end{itemize}
Furthermore, let
$$g^a_k := \nabla_k f(x^{\pi,t})~~~~~~\text{and}~~~~~~\hhd := \argmax_d~W(d,g^a_k,x_k^{\pi,t}).$$
In other words, $x_k^{\pi(k),t} = x_k^{t-2q+1} + \sum_{r=1}^\ell d_{1r}$ and $x_k^{\pi,t} = x_k^{t-2q+1} + \sum_{r=1}^\ell d_{2r}$.

By Lemma~\ref{lm:W-x-shift::re},
\begin{align*}
&
\hW(g^a_k,x_k^{\pi(k),t})
+ \Psi(x_k^{\pi,t}) - \Psi(x_k^{\pi(k),t})
\ge
\hW(g^a_k,x_k^{\pi,t}) - g^a_k\cdot (x_\rjc{k}^{\pi,t} - x_k^{\pi(k),t})\\
&\hspace{2.0in}- \G \hhd \cdot (x_k^{\pi,t} - x_k^{\pi(k),t})
~-~ \frac{\G}{2} \cdot (x_k^{\pi,t} - x_k^{\pi(k),t})^2.
\end{align*}

On the other hand, let $g^b_k$ be the gradient used to compute the update $d_{2\ell}$. By Lemma~\ref{lm:psi-shift},
on setting $x_2 = x_k^{\pi,t} - d_{2\ell}$ and $x_1 = x_k^{\pi(k),t} - d_{1\ell}$,
and noting that $\hd_1 = d_{1\ell}$ and $\hd_2 = d_{2\ell}$, we obtain
$$
\Psi(x_k^{\pi,t}) - \Psi(x_k^{\pi(k),t}) \le g^b_k (x_k^{\pi(k),t} - x_k^{\pi,t}) + \frac{\G}{2} (x_k^{\pi(k),t} - x_k^{\pi,t} + d_{2\ell})^2.
$$

Combining the above two inequalities, and letting $\delta := x_k^{\pi,t} - x_k^{\pi(k),t}$, yields
\begin{align*}
&\hW(g^a_k,x_k^{\pi(k),t}) 
 \ge
\hW(g^a_k,x_k^{\pi,t}) + (g^b_k - g^a_k)\cdot \delta - \G \hhd \cdot \delta - \frac{\G}{2} \cdot \delta^2 - \frac{\G}{2}\cdot (d_{2\ell} - \delta)^2\\
& \hspace*{0.5in}\rjc{\ge  \hW(g^a_k,x_{k}^{\pi,t}) - \frac 1{2\G}(g^b_k - g^a_k)^2  - \frac{\G}{2} \delta^2 - \G (\hhd -\delta) \delta -  \G \delta^2  - \frac{\G}{2} \delta^2 }\\
& \hspace*{2in}\rjc{-  \frac{\G}{2} (d_{2\ell})^2 + \frac{\G}{2} \cdot 2 d_{2\ell} \cdot \delta - \frac{\G}{2} \delta^2} \\
 & \hspace*{0.5in}\rjc{= \hW(g^a_k,x_k^{\pi,t}) - \frac 1{2\G}(g^b_k - g^a_k)^2 - \frac{3}{2} \G  \delta^2 -  \frac{\G}{2} (d_{2\ell})^2 - \G (\hhd -\delta) \delta + \G (d_{2\ell} -\delta) \delta} \\
& \hspace*{0.5in}\ge  \hW(g^a_k,x_{\rjc{k}}^{\pi,t}) - \frac 1{2\G} (g^b_k - g^a_k)^2 - \frac{3}{2} \G \delta^2 - \frac{\G}{2} (d_{2\ell})^2 - \G |\hhd - d_{2\ell}|\cdot |\delta|\\
& \hspace*{0.5in}\ge
\hW(g^a_k,x_k^{\pi,t}) - \frac 1{2\G} (g^b_k - g^a_k)^2 - 2 \G \delta^2
- \frac{\G}{2} (d_{2\ell})^2 - \frac\G 2 (\hhd - d_{2\ell})^2.
\end{align*}
\begin{align}
\nonumber
&\text{By Lemma~\ref{lem:change-of-Dp-vs-change-of-g-gen},}~~~~
\frac\G 2 \cdot (\hhd - d_{2\ell})^2 \le \G\cdot  (d_{2\ell})^2 + \frac 1 \G \cdot (g_k^a - g_k^b)^2.
\\
\label{eq:key-shift}
&\text{Thus,}~~~~ \hW(g^a_k,x_k^{\pi(k),t}) \ge \hW(g^a_k,x_k^{\pi,\rjc{t}})
- \frac 3{2\G} \cdot (g^b_k - g^a_k)^2 - 2 \G \delta^2- \frac{3\G}{2} \cdot (d_{2\ell})^2.
\end{align}
\end{proof}

\begin{pfof}{Lemma~\ref{lem:delta-bound}} %
\YKC{The lemma was removed from the main text. Is this proof no longer needed?}To see the first inequality,
note that
\begin{align*}
x_k^{\pi,t}, x_k^{\pi,t-2q} &~\in~ \bigg[ \sum_{1\le i\le \ell}\Dmint x_k^{\pi,t_i}~~,~~
\sum_{1\le i\le \ell}\Dmaxt x_k^{\pi,t_i} \bigg],\\
x_k^{\pi\rjc{(k)},t}, x_k^{\pi\rjc{(k)},t-2q} &~\in~ \bigg[ \sum_{1\le i\le \ell}\Dmint x_k^{\pi(k),t_i}~~,~~
\sum_{1\le i\le \ell}\Dmaxt x_k^{\pi(k),t_i} \bigg],
\end{align*}
and
the above two intervals overlap, as the synchronous update $\Delta^t x_k^{\mathcal{S}, \pi, t_i} = \Delta^t x_k^{\mathcal{S}, \pi(k), t_i}$
lies in both intervals.
\rjc{
Thus
\begin{align*}
\bigg(x_k^{\pi,t} - x_k^{\pi(k),t} \bigg)^2 \le
\bigg[\sum_{1\le i \le \ell}
\Dspan{\pi,t_i}{k}{t, \emptyset}
+
\Dspan{\pi(k),t_i}{k}{t, \emptyset}
\bigg]^2.
\end{align*}
}

The second inequality follows by applying the Cauchy-Schwartz inequality.
\end{pfof}

\begin{pfof}{Lemma~\ref{lem::F-prog-general}}
Recall that we write $\pi(k,t)$ to denote the path in which coordinate $\kt$ at time $t$ is replaced by coordinate $k$,
and to reduce clutter we abbreviate this as $\pi(k)$.
Recall also that we let
$\prev(t,k)$ denote the time of the most recent update to coordinate $k$, if any,
in the time range $[t-2q,t-1]$; otherwise, we set it to $t$.
From \eqref{eqn:F-prog},
\begin{align*}
& \expectpi{F(x^{t}) -F(x^{t+1})}\\
& ~~\ge \frac 1{2n} \mathbb{E}_{\pi}\bigg[\sum_{k_t=1}^n \hWkt(\gktpiktt),x^{\pi(k_t), t}_{k_t})\bigg]  +
\frac \G 8\left(\DE_t\right)^2 - \frac 1\G \expectpi{\left(\gktpit - \tgktpit\right)^2}\\
&~~= \frac 1{2n}  \mathbb{E}_{\pi}\bigg[\sum_{k=1}^n \hWk(\gkpikt, x^{\pi(k), t}_k)\bigg] +
\frac \G 8\left(\DE_t\right)^2 - \frac 1\G \expectpi{\left(\gktpit - \tgktpit\right)^2}\\
& ~~\ge \frac 1{2n}  \mathbb{E}_{\pi}\bigg[\sum_{k=1}^n \frac 1n \sum_{k_t=1}^n \left(
\frac 23 \hWk(\gkpiktt,x^{\pi(k), t}_k) - \frac 4{3\G}\left(\gkpikt - \gkpiktt\right)^2 \right) \bigg] \\
&~~~~~~~~+
\frac \G 8\left(\DE_t\right)^2 - \frac 1\G \mathbb{E}_{\pi}\bigg[\left(\gktpit - \tgktpit\right)^2\bigg]
~~~~\text{(by Lemma~\ref{lem:W-shift-re})}.
\end{align*}
\rjc{
For the next bound, we will be applying Lemma~\ref{lem::What-on-two-paths}
to shift the $x^{\pi(k), t}_k$ parameter in $\hWk$ to $x_k^{\pi,t} = x_k^{\pi(k_t),t}$.
Note that applying this lemma introduces additional terms (the case $l>0$) only if $x_k$ is updated at some time $s\in[t-2q,t-1]$; this means that $k=k_s$ where $t-2q \le s < t$, and to avoid double counting the effect of updates to the same coordinate, we can further limit $s$ to $s = \prev(t,k)$, or equivalently that $t-2q \le s < t$ and $s = \prev(t,k_s)$. This yields
the claimed result.}
\end{pfof}
}

%% file: app-exptd-prog.tex
\subsection{The Expected Progress, Lemma~\ref{lem::F-prog-general}}
\label{app-sec::exptd-prog}

\begin{pfof}{Lemma~\ref{lem::F-prog-general}}
Recall that we write $\pi(k,t)$ to denote the path in which coordinate $\kt$ at time $t$ is replaced by coordinate $k$,
and to reduce clutter we abbreviate this as $\pi(k)$.
Recall also that we let
$\prev(t,k)$ denote the time of the most recent update to coordinate $k$, if any,
in the time range $[t-2q,t-1]$; otherwise, we set it to $t$.
From \eqref{eqn:F-prog},

\begin{align*}
& \expectpi{F(x^{t}) -F(x^{t+1})}\\
& ~~\ge \frac 1{2n} \mathbb{E}_{\pi}\bigg[\sum_{k_t=1}^n \hWkt(\gktpiktt),x^{\pi(k_t), t}_{k_t})\bigg]  +
\frac \G 8\left(\DE_t\right)^2 - \frac 1\G \expectpi{\left(\gktpit - \tgktpit\right)^2}\\
&~~= \frac 1{2n}  \mathbb{E}_{\pi}\bigg[\sum_{k=1}^n \hWk(\gkpikt, x^{\pi(k), t}_k)\bigg] +
\frac \G 8\left(\DE_t\right)^2 - \frac 1\G \expectpi{\left(\gktpit - \tgktpit\right)^2}\\
& ~~\ge \frac 1{2n}  \mathbb{E}_{\pi}\bigg[\sum_{k=1}^n \frac 1n \sum_{k_t=1}^n \left(
\frac 23 \hWk(\gkpiktt,x^{\pi(k), t}_k) - \frac 4{3\G}\left(\gkpikt - \gkpiktt\right)^2 \right) \bigg] \\
&~~~~~~~~+
\frac \G 8\left(\DE_t\right)^2 - \frac 1\G \mathbb{E}_{\pi}\bigg[\left(\gktpit - \tgktpit\right)^2\bigg]
~~~~\text{(by Lemma~\ref{lem:W-shift-re})}.
\end{align*}

For the next bound, we will be applying Lemma~\ref{lem::What-on-two-paths}
to shift the $x^{\pi(k), t}_k$ parameter in $\hWk$ to $x_k^{\pi,t} = x_k^{\pi(k_t),t}$.
Note that applying this lemma introduces additional terms (the case $l>0$) only if $x_k$ is updated at some time $s\in[t-2q,t-1]$; this means that $k=k_s$ where $t-2q \le s < t$, and to avoid double counting the effect of updates to the same coordinate, we can further limit $s$ to $s = \prev(t,k)$, or equivalently that $t-2q \le s < t$ and $s = \prev(t,k_s)$. This yields
the claimed result.
\end{pfof}

%% file: app-hw-and-dhat.tex
 \subsection{Bounding How Much $\hW$ and $\hd$ Vary as a Function of Their Arguments}\label{subsect:Wd}

The next five lemmas concern an arbitrary coordinate $x_j$.
To avoid notational clutter, we write $W$, $\hW$, $\hd$, and $\Psi$ in lieu of $W_j$
$\hW_j$, $\hd$, and $\Psi_j$, resp. Also, by $x_1$ and $x_2$ we will mean two
possible values of $x_j$, and by $g_1$ and $g_2$ two possible values of $g_j$.

We first present the proofs of Lemma~\ref{lem:W-shift-re} and \ref{lem:change-of-Dp-vs-change-of-g-gen}.
To prove Lemma~\ref{lem::What-on-two-paths}, we will need two additional lemmas, to be presented below.

\begin{pfof}{Lemma~\ref{lem:W-shift-re}} %
\begin{align*}
\hW(g_1,x)
&~= \max_{d\in \rr} W(d,g_1,x)
\ge W(\hd(g_2),g_1,x)\\
&~= -g_1 \cdot \hd(g_2) - \G \cdot \hd(g_2)^2 / 2 + \Psi(x) - \Psi(x + \hd(g_2))\\
&~= -g_2 \cdot \hd(g_2) - \G \cdot \hd(g_2)^2 / 2 + \Psi(x) - \Psi(x + \hd(g_2))\\
&\hspace*{1in}+ (g_2 - g_1) \cdot \big[\hd(g_1)+ (\hd(g_2) - \hd(g_1))\big]\\
&~\ge \hW(g_2,x) - |g_1 - g_2| \cdot \big|\hd(g_1)\big| - |g_1 - g_2| \cdot \big|\hd(g_2) - \hd(g_1)\big|\\
&~\ge \hW(g_2,x) - |g_1 - g_2| \cdot \big|\hd(g_1)\big| - \frac{1}{\G} (g_1 - g_2)^2\comm{By Lemma \ref{lem:change-of-Dp-vs-change-of-g}}\\
&~\ge \hW(g_2,x) - \frac{1}{\G} (g_1 - g_2)^2  - \frac{\G}{4} (\hd(g_1))^2  - \frac{1}{\G} (g_1 - g_2)^2\comm{AM-GM ineq.}\\
&~\ge \hW(g_2,x) - \frac{2}{\G} (g_1 - g_2)^2 - \frac 1{2} \hW(g_1,x).\comm{By Lemma \ref{lem:prog-better-than-quadratic}}
\end{align*}

\end{pfof}

Next, we demonstrate Lemma~\ref{lem:change-of-Dp-vs-change-of-g-gen};
it is a simple corollary of Lemma~\ref{lem:change-of-Dp-vs-change-of-g} and the following lemma.

\begin{lemma}\label{lem:change-of-Dp-vs-change-of-x}
For any $g,x_1,x_2\in\rr$, 
$\big| \hd(g,x_1) - \hd(g,x_2) \big| \leq \left|x_1 - x_2\right|.$
\end{lemma}
\begin{proof}
For $i=1,2$, let $d_i := \hd(g,x_i)$.
By the definition of $\hd$, for $i=1,2$,
there exists a subgradient $\Psi'(x_i+d_i)$ such that
$$
g + \G \cdot d_i + \Psi'(x_i+d_i) ~=~ 0.
$$

If $d_1 = d_2$, we are done. If $d_1 > d_2$, then $\Psi'(x_1 + d_1) < \Psi'(x_2 + d_2)$.
Since $\Psi$ is convex, $x_1 + d_1 \le x_2 + d_2$ and hence $0 < d_1 - d_2 \le x_2 - x_1$.

If $d_2 > d_1$, by the same argument as above we have $0 < d_2 - d_1 \le x_1 -x_2$.
\end{proof}

Lemma~\ref{lem:change-of-Dp-vs-change-of-g-gen} is a simple corollary of Lemmas~\ref{lem:change-of-Dp-vs-change-of-g} and~\ref{lem:change-of-Dp-vs-change-of-x}.

\begin{pfof}{Lemma~\ref{lem:change-of-Dp-vs-change-of-g-gen}}
\begin{align*}
\Big(\hd(g_1,x_1) - \hd(g_2,x_2)\Big)^2
&= \Big(\hd(g_1,x_1) - \hd(g_1,x_2) + \hd(g_1,x_2) - \hd(g_2,x_2)\Big)^2\\
&\le  2
{ \left(\hd(g_1,x_1) - \hd(g_1,x_2)\right)^2
+ 2\left(\hd(g_1,x_2) - \hd(g_2,x_2)\right)^2}\\
& {\le  2\left(x_1 - x_2\right)^2 + \frac {2}{\Gamma^2} \left(g_1 - g_2\right)^2.}
\end{align*}
\end{pfof}

The next two lemmas will be needed to prove Lemma~\ref{lem::What-on-two-paths}.

\begin{lemma}[$\hW$ Shifting on $x$ parameter]\label{lm:W-x-shift::re}
Let $\hW(g,x_1)  = W (\hhd_1,g,x_1)$ and $\hW(g,x_2) = W(\hhd_2,g,x_2)$. Then
$$
\hW(g,x_1) + \Psi(x_2) - \Psi(x_1)
~~\ge~~ \hW(g,x_2) - g(x_2 - x_1) - \G \hhd_2 (x_2-x_1) - \frac{\G}{2}\cdot (x_2 - x_1)^2.
$$
\end{lemma}
\begin{proof}
We use Lemma~\ref{lm:three-pp} with $\dminus = 0$, $\dplus=\hhd_1$,
and $Y(d) = gd - \Psi(x_1) + \Psi(x_1+d)$.
We note that $Y(\dpi) + \frac{\G}{2}\cdot  \dpi^2 = - W(\dpi,g,x_1)$.
Also, we observe that $-W(\dpi,g,x_1)$ is strongly convex with strong convexity parameter $\Gamma$. As $W(\dpi,g,x_1)$ is maximized at $\hhd_1$,
we conclude that $-W(\dpi,g,x_1) \ge -W(\hhd_1,g,x_1) + \frac {\Gamma}{2} (\hhd_1 - \dpi)^2 = -\hW(g,x_1)+ \frac {\Gamma}{2} (\hhd_1 - \dpi)^2$.
Thus
$$
Y(\dpi) + \frac{\G}{2}\cdot (\dpi)^2 ~\ge~ -\hW(g,x_1) + \frac{\G}{2}\cdot  (\dpi-\hhd_1)^2.
$$
The above inequality holds for any $\dpi$. In particular, we pick $\dpi = x_2 - x_1 + \hhd_2$, yielding
\begin{align*}
\hW(g,x_1) &\ge -g(x_2 - x_1 + \hhd_2) + \Psi(x_1) ~-~ \Psi(x_2+\hhd_2)
- \frac{\G}{2}\cdot (x_2 - x_1 + \hhd_2)^2 \\
&~~~~~~~~~~~~~+ \frac{\G}{2}\cdot (x_2 - x_1 + \hhd_2 - \hhd_1)^2.
\end{align*}
By adding $\Psi(x_2) - \Psi(x_1)$ to both sides, we obtain
\begin{align*}
& \hW(g,x_1) + \Psi(x_2) - \Psi(x_1)\\
&~~\ge -g(x_2 - x_1 + \hhd_2) + \Psi(x_2) - \Psi(x_2+\hhd_2) - \frac{\G}{2}\cdot (x_2 - x_1 + \hhd_2)^2 \\
&~~~~~~~~~~~~~+ \frac{\G}{2}\cdot (x_2 - x_1 + \hhd_2 - \hhd_1)^2\\
&~~=~
\hW(g,x_2)- g(x_2 - x_1) - \G \hhd_2 (x_2-x_1) - \frac{\G}{2}\cdot (x_2 - x_1)^2 
+ \frac{\G}{2}\cdot (x_2 - x_1 + \hhd_2 - \hhd_1)^2\\
&~~\ge~
\hW(g,x_2)  - g(x_2 - x_1) - \G \hhd_2 (x_2-x_1) - \frac{\G}{2}\cdot (x_2 - x_1)^2.
\end{align*}
\end{proof}

\begin{lemma}[$\Psi$ Shifting]\label{lm:psi-shift}
Let
$\hW(g_1,x_1)  = W(\hd_1,g_1,x_1)$ and
$\hW(g_2,x_2) = W(\hd_2,g_2,x_2)$. Then
$$
\Psi(x_2 + \hd_2) - \Psi(x_1 + \hd_1) ~~\le~~ g_2 (x_1 - x_2 + \hd_1 - \hd_2) + \frac{\G}{2} \cdot (x_1 - x_2 + \hd_1)^2.
$$
\end{lemma}
\begin{proof}
By the definition of $\hd_2$, we have the following inequality, which directly implies the one stated in the lemma.
$$
-g_2 \hd_2 - \frac{\G}{2} \cdot (\hd_2)^2 - \Psi(x_2 + \hd_2) ~~\ge~~ -g_2 (x_1 - x_2 + \hd_1) - \frac{\G}{2} \cdot (x_1 - x_2 + \hd_1)^2 - \Psi(x_1 + \hd_1).
$$
\end{proof}


\begin{proof}{of Lemma~\ref{lem::What-on-two-paths}.}
Suppose the latest update to coordinate $k$ occurred at time $\breve{t}$.
Also suppose that
\begin{itemize}
\item the changes to $x_k$ from $x_k^{t-2q}$ to $x_k^{\pi(k),t}$ are $d_{11},d_{12},\cdots,d_{1\ell}$;
\item the changes to $x_k$ from $x_k^{t-2q}$ to $x_k^{\pi,t}$ are $d_{21},d_{22},\cdots,d_{2\ell}$.
\end{itemize}
Furthermore, let
$$g^a_k := \nabla_k f(x^{\pi,t})~~~~~~\text{and}~~~~~~\hhd := \argmax_d~W(d,g^a_k,x_k^{\pi,t}).$$
In other words, $x_k^{\pi(k),t} = x_k^{t-2q} + \sum_{r=1}^\ell d_{1r}$ and $x_k^{\pi,t} = x_k^{t-2q} + \sum_{r=1}^\ell d_{2r}$.

By Lemma~\ref{lm:W-x-shift::re},
\begin{align*}
&
\hW(g^a_k,x_k^{\pi(k),t})
+ \Psi(x_k^{\pi,t}) - \Psi(x_k^{\pi(k),t})
\ge
\hW(g^a_k,x_k^{\pi,t}) - g^a_k\cdot (x_k^{\pi,t} - x_k^{\pi(k),t})\\
&\hspace{2.0in}- \G \hhd \cdot (x_k^{\pi,t} - x_k^{\pi(k),t})
~-~ \frac{\G}{2} \cdot (x_k^{\pi,t} - x_k^{\pi(k),t})^2.
\end{align*}

On the other hand, let $g^b_k$ be the gradient used to compute the update $d_{2\ell}$. By Lemma~\ref{lm:psi-shift},
on setting $x_2 = x_k^{\pi,t} - d_{2\ell}$ and $x_1 = x_k^{\pi(k),t} - d_{1\ell}$,
and noting that $\hd_1 = d_{1\ell}$ and $\hd_2 = d_{2\ell}$, we obtain
$$
\Psi(x_k^{\pi,t}) - \Psi(x_k^{\pi(k),t}) \le g^b_k (x_k^{\pi(k),t} - x_k^{\pi,t}) + \frac{\G}{2} (x_k^{\pi(k),t} - x_k^{\pi,t} + d_{2\ell})^2.
$$

Combining the above two inequalities, and letting $\delta := x_k^{\pi,t} - x_k^{\pi(k),t}$, yields
\begin{align*}
&\hW(g^a_k,x_k^{\pi(k),t}) 
 \ge
\hW(g^a_k,x_k^{\pi,t}) + (g^b_k - g^a_k)\cdot \delta - \G \hhd \cdot \delta - \frac{\G}{2} \cdot \delta^2 - \frac{\G}{2}\cdot (d_{2\ell} - \delta)^2\\
& \hspace*{0.5in}\ge  \hW(g^a_k,x_{k}^{\pi,t}) - \frac 1{2\G}(g^b_k - g^a_k)^2  - \frac{\G}{2} \delta^2 - \G (\hhd -\delta) \delta -  \G \delta^2  - \frac{\G}{2} \delta^2 \\
& \hspace*{2in}-  \frac{\G}{2} (d_{2\ell})^2 + \frac{\G}{2} \cdot 2 d_{2\ell} \cdot \delta - \frac{\G}{2} \delta^2 \\
 & \hspace*{0.5in}= \hW(g^a_k,x_k^{\pi,t}) - \frac 1{2\G}(g^b_k - g^a_k)^2 - \frac{3}{2} \G  \delta^2 -  \frac{\G}{2} (d_{2\ell})^2 - \G (\hhd -\delta) \delta + \G (d_{2\ell} -\delta) \delta \\
& \hspace*{0.5in}\ge  \hW(g^a_k,x_{k}^{\pi,t}) - \frac 1{2\G} (g^b_k - g^a_k)^2 - \frac{3}{2} \G \delta^2 - \frac{\G}{2} (d_{2\ell})^2 - \G |\hhd - d_{2\ell}|\cdot |\delta|\\
& \hspace*{0.5in}\ge
\hW(g^a_k,x_k^{\pi,t}) - \frac 1{2\G} (g^b_k - g^a_k)^2 - 2 \G \delta^2
- \frac{\G}{2} (d_{2\ell})^2 - \frac\G 2 (\hhd - d_{2\ell})^2.
\end{align*}
\begin{align}
\nonumber
&\text{By Lemma~\ref{lem:change-of-Dp-vs-change-of-g-gen},}
~~~~
\frac\G 2 \cdot (\hhd - d_{2\ell})^2 \le \G\cdot  (d_{2\ell})^2 + \frac 1 \G \cdot (g_k^a - g_k^b)^2.
\\
\label{eq:key-shift}
&\text{Thus,}~~~~ \hW(g^a_k,x_k^{\pi(k),t}) \ge \hW(g^a_k,x_k^{\pi,t})
- \frac 3{2\G} \cdot (g^b_k - g^a_k)^2 - 2 \G \delta^2- \frac{3\G}{2} \cdot (d_{2\ell})^2.
\end{align}
\end{proof}

\hide{
\begin{pfof}{Lemma~\ref{lem:delta-bound}} %
\YKC{The lemma was removed from the main text. Is this proof no longer needed?}To see the first inequality,
note that
\begin{align*}
x_k^{\pi,t}, x_k^{\pi,t-2q} &~\in~ \bigg[ \sum_{1\le i\le \ell}\Dmint x_k^{\pi,t_i}~~,~~
\sum_{1\le i\le \ell}\Dmaxt x_k^{\pi,t_i} \bigg],\\
x_k^{\pi\rjc{(k)},t}, x_k^{\pi\rjc{(k)},t-2q} &~\in~ \bigg[ \sum_{1\le i\le \ell}\Dmint x_k^{\pi(k),t_i}~~,~~
\sum_{1\le i\le \ell}\Dmaxt x_k^{\pi(k),t_i} \bigg],
\end{align*}
and
the above two intervals overlap, as the synchronous update $\Delta^t x_k^{\mathcal{S}, \pi, t_i} = \Delta^t x_k^{\mathcal{S}, \pi(k), t_i}$
lies in both intervals.
\rjc{
Thus
\begin{align*}
\bigg(x_k^{\pi,t} - x_k^{\pi(k),t} \bigg)^2 \le
\bigg[\sum_{1\le i \le \ell}
\Dspan{\pi,t_i}{k}{t, \emptyset}
+
\Dspan{\pi(k),t_i}{k}{t, \emptyset}
\bigg]^2.
\end{align*}
}

The second inequality follows by applying the Cauchy-Schwartz inequality.
\end{pfof}
}

\hide{
\begin{pfof}{Lemma~\ref{lem::F-prog-general}}
Recall that we write $\pi(k,t)$ to denote the path in which coordinate $\kt$ at time $t$ is replaced by coordinate $k$,
and to reduce clutter we abbreviate this as $\pi(k)$.
Recall also that we let
$\prev(t,k)$ denote the time of the most recent update to coordinate $k$, if any,
in the time range $[t-2q,t-1]$; otherwise, we set it to $t$.
From \eqref{eqn:F-prog},
\begin{align*}
& \expectpi{F(x^{t}) -F(x^{t+1})}\\
& ~~\ge \frac 1{2n} \mathbb{E}_{\pi}\bigg[\sum_{k_t=1}^n \hWkt(\gktpiktt),x^{\pi(k_t), t}_{k_t})\bigg]  +
\frac \G 8\left(\DE_t\right)^2 - \frac 1\G \expectpi{\left(\gktpit - \tgktpit\right)^2}\\
&~~= \frac 1{2n}  \mathbb{E}_{\pi}\bigg[\sum_{k=1}^n \hWk(\gkpikt, x^{\pi(k), t}_k)\bigg] +
\frac \G 8\left(\DE_t\right)^2 - \frac 1\G \expectpi{\left(\gktpit - \tgktpit\right)^2}\\
& ~~\ge \frac 1{2n}  \mathbb{E}_{\pi}\bigg[\sum_{k=1}^n \frac 1n \sum_{k_t=1}^n \left(
\frac 23 \hWk(\gkpiktt,x^{\pi(k), t}_k) - \frac 4{3\G}\left(\gkpikt - \gkpiktt\right)^2 \right) \bigg] \\
&~~~~~~~~+
\frac \G 8\left(\DE_t\right)^2 - \frac 1\G \mathbb{E}_{\pi}\bigg[\left(\gktpit - \tgktpit\right)^2\bigg]
~~~~\text{(by Lemma~\ref{lem:W-shift-re})}.
\end{align*}
\rjc{
For the next bound, we will be applying Lemma~\ref{lem::What-on-two-paths}
to shift the $x^{\pi(k), t}_k$ parameter in $\hWk$ to $x_k^{\pi,t} = x_k^{\pi(k_t),t}$.
Note that applying this lemma introduces additional terms (the case $l>0$) only if $x_k$ is updated at some time $s\in[t-2q,t-1]$; this means that $k=k_s$ where $t-2q \le s < t$, and to avoid double counting the effect of updates to the same coordinate, we can further limit $s$ to $s = \prev(t,k)$, or equivalently that $t-2q \le s < t$ and $s = \prev(t,k_s)$. This yields
the claimed result.}
\end{pfof}
}

%% file: app-SACD-analysis-gen-new-temp.tex
\subsection{The Recursive Analysis yielding a proof of  Lemma~\ref{lem::grad::sync::bound::gen}} 
\label{app::rec-bound-gen}

Recall that the definition of $\Delta_{\max}^{u,\emptyset} x_{k_s}^{\pi, s}$
(see Section~\ref{sec::addnl-notation}) assumes the first $u-4q$ updates are already fixed.

To prove Lemma~\ref{lem::grad::sync::bound::gen}, 
we will make use of a recursive bound on

\vspace*{-0.1in}

\[\bigg(\sum_{l_0 \in [u - q, t + q] \setminus \{u\}}  L_{k_{l_0}, k_u}\left( \Dspan{\pi, l_0}{k_{l_0}}{t, \emptyset}\right)\bigg)^2,\]
a bound that expresses the effect of performing one level of recursion,
along with additional computation --- suitable overestimates, and the calculation of some expectations. This bound
is presented in the next lemma.
It is applied repeatedly in order to prove Lemma~\ref{lem::grad::sync::bound::gen}.

Both lemmas make use of two techniques. It will be helpful to explain them upfront.

\smallskip

\paragraph{Unbalanced Cauchy-Schwartz Inequality}
To bound an expression of the form\\ $\left(\sum_{1\le j\le q'} z_j\right)^2$,
rather than have the bound $q'\sum_{1\le j\le q'}z_j^2$, we would like to multiply
a few of the $z_j$ by just a constant. We achieve this with two applications of the usual Cauchy-Schwartz Inequality, and we illustrate this for the case that ``a few'' means two of the $z_j$.
\begin{align*}
\Big(\sum_{1\le j\le q'} z_j\Big)^2
& \le 2(z_k+z_l)^2 + 2\Big(\sum_{\substack{1\le j\le q'\\j\ne k,l}} z_j\Big)^2
\le 4(z_k^2 + z_l^2) + 2(q'-2) \sum_{\substack{1\le j\le q'\\j\ne k,l}} z_j^2\\
&\le 4(z_k^2 + z_l^2) + 2q' \sum_{\substack{1\le j\le q'\\j\ne k,l}} z_j^2
~~~~\text{(this is just a convenient over-estimate).}
\end{align*}

\smallskip

\paragraph{Recentering}
Suppose $a\le b, c \le d$.
Then $b^2 \le [c + (d-a)]^2 \le 2c^2 + 2(d-a)^2$.
we use this bound when $b$ denotes a value (of the form
$\Delta x_k$) which varies
over the different paths across which we want to average,
while the values $c$ and $d-a$ are unvarying.
This is a technique we already used in Section~\ref{sec::span-bound}.

\smallskip

Lemma~\ref{basic::recursion::gen} is a bound on a sum of $\Delta_{\mathsf{var}}$ terms.
However, we will start by obtaining a recursive bound on an analogous sum
of $\Delta_{\mathsf{span}}$ terms.
To this end, we define
\begin{align*}
\mathcal{V}_m \triangleq~ &\mathbb{E}\bigg[ \sum_{l_0 \in [u - q, t + q] \setminus \{u\}} \Big( \sum_{\substack{l_1, l_2, \cdots, l_{m} \in [u - q, t+q], \\ \text{all } l_s \text{ distinct, all } l_s  \ne u, l_0; \\ \text{all }l_s \leq t_{s-1} + q, \text{ where }\\  t_{s-1} = \min\{t, l_0, l_1, \cdots, l_{s-1}\}; \\  S_{m} \subseteq \{ l_s | k_{l_s} = k_{l_{s-1}} \}.}}
\Big(\prod_{\substack{l_s \in R_{m} \setminus S_{m}\\ \text{where } R_{m} \\ = \{ l_1, l_2, \cdots, l_{m}\}.}} \frac{L^2_{k_{l_s}, k_{l_{s-1}}}}{\Gamma^2}\Big)  \\
&\hspace*{1.6in}\cdot L^2_{k_{l_0}, k_{u}} \left( \Dspan{\pi, l_{m}}{k_{l_{m}}}{t_{m}, R_{m} \setminus \{l_{m}\}} \right)^2  \Big) \bigg]. \numberthis \label{def::V::m}
\end{align*}
We note that the second summation is over all choices
of $l_s$, $1\le s \le m$, and of $S_{m}$ which satisfy the stated conditions.
This comment also applies to similar subsequent summations.

Note that $\mathcal{V}_{3q} = 0$ and $\mathcal{V}_{0} = \mathbb{E}\bigg[  \sum_{l_0 \in [u - q, t + q] \setminus \{u\}}  L_{k_{l_0}, k_u}^2\left( \Dspan{\pi, l_0}{k_{l_0}}{\min\{t, l_0\}, \emptyset} \right)^2   \bigg]$.

\vspace*{-0.1in}

\begin{align}
\label{eqn::lm-bound}
\text{Also,}~~l_{m-1} \le \min\{l_{m-1}, t_{m-2} + q\} \leq t_{m-1} + q.\hspace*{1.55in}
\end{align}

Lemma~\ref{basic::recursion::gen} below gives our recursive bound;
it is then used to prove Lemmas~\ref{lem::bound-on-calV-0} and~\ref{lem::grad::sync::bound::gen}.
Finally, we prove Lemma~\ref{basic::recursion::gen} via Lemma~\ref{eqn::intermed-bound-for-rec}, which is stated right before
the proof of Lemma~\ref{basic::recursion::gen}.

\begin{lemma}\label{basic::recursion::gen}
For $t-2q \le u \le t$,
\begin{align}
\label{eqn::lem-rec-gen-term1}
&\mathcal{V}_{m-1} \le 40q \mathcal{V}_{m}
+\mathbb{E}\bigg[\sum_{s \in [t - 7q, t+q] \setminus \{u\}} 120q \Gamma^2 \frac{\left( \Lam^2 \right)^{m+1}
 (4q)^{m}}{n^{m+1}}
\left(\left(\Dbarspan{\pi,s}{k_s}\right)^2 + \left(\Delta x_{k_s}^{\pi, s} \right)^2\right) \bigg]\\
\label{eqn::lem-rec-gen-term2}
&\hspace*{1.5in}+\mathbb{E}\bigg[24 L_{\max}^2 \frac{\left( \Lam^2 \right)^m
(4q)^{m}}{n^{m}} \left(\left( \Dbarspan{\pi, u}{k_u} \right)^2 +  \left(\Delta x^{\pi, u}_{k_u} \right)^2\right)\bigg].
\end{align}
\end{lemma}

\vspace*{-0.1in}

Lemma~\ref{basic::recursion::gen} readily yields a bound on $\mathcal{V}_0$.
Recall that $r  = \frac{160 q^2 \Lam^2}{n}$.
\begin{lemma}
\label{lem::bound-on-calV-0}
\begin{align*}
\mathcal{V}_0 \le \frac{3r^2 \Gamma^2} {160(1-r)q^2}
\sum_{\substack{s \in [t - 7q,
 t+q]\\ \hspace*{0.2in} \setminus \{u\}}}
\left[ \big(\Df_s \big)^2 + \big( \DE_s \big)^2 \right]
+ \frac{3r \Gamma^2}{5q(1-r)} \cdot \left[\left(\Df_u \right)^2 + \left( \DE_u \right)^2 \right].
\end{align*}
\end{lemma}
\begin{proof}
We will apply Lemma~\ref{basic::recursion::gen} recursively to $\mathcal{V}_0$.
Note that as $m$ increases by 1, each sum is multiplied by $\frac{\Lambda^2 \cdot (4q)} {n}$, and the first term has a multiplier of $40 q$ (which is
why we chose $r = 40 q\cdot \frac{\Lam^2(4 q ) }{n}$).
We obtain

\begin{align*}
\mathcal{V}_{0}  & \leq (1 + r + r^2 + \cdots) ~\cdot ~
\mathbb{E}\bigg[\bigg(\sum_{\substack{s \in [ t- 7q, t + q] \setminus \{u\}}} \frac{ 480 q^2 \Gamma^2 (\Lam^2)^2}{ n^2} \left(\left( \Dbarspan{\pi, s}{k_s} \right)^2  + \Big(\Delta x^{\pi, s}_{{k_s}}\Big)^2 \right) \\
&\hspace*{2.0in} + \frac{96 q  (\Lam^2) L_{\max}^2}{n}\left(\left( \Dbarspan{\pi, u}{k_u} \right)^2  + \Big(\Delta x^{\pi, u}_{{k_u}}\Big)^2 \right)  \bigg)\bigg]\\
&  \le  \frac{3 r^2 \Gamma^2} {160(1-r)q^2}
\sum_{s \in [t - 7q, t+q] \setminus \{u\}}
\left[ \big(\Df_s \big)^2 + \big( \DE_s \big)^2 \right]
+ \frac{3r \Gamma^2}{5q(1-r)} \left[\big(\Df_u \big)^2 + \big( \DE_u \big)^2 \right],
\end{align*}
as $\Lmax \le \G$, $r < 1$, and replacing $160q^2\Lambda^2/n$ by $r$.
\end{proof}


Recall that $R_{m-1} = \{l_0,l_1,\cdots,l_{m-1}\}$
and $t_{m-1} =\min\{t,l_0,l_1,\cdots,l_{m-1}\}$.
In the proofs of Lemmas~\ref{lem::grad::sync::bound::gen} and \ref{basic::recursion::gen}, we define $\widetilde{\Delta}^{t_{m-1}} x_{k_s}^{\pi, s}$ as follows.
If $s \notin A^{\pi, t_{m-1}}$, then $\widetilde{\Delta}^{t_{m-1}} x_{k_s}^{\pi, s}$ is the value of $\Delta x_{k_s}^{\pi, s}$
when $\calU_s$ reads all its inputs and makes its update immediately after first $t_{m-1} - 4q - 1$ updates committed;
otherwise, $\widetilde{\Delta}^{t_{m-1}} x_{k_s}^{\pi, s}  = \Delta_{\max}^{t_{m-1}, R_{m-1}} x_{k_s}^{\pi, s}$. We make the following observation.

\begin{obs}\label{obs::tilde::delta}
For $m\ge 1$,

\noindent
i. If $s\in[t_{m-1}-4q,t_{m-1}+q]\setminus\{u\}$, then $\widetilde{\Delta}^{t_{m-1}} x_{k_s}^{\pi, s} \in \big[\Delta^{t_{m-1}, R_{m-1}}_{\min} x^{\pi, s}_{k_{s}}, \Delta^{t_{m-1}, R_{m-1}}_{\max} x^{\pi, s}_{k_{s}}\big]$;

\noindent
ii. If $s\in[t_{m-1}-4q,t_{m-1}+q]\setminus\{u\}$, then $\widetilde{\Delta}^{t_{m-1}} x_{k_s}^{\pi, s}$ is independent of $\mathcal{U}_u$ and $\mathcal{U}_v$ for $v \in R_{m-1}$;

\noindent
iii. If $s\in[t-4q,t+q]\setminus\{u\}$, then $\widetilde{\Delta}^{t} x_{k_s}^{\pi, s} \in \big[\Delta^{t_{m-1}, \emptyset}_{\min} x^{\pi, s}_{k_{s}}, \Delta^{t_{m-1}, \emptyset}_{\max} x^{\pi, s}_{k_{s}}\big]$;

\noindent
iv. If $s\in[t-4q,t+q]$, then $\widetilde{\Delta}^{t} x_{k_s}^{\pi, s}$ is independent of $\mathcal{U}_u$.
\end{obs}
\begin{proof}
i. 
When $s\in A^{\pi, t_{m-1}}$, $\widetilde{\Delta}^{t} x_{k_s}^{\pi, s} = \Delta^{t_{m-1}, R_{m-1}}_{\max} x^{\pi, s}_{k_{s}}$m and so the result is immediate. We also note, for use in (ii), that in this case
$\widetilde{\Delta}^{t} x_{k_s}^{\pi, s}$ is independent of $R_{m-1}$.

While if $s \notin A^{\pi, t_{m-1}}$, we argue as follows.
Note that every element in $R_{m-1}$ is at least $t_{m-1}$.
By Lemma \ref{lem::SCCrange}, the updates in $R_{m-1}$ have commit time at least $t_{m-1}+1$,
while each of the first $t_{m-1}-4q-1$ updates have commit time at most $t_{m-1}-4q-1 + q + 1 = t_{m-1}-3q$.
Thus, in this case too, the updates in $R_{m-1}$ are excluded from the computation of update $\calU_s$.

Also, by the definition of $\widetilde{\Delta}^{t} x_{k_s}^{\pi, s}$,
$\calU_s$ reads all its inputs and makes its update immediately after first $t_{m-1} - 4q - 1$ updates committed,
all updates in $A^{\pi,t_{m-1}}$ have been fixed before $\calU_s$
starts its reads.
Consequently (i) holds in this case too.

ii. 
We have already shown the independence for $v\in R_{m-1}$ in the proof of (i).

For update $\mathcal{U}_u$, first recall that $t_{m-1}-2q \le u$.
If $s \in A^{\pi, t_{m-1}}$, we want to show that $\Delta^{t_{m-1}, R_{m-1}}_{\max} x_{k_{s}}^{\pi, s}$ is independent of $\calU_u$.
By the definition of $A^{\pi,t_{m-1}}$, $\calU_{s}$ commits earlier than some update $\calU_p$ where $p<t_{m-1}-4q$.
By Lemma~\ref{lem::SCCrange}, the commit time of $\calU_p$ is at most $p+q+1 \le t_{m-1}-3q \le u-q$, hence $\calU_{s}$ must commit before time $u-q$,
and hence it is independent of $\calU_u$ whose start time is at least $u-q+1$ by Lemma~\ref{lem::SCCrange}.
If $s \notin A^{\pi, t}$, $\widetilde{\Delta}^t x_{k_{s}}^{\pi, s}$ depends only on the first $t_{m-1}-4q - 1 < u-2q$ updates,
and these updates must be independent of $\calU_u$ by Lemma~\ref{lem::SCCrange}.

iii. The argument is similar to (i).

iv. The argument is similar to (ii).
\end{proof}

\begin{pfof}{Lemma~\ref{lem::grad::sync::bound::gen}}
For brevity, we write $\mathcal{V} = \bigg(\sum_{l_0 \in [t - 4q, t + q] \setminus \{u\}}  L_{k_{l_0}, k_u} \Dvar{\pi, l_0}{k_{l_0}}{t, \emptyset}\bigg)^2$.
Recall that
\begin{align*}&\mathcal{V} = \bigg(\sum_{l_0 \in [t - 4q, t + q] \setminus \{u\}}  L_{k_{l_0}, k_u} \Dvar{\pi, l_0}{k_{l_0}}{t, \emptyset}\bigg)^2\\
&\hspace*{0.2in} = \bigg(\sum_{l_0 \in [t - 4q, t + q] \setminus \{u\}}  L_{k_{l_0}, k_u} \max\Big\{\Dspan{\pi, l_0}{k_{l_0}}{t, \emptyset}, \big|\Delta_{\max}^{t, \emptyset} x_{k_{l_0}}^{\pi, l_0}\big|, \big|\Delta_{\min}^{t, \emptyset} x_{k_{l_0}}^{\pi, l_0}\big|\Big\}\bigg)^2.
\end{align*}

\hide{If $l_0 \notin A^{\pi, t}$, then let $\widetilde{\Delta}^t x_{k_{l_0}}^{\pi, l_0}$ be the value of $\Delta x_{k_{l_0}}^{\pi, l_0}$ when the update reads
all the coordinate values immediately after first $t - 4q-1$ updates committed;
otherwise, let $\widetilde{\Delta}^t x_{k_{l_0}}^{\pi, l_0} = \Delta^{t, \emptyset}_{\max} x_{k_{l_0}}^{\pi, l_0}$.
Note that in the former case, update $\calU_{l_0}$ will start its reads after all the updates in $A^{\pi,t}$ have committed,
since they all committed before at least one of the first $t-4q-1$ updates.

We claim that $\widetilde{\Delta}^t x_{k_{l_0}}^{\pi, l_0}$ is independent of $\calU_u$ and hence of $k_u$.To see why, first recall that $t-2q \le u\le t$.
If $l_0 \in A^{\pi, t}$, we want to show that $\Delta^{t, \emptyset}_{\max} x_{k_{l_0}}^{\pi, l_0}$ is independent of $\calU_u$.
By the definition of $A^{\pi,t}$, $\calU_{l_0}$ commits earlier than some update $\calU_p$ where $p<t-4q$.
By Lemma~\ref{lem::SCCrange}, the commit time of $\calU_p$ is at most $p+q+1 \le t-3q \le u-q$, hence $\calU_{l_0}$ must commit before time $u-q$,
and hence it is independent of $\calU_u$ whose start time is at least $u-q+1$ by Lemma~\ref{lem::SCCrange}.
If $l_0 \notin A^{\pi, t}$, $\widetilde{\Delta}^t x_{k_{l_0}}^{\pi, l_0}$ depends only on the first $t-4q - 1 < u-2q$ updates,
and these updates must be independent of $\calU_u$ by Lemma~\ref{lem::SCCrange}.

We also claim that $\widetilde{\Delta}^t x_{k_{l_0}}^{\pi, l_0} \in \left[\Delta_{\min}^{t, \emptyset} x_{k_{l_0}}^{\pi, l_0}, \Delta_{\max}^{t, \emptyset} x_{k_{l_0}}^{\pi, l_0}\right]$.
This claim is trivially true if $l_0 \in A^{\pi, t}$.
If $l_0 \notin A^{\pi, t}$, in the definitions of
$\Delta^{t, \emptyset}_{\min} x_{k_{l_0}}^{\pi, l_0}$ and $\Delta^{t, \emptyset}_{\max} x_{k_{l_0}}^{\pi, l_0}$,
we are allowed to exclude any subset of updates in $[t-4q,t+q]\setminus A^{\pi,t}\cup \{l_0\}$, and in particular, to exclude this whole set; the claim now follows.}
By Observation~\ref{obs::tilde::delta}(iii), recentering and applying the Cauchy-Schwarz inequality,
\begin{align*}
&\mathcal{V} \leq  10q \sum_{l_0 \in [t - 4q, t + q] \setminus \{u\} }  L^2_{k_{l_0}, k_u}  \left[\left(\Dspan{\pi, l_0}{k_{l_0}}{t, \emptyset}\right)^2+  \left(\widetilde{\Delta}^t x_{k_{l_0}}^{\pi, l_0}\right)^2 \right] \\
\hide{& \hspace*{1.0in} + 4 L^2_{k_{u}, k_u}  \left[\left(\Dspan{\pi, u}{k_{u}}{t, \emptyset}\right)^2+  \left(\widetilde{\Delta}^t x_{k_{u}}^{\pi, u}\right)^2 \right]\\}
&\hspace*{0.2in}\leq  10q \sum_{l_0 \in [u - q, t + q] \setminus\{u\}} L^2_{k_{l_0}, k_u}  \left(\Dspan{\pi, l_0}{k_{l_0}}{t, \emptyset}\right)^2 + 10q \sum_{s \in [ t - 4q, u - q)} L^2_{k_{s}, k_u} \left(\Dspan{\pi, s}{k_{s}}{t, \emptyset}\right)^2\\
&\hspace*{0.4in} + 10q  \sum_{s \in [ t - 4q, t + q] \setminus\{u\}}  L^2_{k_{s}, k_u} \left(\widetilde{\Delta}^t x_{k_{s}}^{\pi, s}\right)^2.
\end{align*}

Note that by Lemma~\ref{Dmax-defn-works}(ii),
$\Dspan{\pi, l_0}{k_{l_0}}{t, \emptyset} \leq \Dspan{\pi, l_0}{k_{l_0}}{\min\{t, l_0\}, \emptyset}$.
For $s < u-q$, $\Dspan{\pi, s}{k_{s}}{t, \emptyset}$ is independent of update $\calU_u$, as $\Dspan{\pi, s}{k_{s}}{t, \emptyset}$ is independent of updates $\calU_v$ for $v > s+q$. 
Also, as already noted, $\widetilde{\Delta}^t x_{k_{s}}^{\pi, s}$ is independent of $k_u$ by Observation~\ref{obs::tilde::delta}(iv).
Thus, taking the expectation of the previous bound as $k_u$ varies, yields

\begin{align*}
&\mathbb{E}_{k_u}[\mathcal{V}] \leq \mathbb{E}_{k_u}\bigg[10q \sum_{l_0 \in [u - q, t + q] \setminus \{u\}} L^2_{k_{l_0}, k_u}  \left(\Dspan{\pi, l_0}{k_{l_0}}{\min\{t, l_0\}, \emptyset}\right)^2\bigg] \\
&\hspace*{0.5in} + \mathbb{E}_{k_u}\bigg[10q \sum_{s \in [ t - 4q, u - q)} \frac{\Lresbar^2}{n} \left(\Dspan{\pi, s}{k_{s}}{t, \emptyset}\right)^2\bigg] \\
&\hspace*{0.6in} + \mathbb{E}_{k_u}\bigg[10q \sum_{s \in [ t - 4q, t + q] \setminus\{u\}} \frac{\Lresbar^2}{n}  \left(\widetilde{\Delta}^t x_{k_{s}}^{\pi, s}\right)^2\bigg].
\end{align*}

Note that for $s \in [ t - 4q, t + q]\setminus\{u\}$, $\widetilde{\Delta}^t x_{k_s}^{\pi, s}, \Delta x_{k_s}^{\pi, s}\in [\overline{\Delta}_{\min} x_{k_{s}}^{\pi, s}, \overline{\Delta}_{\max} x_{k_{s}}^{\pi, s}]$.
Thus, by recentering,
$\left(\widetilde{\Delta}^t x_{k_s}^{\pi, s}\right)^2 \le 2\left(\Delta x_{k_s}^{\pi, s}\right)^2 + 2\left(\overline{\Delta}_{\mathsf{span}} x_{k_{l_0}}^{\pi, s}\right)^2$.
Then we average over all paths $\pi$:
\begin{align*}
&\mathbb{E}[\mathcal{V}] \leq \mathbb{E}\bigg[10q \sum_{l_0 \in [u - q, t + q] \setminus \{u\}} L^2_{k_{l_0}, k_u}  \left(\Dspan{\pi, l_0}{k_{l_0}}{\min\{t, l_0\}, \emptyset}\right)^2\bigg] \\
&\hspace*{0.4in} + \mathbb{E}\bigg[30q \sum_{s \in [ t - 4q, t + q] \setminus\{u\}} \frac{\Lresbar^2}{n} \left(\left(\Dbarspan{\pi, s}{k_{s}}\right)^2
+ \left(\Delta x_{k_{s}}^{\pi, s}\right)^2\right)\bigg] \\
& \hspace*{0.2in}\leq  \frac{3r^2 \Gamma^2} {16(1-r)q} \sum_{\substack{s \in [t - 7q,
 t+q]\\ \hspace*{0.2in} \setminus \{u\}}}
\left[ \big(\Df_s \big)^2 + \big( \DE_s \big)^2 \right]
+ \frac{6r \Gamma^2}{1-r}  \Big[\big(\Df_u \big)^2 + \big( \DE_u \big)^2 \Big]\\
& \hspace*{0.4in} + 30q \sum_{s \in [ t - 4q, t + q] \setminus\{u\}}  \frac{ \Lresbar^2}{n}
\left[ \big(\Df_s \big)^2 + \big( \DE_s \big)^2 \right]~~~~\text{(by Lemma~\ref{lem::bound-on-calV-0})} \\
& \hspace*{0.2in}\leq  \frac{3r^2 \Gamma^2} {16(1-r)q} \sum_{\substack{s \in [t - 7q,
 t+q]\\ \hspace*{0.2in} \setminus \{u\}}}
\left[ \big(\Df_s \big)^2 + \big( \DE_s \big)^2 \right]
+ \frac{6r \Gamma^2}{1-r} \cdot  \Big[\big(\Df_u \big)^2 + \big( \DE_u \big)^2 \Big]\\
& \hspace*{0.4in} + \frac{3r \Gamma^2}{16 q}   \sum_{s \in [ t - 4q, t + q] \setminus\{u\}}
\left[ \big(\Df_s \big)^2 + \big( \DE_s \big)^2 \right]~~~~\text{$\Big($as $r = \frac{160 \Lam^2 q^2}{n}$ and $\Lam = \frac{\Lresbar^2}{\Gamma^2} + 1.\Big)$}
\end{align*}
\end{pfof}

\hide{
Observe that the first term on the RHS has a similar structure to the LHS of Lemma~\ref{basic::recursion::gen} when $m = 1$.
Also note that the first two terms on the RHS of Lemma~\ref{basic::recursion::gen} have a similar structure to the LHS,
but with $m$ increased by $1$.
For we form $R_m = R_{m-1} \cup \{l_m\}$, and we form $S_m$ either as
(i) $S_{m-1} \cup  \{l_m\}$, but only if $k_{l_m} = k_{l_{m-1}}$, or as  (ii)  $S_{m-1}$.
In the first case, $R_m\setminus S_m = R_{m-1} \setminus S_{m-1}$,
giving the first term on the RHS;
in the second case, $R_m \setminus S_m = ( R_{m-1} \setminus S_{m-1}) \cup \{l_{m}\}$, which gives the second term.
In fact, in combination, ignoring the constant multiplier (of $40q$),
they form the term on the LHS with $m$ increased by 1.
Let
\begin{align}
\nonumber
\mathcal{V}_{m-1} &:= \mathbb{E}\bigg[ \sum_{l_0 \in [\yxt{u - q}, t + q] \setminus \{u\}} \bigg( \sum_{\substack{l_1, l_2, \cdots, l_{m-1} \in [\yxt{u - q}, t+q] \\
\text{all } l_s \text{ distinct, all } l_s  \ne u, l_0; \\ \text{all }l_s \leq t_{s-1} + q, \\ \text{where } t_{s-1} = \min\{t, l_0, l_1, \cdots, l_{s-1}\}; \\  \text{all } {S_{m-1} \subseteq \{ l_s | k_{l_s} = k_{l_{s-1}}\}.} }}\\
\nonumber
&\hspace*{1.4in} \Big(\prod_{\substack{l_s \in R_{m-1} \setminus S_{m-1} \\ \text{ where }\\ R_{m-1} = \{ l_1, l_2, \cdots, l_{m-1}\}.}} \frac{L^2_{k_{l_s}, k_{l_{s-1}}}}{\Gamma^2}\Big) L^2_{k_{l_0}, k_{u}} \\
&\hspace*{2.0in} \left( \Dspan{\pi, l_{m-1}}{k_{l_{m-1}}}{t_{m-1}, R_{m-1} \setminus \{l_{m-1}\}} \right)^2  \bigg) \bigg].
\label{def:Vm}
\end{align}

Note that $\mathcal{V}_{3q} = 0$ and $\mathcal{V}_{0} = \mathbb{E}\bigg[  \sum_{l_0 \in [u - q, t + q] \setminus \{u\}}  L_{k_{l_0}, k_u}^2\left( \Dspan{\pi, l_0}{k_{l_0}}{\min\{t, l_0\}, \emptyset} \right)^2   \bigg]$.

By Lemma~\ref{Dmax-defn-works}(ii), $\Dspan{\pi, u}{k_u}{t_{m-1} R_{m-1}} \le \Dspan{\pi, u}{k_u}{\min\{u,t_{m-1}\} R_{m-1}}
\le \Dbarspan{\pi, u}{k_u}$, as $u - q \leq t_{m-1} \leq u + 2q$.
By this fact and using Lemma~\ref{basic::recursion::gen}, we see that
\begin{align*}
\mathcal{V}_{m-1}
&\leq  (40q) \mathcal{V}_{m} \\
&\hspace*{0.2in}+\mathbb{E}\bigg[\sum_{s \in [t - 7q, t+q] \setminus \{u\}} 120 q \Gamma^2 \frac{\left(\Lam^2\right)^{m+1}
 (4q)^{m}}{n^{m+1}}
\left(\left(\Dbarspan{\pi, s}{k_s}\right)^2 + \left( \Delta x_{k_s}^{\pi, s} \right)^2 \right) \bigg]\\
&\hspace*{0.2in}+\mathbb{E}\bigg[24 L_{\max}^2 \frac{\left(\Lam^2\right)^{m}
 (4q)^{m}}{n^{m}}
\left( \left( \Dbarspan{\pi, u}{k_u}  \right)^2 + \left(\Delta x_{k_u}^{\pi, u} \right)^2\right)\bigg].
\end{align*}

Note that as $m$ increases by 1, each sum is multiplied by $\frac{\Lambda^2 \cdot 4q} {n}$, and the first term has a multiplier of $40 q$.
We set $r = \frac{40 q \Lam^2(4 q ) }{n} = \frac{160 q^2 \Lam^2}{n}$.
Applying this bound recursively and substituting in \eqref{eqn::intermed-bound-on-delta-var},  we obtain
\begin{align*}
&\mathbb{E}\bigg[\bigg( \sum_{l_0 \in [t - 4q, t + q] \setminus \{u\}}  L_{k_{l_0}, k_u} \Dvar{\pi, l_0}{k_{l_0}}{t, \emptyset} \bigg)^2   \bigg] \\
  &\hspace*{0.2in}= 10 q \cdot \mathcal{V}_{0} \\
  &\hspace*{0.4in} + \mathbb{E}\bigg[30q \sum_{s \in [ t - 4q, t + q] \setminus \{u\}} \frac{\Lresbar^2}{n} \left(\Dbarspan{\pi, s}{k_{s}}\right)^2\bigg] \\
&\hspace*{0.4in} + \mathbb{E}\bigg[30q \frac{\Lresbar^2}{n} \sum_{s \in [ t - 4q, t + q] \setminus\{u\}}  \left(\Delta x_{k_{s}}^{\pi, s}\right)^2\bigg]  \\
\hide{&\hspace*{0.2in} + \expect{12 L^2_{\max}  \left(\left(\Dbarspan{\pi, u}{k_{u}}\right)^2+  \left(\Delta x_{k_{u}}^{\pi, u}\right)^2 \right)} \\}
& \hspace*{0.2in} \leq (1 + r + r^2 + \cdots)\cdot \\
&\hspace*{0.4in} \mathbb{E}\bigg[\bigg(\sum_{s \in [ t- 7q, t + q] \setminus \{u\}} \frac{ 4800 q^3 \Gamma^2 (\Lam^2)^2}{ n^2} \left(\left( \Dbarspan{\pi, s}{k_s} \right)^2  + \Big(\Delta x^{\pi, s}_{{k_s}}\Big)^2 \right) \\
&\hspace*{1.5in} + \frac{960 q^2  (\Lam^2) L_{\max}^2}{n}\left(\left( \Dbarspan{\pi, u}{k_u} \right)^2  + \Big(\Delta x^{\pi, u}_{{k_u}}\Big)^2 \right)  \bigg)\bigg] \\
&\hspace*{0.2in} + \mathbb{E}\bigg[30q \sum_{s \in [ t - 4q, t +q] \setminus\{u\}} \frac{\Lresbar^2}{n} \left(\Dbarspan{\pi, s}{k_{s}}\right)^2 + \left(\Delta x_{k_{s}}^{\pi, s}\right)^2 \bigg] \hide{\\
&\hspace*{0.2in} + \expect{12 L^2_{\max}  \left(\left(\Dbarspan{\pi, u}{k_{u}}\right)^2+  \left(\Delta x_{k_{u}}^{\pi, u}\right)^2 \right)} \\}.
\end{align*}

\smallskip

As $r <1$ and since $ \Gamma \geq L_{\max}$,
\begin{align*}
&\mathbb{E}\bigg[ \bigg(\sum_{l_0 \in [t - 4q, t + q] \setminus \{u\}}  L_{k_{l_0}, k_u} \Dvar{\pi, l_0}{k_{l_0}}{t, \emptyset} \bigg)^2   \bigg] \\
&\hspace*{0.2in} \leq  \frac{\Gamma^2}{1-r} \cdot  \mathbb{E}\bigg[ \bigg(\frac{1}{q} \sum_{s \in [t - 7q, t+q] \setminus \{u\}} \left(\frac{r^2}{3} + \frac{r}{4}\right) \left(\left( \Dbarspan{\pi, s}{k_s} \right)^2  + \Big(\Delta x^{\pi, s}_{k_s}\Big)^2 \right)  \\
&\hspace*{2.4in} + 8 r \left(\left( \Dbarspan{\pi, u}{k_u} \right)^2  + \Big(\Delta x^{\pi, u}_{k_u})^2 \Big) \bigg)\right]
\\
&\hspace*{0.2in} \leq  \frac{(4r^2 + 3r) \Gamma^2} {12(1-r)q}
\sum_{s \in [t - 7q, t+q] \setminus \{u\}}
\left[ \left(\Df_s \right)^2 + \left( \DE_s \right)^2 \right]
+ \left(\frac{8r }{1-r} \right)\Gamma^2 \left[\left(\Df_u \right)^2 + \left( \DE_u \right)^2 \right].
\end{align*}
\hide{\rjc{The $\frac 32$ term in the above equation is not tight; $\frac 43$ would
be tight (but is not needed going forward).}}
\end{pfof}
}

%% file: appendix_recursive-new-temp.tex
In the next lemma we bound  $\Big(\Dspan{\pi, l_{m-1}}{k_{l_{m-1}}}{t_{m-1}, R_{m-1} \setminus \{l_{m-1}\}}\Big)^2$.

\begin{lemma}
\label{eqn::intermed-bound-for-rec}
Suppose $l_0, l_1, \cdots \in [u - q, t + q] \setminus \{u\}$, $t_{s} = \min\{t, l_0, l_1, \cdots, l_{s}\}$, $R_{m-1} = \{l_0, l_1, \cdots, l_{m-1}\}$ and $l_{s} \leq t_{s} + q$ for $s < m$. Then

\begin{align}
&\Big(\Dspan{\pi, l_{m-1}}{k_{l_{m-1}}}{t_{m-1}, R_{m-1} \setminus \{l_{m-1}\}}\Big)^2 \nonumber\\
&\hspace*{0.2in}\leq  40q \sum_{\substack{l_m \in  [t_{m-1} - 4q, \min\{l_{m-1}-1, u - q-1\}] \setminus \left(\{u\} \cup R_{m-1} \right) \\ \text{ and } k_{l_m} = k_{l_{m-1}}}} \nonumber\\
&\hspace*{2.0in}\Big[\underbrace{\left(\Dspan{\pi, l_m}{k_{l_m}}{t_{m-1}, R_{m-1}} \right)^2}_{F} + \underbrace{\left(\widetilde{\Delta}^{t_{m-1}} x^{\pi,l_{m}}_{k_{l_{m}}}  \right)^2}_{G}\Big] \nonumber\\
&\hspace*{0.4in} + 40q \sum_{\substack{l_m \in  [\min\{l_{m-1}-1, u - q-1\} + 1, l_{m-1}-1] \setminus \left(\{u\} \cup R_{m-1} \right) \\ \text{ and } k_{l_m} = k_{l_{m-1}}}}\nonumber \\
&\hspace*{2.0in}\Big[\underbrace{\left(\Dspan{\pi, l_m}{k_{l_m}}{t_{m-1}, R_{m-1}} \right)^2}_{H} + \underbrace{\left(\widetilde{\Delta}^{t_{m-1}} x^{\pi,l_{m}}_{k_{l_{m}}}  \right)^2}_{I}\Big] \nonumber\\
&\hspace*{0.4in} + 8 \cdot \mathds{1}_{(k_{l_{m-1}} = k_u \text{ and } u \leq l_{m-1})}\cdot \underbrace{\Big[\left(\Dspan{\pi, u}{k_{u}}{t_{m-1}, R_{m-1}} \right)^2 +\left(\Delta^{t_{m-1}, R_{m-1}}_{\max} x^{\pi, u}_{k_u}\right)^2\Big]}_{J} \nonumber\\
&\hspace*{0.4in} + \frac{40q}{\Gamma^2} \bigg(\sum_{l_m \in [t_{m-1} - 4q, \min\{t_{m-1} + q, u - q -1 \}] \setminus \left(\{l_{m-1}, u\} \cup R_{m-1}\right)} \nonumber\\
&\hspace*{0.8in} \bigg(\underbrace{L^2_{k_{l_m} k_{l_{m-1}}} \left(\Dspan{\pi, l_m}{k_{l_m}}{t_{m-1}, R_{m-1}} \right)^2}_{K} +   \underbrace{L_{k_{l_m} k_{l_{m-1}}}^2\left(\widetilde{\Delta}^{t_{m-1}} x^{\pi, l_m}_{k_{l_m}}\right)^2}_{L}\bigg)\bigg) \nonumber\\
&\hspace*{0.4in}+ \frac{40q}{\Gamma^2} \bigg(\sum_{l_m \in [\min\{t_{m-1} + q, u - q - 1\} + 1,  t_{m-1} + q] \setminus \left(\{l_{m-1}, u\} \cup R_{m-1}\right)} \nonumber\\
&\hspace*{0.8in}\bigg(\underbrace{L^2_{k_{l_m} k_{l_{m-1}}} \left(\Dspan{\pi, l_m}{k_{l_m}}{t_{m-1}, R_{m-1}} \right)^2}_{M} +   \underbrace{L_{k_{l_m} k_{l_{m-1}}}^2\left(\widetilde{\Delta}^{t_{m-1}} x^{\pi, l_m}_{k_{l_m}}\right)^2}_{N}\bigg)\bigg)
\nonumber
\\
&\hspace*{0.4in} + \frac{8}{\Gamma^2} \cdot \mathds{1}_{(u \leq t_{m-1} + q)}\cdot L_{k_u k_{l_{m-1}}}^2 \underbrace{\left[\left(\Dspan{\pi, u}{k_{u}}{t_{m-1}, R_{m-1}} \right)^2 + \left(\Delta^{t_{m-1}, R_{m-1}}_{\max} x^{\pi, u}_{k_{u}}\right)^2\right]}_{O}.\label{eqn::recursive-second-step}
\end{align}
\end{lemma}

\begin{proof}
Recall that $t_{m-1} = \min\{t, l_0, l_1, \cdots, l_{m-1}\}$, and $R_{m-1}$ contains $l_{m-1}$.
\hide{\RJC{I don't think we use the following in this proof.}
$l_{m-1} \in [u - q, t_{m-1}+q]$.}
Also recall that $$\Dvar{\pi,s}{\ks}{t_{m-1}, R} = \max\left\{ \Delta_{\max}^{t_{m-1}, R} x_{k_s}^{\pi, s} -  \Delta_{\min}^{t_{m-1}, R} x_{k_s}^{\pi, s}, \left| \Delta_{\max}^{t_{m-1}, R} x_{k_s}^{\pi, s}\right|, \left| \Delta_{\min}^{t_{m-1}, R} x_{k_s}^{\pi, s} \right| \right\}.$$

First, let's expand the term
\begin{align*}
&\left(\Dspan{\pi, l_{m-1}}{k_{l_{m-1}}}{t_{m-1}, R_{m-1} \setminus \{l_{m-1}\}}\right)^2 \\
&\hspace*{1.0in} = \left( \Delta^{t_{m-1}, R_{m-1} \setminus \{ l_{m-1}\}}_{\max} x^{\pi, l_{m-1}}_{k_{l_{m-1}}} - \Delta^{t_{m-1}, R_{m-1} \setminus \{ l_{m-1}\}}_{\min} x^{\pi, l_{m-1}}_{k_{l_{m-1}}} \right)^2,
\end{align*}
By Lemma~\ref{lem:change-of-Dp-vs-change-of-g-gen},
and recalling that the definition of $\Dvar{\pi, l_m}{k_{l_m}}{t_{m-1}}$
allows the choice of whether to read the results of updates $\calU_s$ with
$s\in[t_{m-1}-4q, t_{m-1}+q]$, we obtain

\begin{align*}
&\left(\Dspan{\pi, l_{m-1}}{k_{l_{m-1}}}{t_{m-1}, R_{m-1} \setminus \{l_{m-1}\}}\right)^2 \\
&\hspace*{0.24in} \leq 2 \bigg(\sum_{\substack{l_m \in  [t_{m-1} - 4q,\\
\hspace*{0.4in} l_{m-1} - 1] \setminus R_{m-1} \\ \text{ and } k_{l_m} = k_{l_{m-1}}}} \Dvar{\pi, l_m}{k_{l_m}}{t_{m-1}, R_{m-1}}\bigg)^2 \\
&\hspace*{1.9in} {+ \frac{2}{\Gamma^2} \left(\widetilde{g}_{\max,k_{l_{m-1}}}^{t_{m-1},R_{m-1},\pi,l_{m-1}} - \widetilde{g}_{\min,k_{l_{m-1}}}^{t_{m-1},R_{m-1},\pi,l_{m-1}}\right)^2} \\
 &\hspace*{0.2in} \leq 2 \bigg(\sum_{\substack{l_m \in  [t_{m-1} - 4q,\\ \hspace*{0.4in} l_{m-1} - 1] \setminus R_{m-1} \\ \text{ and } k_{l_m} = k_{l_{m-1}}}} \Dvar{\pi, l_m}{k_{l_m}}{t_{m-1}, R_{m-1}}\bigg)^2 \\
\numberthis\label{eqn::recursive-first-step}
&\hspace*{0.4in} + \frac{2}{\Gamma^2} \bigg(\sum_{\substack{l_m \in [t_{m-1} - 4q,\\ \hspace*{0.4in} t_{m-1} + q] \setminus R_{m-1} }}
 L_{k_{l_m} k_{l_{m-1}}} \cdot \Dvar{\pi, l_m}{k_{l_m}}{t_{m-1}, R_{m-1}}\bigg)^2.
\end{align*}

By applying the unbalanced Cauchy-Schwarz inequality,
the first term on the RHS of~\eqref{eqn::recursive-first-step} can be bounded as follows (the $5q$ in the first inequality occurs because by \eqref{eqn::lm-bound}, $l_{m-1} \le t_{m-1}+q$, and hence the range for $l_m$ is of size at most $5q$):
\begin{align}
\nonumber
&2 \bigg( \sum_{\substack{l_m \in  [ t_{m-1} - 4q, l_{m-1} - 1] \setminus R_{m-1} \\ \text{ and } k_{l_m} = k_{l_{m-1}}}}\Dvar{\pi, l_m}{k_{l_m}}{t_{m-1}, R_{m-1}} \bigg)^2 \\
\nonumber
&\hspace*{0.4in} \leq 4 \cdot (5q) \sum_{\substack{l_m \in  [ t_{m-1} - 4q,\\
~~~~~~~~ l_{m-1} - 1] \setminus \left(\{u\} \cup R_{m-1}\right) \\ \text{ and } k_{l_m} = k_{l_{m-1}}}}\left( \Dvar{\pi, l_m}{k_{l_m}}{t_{m-1}, R_{m-1}} \right)^2\\
\label{eqn::first-term-in-Del-diff}
&\hspace*{1.4in} + 4 \cdot \mathds{1}_{(k_{l_{m-1}} = k_u \text{ and } u \leq l_{m-1} - 1)} \cdot \left(\Dvar{\pi, u}{k_{u}}{t_{m-1}, R_{m-1}} \right)^2.
\end{align}

Again, by using an unbalanced Cauchy-Schwarz inequality,
the second term on the RHS of~\eqref{eqn::recursive-first-step} can be bounded as follows:
\begin{align*}
&\frac{2}{\Gamma^2} \bigg(\sum_{l_m \in [t_{m-1} - 4q, t_{m-1} + q] \setminus  R_{m-1}} L_{k_{l_m} k_{l_{m-1}}} \cdot \Dvar{\pi, l_m}{k_{l_m}}{t_{m-1}, R_{m-1}} \bigg)^2\\
&\hspace*{0.4in} \leq \frac{4 (5q)}{\Gamma^2} \sum_{\substack{l_m \in [t_{m-1} - 4q,\\
\hspace*{0.4in} t_{m-1}+q] \\
\hspace*{0.4in}\setminus \left(\{u\} \cup R_{m-1}\right)}} L^2_{k_{l_m} k_{l_{m-1}}} \left(\Dvar{\pi, l_m}{k_{l_m}}{t_{m-1}, R_{m-1}}\right)^2  \\
\numberthis\label{eqn::second-term-in-Del-diff}
&\hspace*{1.4in}+\frac{4 }{\Gamma^2} \cdot \mathds{1}_{(u \leq t_{m-1} + q)}\cdot L^2_{k_{u} k_{l_{m-1}}} \left(\Dvar{\pi, u}{k_{u}}{t_{m-1}, R_{m-1}}\right)^2.
\end{align*}

\hide{
If $s \notin A^{\pi, t_{m-1}}$, then let $\widetilde{\Delta}^{t_{m-1}} x_{k_s}^{\pi, s}$ be the value of $\Delta x_{k_s}^{\pi, s}$
when $\calU_s$ reads all its inputs and makes its update immediately after first $t_{m-1} - 4q - 1$ updates committed;
otherwise, let $\widetilde{\Delta}^{t_{m-1}} x_{k_s}^{\pi, s}  = \Delta_{\max}^{t_{m-1}, R_{m-1}} x_{k_s}^{\pi, s}$.
}


Observation~\ref{obs::tilde::delta}(i) implies that for any $s=l_m\in [t_{m-1} -4q, t_{m-1} + q]\setminus\{u\}$,
$\big(\Delta^{t_{m-1}, R_{m-1}}_{\min} x^{\pi, s}_{k_{s}}\big)^2,$
$\big(\Delta^{t_{m-1}, R_{m-1}}_{\max} x^{\pi, s}_{k_{s}}\big)^2
\le \Big(\big|\Dspan{\pi, s}{k_s}{t_{m-1}, R_{m-1}}\big|\big.$
$+\big|\widetilde{\Delta}^{t_{m-1}} x_{k_s}^{\pi, s} \big|
\Big)^2$,
and consequently by recentering,

\begin{align}
\label{eqn::the-first-var-bound}
\Big(\Dvar{\pi, s}{k_{s}}{t_{m-1}, R_{m-1}}\Big)^2 \leq 2 \Big(\Dspan{\pi, s}{k_s}{t_{m-1}, R_{m-1}} \Big)^2 +  2\Big(\widetilde{\Delta}^{t_{m-1}} x_{k_s}^{\pi, s} \Big)^2.
\end{align}

Similarly, since $\Delta^{t_{m-1}, R_{m-1}}_{\max} x^{\pi, u}_{k_{u}} \in \big[\Delta^{t_{m-1}, R_{m-1}}_{\min} x^{\pi, u}_{k_{u}}, \Delta^{t_{m-1}, R_{m-1}}_{\max} x^{\pi, u}_{k_{u}}\big]$,

\begin{align}
\label{eqn::second-var-bound}
\Big(\Dvar{\pi, u}{k_{u}}{t_{m-1}, R_{m-1}}\Big)^2 \leq 2 \Big(\Dspan{\pi, u}{k_u}{t_{m-1}, R_{m-1}} \Big)^2 +  2\Big(\Delta^{t_{m-1}, R_{m-1}}_{\max} x^{\pi, u}_{k_{u}}\Big)^2.
\end{align}

In summary, \eqref{eqn::recursive-first-step}
is bounded by \eqref{eqn::first-term-in-Del-diff} and
\eqref{eqn::second-term-in-Del-diff};
we split the range of the first term in each of these bounds,
and bound the $\Delta_{\mathsf{var}}$ terms as in
\eqref{eqn::the-first-var-bound}
and \eqref{eqn::second-var-bound}.
This yields the result.
\end{proof}

We can now prove Lemma~\ref{basic::recursion::gen}.

\begin{pfof}{Lemma~\ref{basic::recursion::gen}}
Next, we bound the LHS of the inequality in Lemma~\ref{basic::recursion::gen}.
To obtain our bound, we start with \eqref{eqn::recursive-second-step} and multiply both sides by
$$\Big(\prod_{\substack{l_s \in R_{m-1} \setminus S_{m-1} \\ \text{where } R_{m-1} \\ = \{ l_1, l_2, \cdots, l_{m-1}\}. }} \frac{L^2_{k_{l_s} k_{l_{s-1}}}}{\Gamma^2}\Big) L^2_{k_{l_0} k_{u}},$$
and sum over all choices of $l_0, l_1, \cdots, l_{m-1}$, and over all choices of
a set $S_{m-1}$ defined shortly,  where these parameters satisfy the following constraints.\\
i. $l_0, l_1, \cdots, l_{m-1} \in [u - q, t+q]$.\\
ii. All the $l_s$ are distinct and $l_s \ne u$ for all $s$.\\
iii. Let $t_{s-1} = \min\{t, l_0, l_1, \cdots, l_{s-1}\}$; $l_s \leq t_{s-1} + q$ for all $s$.\\
iv. $S_{m-1} \subseteq \{ l_s | k_{l_s} = k_{l_{s-1}}\}$.\\
v. Let $R_{m-1} = \{ l_1,l_2,\ldots,l_{m}\}$; for all $s\ge 1$,
$l_s \in R_{m-1}\setminus S_{m-1}$.

This multiplication of the LHS term of \eqref{eqn::intermed-bound-for-rec} yields

\begin{align*}
&\mathbb{E}\Bigg[ \sum_{l_0 \in [u - q, t + q] \setminus \{u\}} \Bigg( \sum_{\substack{l_1, l_2, \cdots, l_{m-1} \in [u - q, t+q], \\ \text{all } l_s \text{ distinct, all } l_s  \ne u, l_0;  \\ \text{all }l_s \leq t_{s-1} + q, \text{ where }\\ t_{s-1} = \min\{t, l_0, l_1, \cdots, l_{s-1}\} \\ \text{all } {S_{m-1} \subseteq \{ l_s | k_{l_s} = k_{l_{s-1}}\}.} }} \Big(\prod_{\substack{l_s \in R_{m-1} \setminus S_{m-1} \\ \text{where } R_{m-1} \\ = \{ l_1, l_2, \cdots, l_{m-1}\}. }} \frac{L^2_{k_{l_s} k_{l_{s-1}}}}{\Gamma^2}\Big)  \\
&\hspace*{2.0in} L^2_{k_{l_0} k_{u}}\left(\Dspan{\pi, l_{m-1}}{k_{l_{m-1}}}{t_{m-1}, R_{m-1} \setminus \{l_{m-1}\}}\right)^2  \Bigg) \Bigg],
\end{align*}
which is $\mathcal{V}_{m-1}$.

On the RHS,
we start by looking at terms $H$ and $M$ to obtain the recursive term $\mathcal{V}_m$ on the RHS of the Lemma~\ref{basic::recursion::gen}.
We use the inequality $\big(\Dspan{\pi, l_m}{k_{l_m}}{t_{m-1}, R_{m-1}}\big)^2 \leq \big(\Dspan{\pi, l_m}{k_{l_m}}{t_{m}, R_{m-1}}\big)^2$, which holds by Lemma~\ref{Dmax-defn-works}(ii).
The multiplication applied to these two terms yields

\begin{align*}
&\hspace*{0.0in} \leq \mathbb{E} \Bigg[40 q \sum_{\substack{l_0 \in [u - q, t + q]\\ \hspace*{0.4in} \setminus \{u\}}} \Bigg( \sum_{\substack{l_1, l_2, \cdots, l_{m} \in [u-q, 
t+q], \\ \text{all } l_s \text{ distinct, all } l_s  \ne u, l_0;  \\ \text{all }l_s \leq t_{s-1} + q, \text{where }\\  t_{s-1} = \min\{t, l_0, l_1, \cdots, l_{s-1}\}; \\  \text{all } S_{m-1} \subseteq \{ l_s | k_{l_s} = k_{l_{s-1}}\\
\hspace*{0.6in}\text{and}~s\le m-1\}; \\ k_{l_m} = k_{l_{m-1}}.}}
\Big(\prod_{\substack{l_s \in R_{m-1} \setminus S_{m-1}\\ \text{where }  R_{m-1}\\ = \{ l_1, l_2, \cdots, l_{m-1}\}.}} \frac{L^2_{k_{l_s} k_{l_{s-1}}}}{\Gamma^2}\Big) L^2_{k_{l_0} k_{u}}\\
&\hspace{3.0in}
\cdot \left( \Dspan{\pi, l_m}{k_{l_m}}{t_{m}, R_{m-1}}\right)^2  \Bigg) \Bigg]\\
 &\hspace*{0.2in} + \mathbb{E} \Bigg[40 q \sum_{l_0 \in [u - q, t + q] \setminus \{u\}} \Bigg( \sum_{\substack{l_1, l_2, \cdots, l_{m} \in [u-q, 
t+q], \\ \text{all } l_s \text{ distinct, all } l_s  \ne u, l_0;\\ \text{all }l_s \leq t_{s-1} + q, \text{ where }\\  t_{s-1} = \min\{t, l_0, l_1, \cdots, l_{s-1}\}; \\ S_{m-1} \subseteq \{ l_s | k_{l_s} = k_{l_{s-1}}\\ \hspace*{0.6in} \text{ and } s\le m-1\}.} }
\Big(\prod_{\substack{l_s \in R_{m-1} \setminus S_{m-1} \\ \text{where } R_{m-1} \\ = \{ l_1, l_2, \cdots, l_{m-1}\}.}} \frac{L^2_{k_{l_s} k_{l_{s-1}}}}{\Gamma^2}\Big)  \\
&\hspace*{2.0in}\cdot L^2_{k_{l_0} k_{u}}\cdot \frac{L^2_{k_{l_m} k_{l_{m-1}}}}{\Gamma^2}\left( \Dspan{\pi, l_m}{k_{l_m}}{t_{m}, R_{m-1}}\right)^2  \Bigg) \Bigg]
\end{align*}

These two terms correspond to including $l_m$ in $S_m$ and not including it
in the constraint for $\mathcal{V}_m$; in other words, summed together
these two terms equal $\mathcal{V}_m$.

We now analyze the remaining non-recursive terms.

\hide{First, note that, on the RHS of \eqref{eqn::recursive-second-step}, for any $s \leq t_{m-1} + q$, $\widetilde{\Delta}^{t_{m-1}} x_{k_s}^{\pi, s}$ is independent of any updates in  $R_{m-1}$ or update $\mathcal{U}_u$.
as are $\Dspan{\pi, s}{k_{s}}{t_{m-1}}, R_{m-1}$ for $s < u-q$.
The claim for $\Dspan{\pi, s}{k_{s}}{t_{s}, R_{m-1}}$ is clear.
For $\widetilde{\Delta}^{t_{m-1}} x_{k_s}^{\pi, s}$, first note that this
update commits either just after or before the first $t_{m-1} - 4q - 1$ updates committed.
Next recall that any update $\calU_v$ with $v>r+q$ will commit
later than $\calU_{r}$.
Thus it suffices to show that the updates $\calU_v$ with $v\in R_{m-1}$ and $\calU_u$ commit later than all the first $t_{m-1} - 4q$ updates.
Next, note that $t_{m-1} \leq l_s$ for every $l_s \in R_{m-1}$;
so each $\calU_{l_s}$ commits later than all the first $t_{m-1} -q$ updates.
Finally, $u \ge t-2q \ge t_{m-1} - 2q$, and consequently
$\calU_u$ commits later than all the first $t_{m-1} - 3q$ updates.}

Recall that $\Lam^2 = \frac{L_{\overline{res}}^2}{\Gamma^2} + 1$. Term $J$ is bounded as follows, as we explain below:
\begin{align*}
&\mathbb{E}\bigg[ \sum_{\substack{l_0 \in [u - q, t + q]\\ \hspace*{0.2in}\setminus \{u\}}} \bigg( \sum_{\substack{l_1, l_2, \cdots, l_{m-1} \in [u - q, t+q] \\ \text{all } l_s \text{ distinct, all } l_s  \ne u, l_0; \\ \text{all }l_s \leq t_{s-1} + q, \text{ where }\\  t_{s-1} = \min\{t, l_0, l_1, \cdots, l_{s-1}\}; \\  \text{all } {S_{m-1} \subseteq \{ l_s | k_{l_s} = k_{l_{s-1}}\}.} }}
\Big(\prod_{\substack{l_s \in R_{m-1} {\setminus} S_{m-1}\\ \text{where } R_{m-1} \\ = \{ l_1, l_2, \cdots, l_{m-1}\}. \\}} \frac{L^2_{k_{l_s} k_{l_{s-1}}}}{\Gamma^2}\Big) L^2_{k_{l_0} k_{u}} \\
& \hspace*{0.6in} \cdot 8 \cdot \mathds{1}_{\substack{k_{l_{m-1}} = k_u\\ \text{ and } u \leq l_{m-1}}} \left[\left(\Dspan{\pi, u}{k_{u}}{t_{m-1}, R_{m-1}} \right)^2
+\left(\Delta^{t_{m-1}, R_{m-1}}_{\max} x^{\pi, u}_{k_u}\right)^2\right]\bigg)\bigg]\\
\numberthis\label{eqn::result-of-I}
&\hspace*{0.0in} \leq \mathbb{E}\bigg[\sum_{\substack{l_0, l_1, \cdots, l_{m-1}\\
\hspace*{0.2in} \in [u - q, t + q] \setminus\{u\} \\ u \leq t_{m-1} + q }} 8 L_{\max}^2 \frac{(\Lam^2)^{m - 1} \hide{(4q)^{m}}}{n^m}
\left[\left(\Dspan{\pi, u}{k_{u}}{t_{m-1}, R_{m-1}} \right)^2
+\left(\Delta^{t_{m-1}, R_{m-1}}_{\max} x^{\pi, u}_{k_u}\right)^2\right]\bigg].
\end{align*}
Recall that by \eqref{eqn::lm-bound}, $\l_{m-1} \le  t_{m-1} +q$, justifying the bound on the sum.
We first bound $L_{k_{l_0} k_u}^2$ by $ L_{\max}^2$.
Then we can average, in turn, over the random coordinate choices of $k_{l_0}$, $k_{l_1}$, $\ldots$, $k_{l_{m-1}}$.
As we will explain below, depending on whether each $l_{s-1}$ is in $S_{m-1}$ or not, the averaging of $\frac{L^2_{k_{l_s} k_{l_{s-1}}}}{\Gamma^2}$ over $k_{l_{s-1}}$ will provide
either an $\frac{1}{n}$ or an $\frac 1n \cdot \frac{L_{\overline{res}}^2}{\Gamma^2}$ factor.
To elaborate, the factor $\frac{(\Lam^2)^{m - 1} }{n^m}$
on the RHS is due to a combination of the following observations:
\\
i. As in Section~\ref{sec::span-bound},
the $m-1$ factors of $\frac{\left(\frac{\Lresbar^2}{\Gamma^2} +  1\right)}{n} = \frac{\Lam^2}{n}$
are due to the expectation over the bundle of $n$
paths obtained by varying $k_{l_{s-1}}$, which does not affect the terms
$\Dspan{\pi,u}{k_{u}}{t, R_{m-1}}$, $\Delta^{t, R_{m-1}}_{\max} x^{\pi, u}_{k_u}$, as $l_{s-1} \in R_{m-1}$ and $l_{s-1} \neq u$.
In more detail, there is always a term
due to the case $k_{l_{s-1}}\notin S_{m-1}$, which, on averaging,
yields a term $\frac 1n \cdot \frac {\Lresbar^2}{\G^2}$;
when
$k_{l_s} = k_{l_{s-1}}$, there is also a term
 for the case that $k_{l_{s-1}}\in S_{m-1}$, which, on averaging, yields a term of the form $\frac 1n$; in combination, they yield a term $\frac{\Lambda^2}{n}$.
(if $l_s\in S_{m-1}$,
then $k_{l_s} = k_{l_{s-1}}$, and in this case taking the expectation
yields an $1/n$ factor).\\
ii. The extra factor of $n$ in the denominator is due to the expectation of
$\mathds{1}_{k_{l_{m-1}} = k_u}$.

For the remaining terms, first note that, on the RHS of \eqref{eqn::recursive-second-step}, for any $s \in[t_{m-1}-4q, t_{m-1} + q]\setminus\{u\}$, $\widetilde{\Delta}^{t_{m-1}} x_{k_s}^{\pi, s}$ is independent of any updates in  $R_{m-1}$ or update $\mathcal{U}_u$ by Observation~\ref{obs::tilde::delta}(ii).
In addition, for $s < u - q$, $\Dspan{\pi, s}{k_{s}}{t_{m-1}, R_{m-1}}$ also doesn't depend on any updates in  $R_{m-1}$ or update $\mathcal{U}_u$.
\hide{The claim for $\Dspan{\pi, s}{k_{s}}{t_{m-1}, R_{m-1}}$ is clear.
For $\widetilde{\Delta}^{t_{m-1}} x_{k_s}^{\pi, s}$, first note that this
update commits either just after or even before the first $t_{m-1} - 4q-1$ updates committed.
Next recall that by Lemma~\ref{lem::SCCrange},
any update $\calU_v$ with $v>r+q$ will commit
later than $\calU_{r}$.
Thus it suffices to show that the updates $\calU_v$ with $v\in R_{m-1}$ and $\calU_u$ commit later than all the first $t_{m-1} - 4q-1$ updates.
Next, note that $t_{m-1} \leq l_s$ for every $l_s \in R_{m-1}$;
so each $\calU_{l_s}$ commits later than all the first $t_{m-1} -q$ updates.
Finally, $u \ge t-2q \ge t_{m-1} - 2q$, and consequently
$\calU_u$ commits later than all the first $t_{m-1} - 3q$ updates.}

For term $F$, note that on the RHS, $l_m$ is being renamed $s$.
Here we can start by averaging over $k_u$, as $l_m \le u-q$,
which yields a factor of $\frac {\Lresbar^2}{n} \le \frac {\Gamma^2\Lambda^2}{n}$.
The remaining $m-1$ factors of $\frac {\Lambda^2}{n}$ are
as before.
The final factor of $\frac 1n$ is due to the condition $k_{l_m} = k_{l_{m-1}}$.
\begin{align*}
&\mathbb{E}\bigg[ \sum_{\substack{l_0 \in [u - q, t + q]\\ \hspace*{0.2in}\setminus \{u\}}} \bigg( \sum_{\substack{l_1, l_2, \cdots, l_{m-1} \in [u - q, t+q] \\ \text{all } l_s \text{ distinct, all } l_s  \ne u, l_0; \\ \text{all }l_s \leq t_{s-1} + q, \text{ where }\\  t_{s-1} = \min\{t, l_0, l_1, \cdots, l_{s-1}\}; \\  \text{all } {S_{m-1} \subseteq \{ l_s | k_{l_s} = k_{l_{s-1}}\}.} }}
\Big(\prod_{\substack{l_s \in R_{m-1} {\setminus} S_{m-1}\\ \text{where } R_{m-1}\\  = \{ l_1, l_2, \cdots, l_{m-1}\}; \\}} \frac{L^2_{k_{l_s} k_{l_{s-1}}}}{\Gamma^2}\Big) L^2_{k_{l_0} k_{u}} \\
& \hspace*{2.0in} \cdot 40 q \cdot \sum_{\substack{ l_m \in [t_{m-1} - 4q, u - q - 1] \\ \hspace*{0.2in}\setminus (\{u\} \cup R_{m-1}) \\ k_{l_m} = k_{l_{m-1}}}} \left(\Dspan{\pi, l_m}{k_{l_m}}{t_{m-1}, R_{m-1}}\right)^2\\
\numberthis\label{eqn::result-of-F}
&\hspace*{0.2in} \leq \mathbb{E}\bigg[\sum_{\substack{l_0, l_1, \cdots, l_{m-1}\\ \hspace*{0.2in}\in [u - q, t + q] \setminus\{u\}}}40 q \Gamma^2 \frac{(\Lam^2)^{m} }{n^{m+1}}
\sum_{\substack{ s \in [t_{m-1}\\ \hspace*{0.2in} - 4q, u - q)}} \left(\Dspan{\pi, s}{k_{s}}{t_{m-1}, R_{m-1}}\right)^2\bigg].
\end{align*}

For the next expression, we sum terms $G$ and $I$.
Again, we can start by averaging over $k_u$, as
$\widetilde{\Delta}^{t_{m-1}} x_{k_{l_m}}^{\pi, l_m}$ is independent
of $\calU_u$ by Observation~\ref{obs::tilde::delta}(ii), as $l_m \in [t_{m-1} - 4q, l_{m-1}-1] \subset [t_{m-1} - 4q, t_{m-1} + q]$ (as can be seen by applying \eqref{eqn::lm-bound}).

\begin{align*}
&\mathbb{E}\bigg[ \sum_{\substack{l_0 \in [u - q, t + q]\\ \hspace*{0.2in}\setminus \{u\}}} \bigg( \sum_{\substack{l_1, l_2, \cdots, l_{m-1} \in [u - q, t+q] \\ \text{all } l_s \text{ distinct, all } l_s  \ne u, l_0; \\ \text{all }l_s \leq t_{s-1} + q, \text{ where } \\ t_{s-1} = \min\{t, l_0, l_1, \cdots, l_{s-1}\}; \\  \text{all } {S_{m-1} \subseteq \{ l_s | k_{l_s} = k_{l_{s-1}}\}.} }}
\Big(\prod_{\substack{l_s \in R_{m-1} {\setminus} S_{m-1}\\ \text{where } R_{m-1} \\ = \{ l_1, l_2, \cdots, l_{m-1}\}; \\}} \frac{L^2_{k_{l_s} k_{l_{s-1}}}}{\Gamma^2}\Big) L^2_{k_{l_0} k_{u}} \\
& \hspace*{2.2in} \cdot 40 q \cdot \sum_{\substack{ l_m \in [t_{m-1} - 4q, l_{m-1} - 1]\\ \hspace*{0.2in} \setminus (\{u\} \cup R_{m-1}) \\ k_{l_m} = k_{l_{m-1}}}} \left(\widetilde{\Delta}^{t_{m-1}} x_{k_{l_m}}^{\pi, l_m} \right)^2\\
&\hspace*{0.2in} \leq \mathbb{E}\bigg[\sum_{\substack{l_0, l_1, \cdots, l_{m-1}\\ \hspace*{0.2in} \in [u - q, t + q] \setminus\{u\}}}40q \Gamma^2 \frac{(\Lam^2)^{m} }{n^{m+1}}
\sum_{ s \in [t_{m-1} - 4q, l_{m-1} - 1] \setminus \{u\} } \left(\widetilde{\Delta}^{t_{m-1}} x_{k_{s}}^{\pi, s} \right)^2\bigg].
\numberthis\label{eqn::result-of-GH}
\end{align*}

For term $K$, there is an additional term
$L^2_{k_{l_m}} k_{l_{m-1}}$ to average over;
we bound it by $\frac{\Lresbar^2}{n}$.

\begin{align*}
&\mathbb{E}\bigg[ \sum_{\substack{l_0 \in [u - q, t + q]\\ \hspace*{0.2in}\setminus \{u\}}} \bigg( \sum_{\substack{l_1, l_2, \cdots, l_{m-1} \in [u - q, t+q] \\ \text{all } l_s \text{ distinct, all } l_s  \ne u, l_0; \\ \text{all }l_s \leq t_{s-1} + q, \text{ where }\\  t_{s-1} = \min\{t, l_0, l_1, \cdots, l_{s-1}\}; \\  \text{all } {S_{m-1} \subseteq \{ l_s | k_{l_s} = k_{l_{s-1}}\}.} }}
\Big(\prod_{\substack{l_s \in R_{m-1} {\setminus} S_{m-1}\\ \text{where } R_{m-1}\\  = \{ l_1, l_2, \cdots, l_{m-1}\}; \\}} \frac{L^2_{k_{l_s} k_{l_{s-1}}}}{\Gamma^2}\Big) L^2_{k_{l_0} k_{u}} \\
& \hspace*{1.2in} \cdot 40 q  \sum_{\substack{ l_m \in [t_{m-1} - 4q, u - q - 1]\\ \hspace*{0.2in} \setminus (\{u\} \cup R_{m-1})}} \frac{L^2_{k_{l_m} k_{l_{m-1}}}}{\Gamma^2}\left(\Dspan{\pi, l_m}{k_{l_m}}{t_{m-1}, R_{m-1}}\right)^2\\
\numberthis\label{eqn::result-of-J}
&\hspace*{0.2in} \leq \mathbb{E}\bigg[\sum_{\substack{l_0, l_1, \cdots, l_{m-1} \\ \hspace*{0.2in} \in [u - q, t + q]\setminus\{u\}}}40q\Gamma^2 \frac{ \Lresbar^2 (\Lam^2)^{m} }{\Gamma^2 n^{m+1}}
\sum_{ s \in [t_{m-1} - 4q, u - q - 1]} \left(\Dspan{\pi, s}{k_{s}}{t_{m-1}, R_{m-1}}\right)^2\bigg].
\end{align*}

For the next expression, the sum of terms $L$ and $N$, we obtain
\begin{align*}
&\mathbb{E}\bigg[ \sum_{\substack{l_0 \in [u - q, t + q]\\ \hspace*{0.2in}\setminus \{u\}}} \bigg( \sum_{\substack{l_1, l_2, \cdots, l_{m-1} \in [u - q, t+q] \\ \text{all } l_s \text{ distinct, all } l_s  \ne u, l_0; \\ \text{all }l_s \leq t_{s-1} + q, \text{ where }\\  t_{s-1} = \min\{t, l_0, l_1, \cdots, l_{s-1}\}; \\  \text{all } {S_{m-1} \subseteq \{ l_s | k_{l_s} = k_{l_{s-1}}\}.} }}
\Big(\prod_{\substack{l_s \in R_{m-1} {\setminus} S_{m-1}\\ \text{where } R_{m-1}\\  = \{ l_1, l_2, \cdots, l_{m-1}\}; \\}} \frac{L^2_{k_{l_s} k_{l_{s-1}}}}{\Gamma^2}\Big) L^2_{k_{l_0} k_{u}} \\
& \hspace*{1.4in} \cdot 40q \cdot \sum_{\substack{ l_m \in [t_{m-1} - 4q, t_{m-1} + q] \\\hspace*{0.2in}\setminus (\{u\} \cup R_{m-1}) }} \frac{L^2_{k_{l_m} k_{l_{m-1}}}}{\Gamma^2} \left(\widetilde{\Delta}^{t_{m-1}} x_{k_{l_m}}^{\pi, l_m} \right)^2\\
\numberthis\label{eqn::result-of-KL}
&\hspace*{0.2in} \leq \mathbb{E}\bigg[\sum_{\substack{l_0, l_1, \cdots, l_{m-1}\\ \hspace*{0.2in} \in [u - q, t + q] \setminus\{u\}}}40q \Gamma^2 \frac{\Lresbar^2(\Lam^2)^{m} }{\Gamma^2 n^{m+1}}
\sum_{ s \in [t_{m-1} - 4q, t_{m-1} + q] \setminus \{u\}} \left(\widetilde{\Delta}^{t_{m-1}} x_{k_{s}}^{\pi, s} \right)^2\bigg].
\end{align*}

Finally, for term $O$, the bound is
\begin{align*}
&\mathbb{E}\bigg[ \sum_{\substack{l_0 \in [u - q, t + q]\\ \hspace*{0.2in}\setminus \{u\}}} \bigg( \sum_{\substack{l_1, l_2, \cdots, l_{m-1} \in [u - q, t+q] \\ \text{all } l_s \text{ distinct, all } l_s  \ne u, l_0; \\ \text{all }l_s \leq t_{s-1} + q, \text{ where }\\  t_{s-1} = \min\{t, l_0, l_1, \cdots, l_{s-1}\}; \\  \text{all } {S_{m-1} \subseteq \{ l_s | k_{l_s} = k_{l_{s-1}}\}.} }}
\Big(\prod_{\substack{l_s \in R_{m-1} {\setminus} S_{m-1}\\ \text{where } R_{m-1}\\  = \{ l_1, l_2, \cdots, l_{m-1}\}; \\}} \frac{L^2_{k_{l_s} k_{l_{s-1}}}}{\Gamma^2}\Big) L^2_{k_{l_0} k_{u}} \\
& \hspace*{0.6in} \cdot 8 \cdot \frac{L_{k_u k_{l_{m-1}}}^2}{\Gamma^2}  \mathds{1}_{ u \leq t_{m-1} + q} \left[\left(\Dspan{\pi, u}{k_{u}}{t_{m-1}, R_{m-1}} \right)^2
+\left(\Delta^{t_{m-1}, R_{m-1}}_{\max} x^{\pi, u}_{k_u}\right)^2\right]\bigg)\bigg]\\
\numberthis\label{eqn::result-of-M}
&\hspace*{0.0in} \leq \mathbb{E}\bigg[\sum_{\substack{l_0, l_1, \cdots, l_{m-1} \\ \hspace*{0.2in} \in [u - q, t + q]\\ \hspace*{0.2in}\setminus\{u\} \\ u \leq t_{m-1} + q }} 8 L_{\max}^2 \frac{\Lresbar^2 (\Lam^2)^{m - 1} }{\Gamma^2 n^m}
\left[\left(\Dspan{\pi, u}{k_{u}}{t_{m-1}, R_{m-1}} \right)^2
+\left(\Delta^{t_{m-1}, R_{m-1}}_{\max} x^{\pi, u}_{k_u}\right)^2\right]\bigg].
\end{align*}
\hide{For the second and the fifth expectation term, as
$\left(\widetilde{\Delta}^{t_{m-1}} x^{\pi, s}_{k_{s}}\right)^2$ is fixed over all paths $\pi$ obtained by varying $k_u$,
we can do the averaging without needing to introduce the $4\Lmax^2$ term.}

\hide{
We obtain
\begin{align*}
&\mathbb{E}\Bigg[ \sum_{l_0 \in [u - q, t + q] \setminus \{u\}} \Bigg( \sum_{\substack{l_1, l_2, \cdots, l_{m-1} \in [u - q, t+q], \\ \text{all } l_s \text{ distinct, all } l_s  \ne u, l_0;  \\ \text{all }l_s \leq t_{s-1} + q, \text{ where }\\ t_{s-1} = \min\{t, l_0, l_1, \cdots, l_{s-1}\} \\ \text{all } {S_{m-1} \subseteq \{ l_s | k_{l_s} = k_{l_{s-1}}\}.} }} \Big(\prod_{\substack{l_s \in R_{m-1} \setminus S_{m-1} \\ \text{where } R_{m-1} \\ = \{ l_1, l_2, \cdots, l_{m-1}\}. }} \frac{L^2_{k_{l_s} k_{l_{s-1}}}}{\Gamma^2}\Big)  \\
&\hspace*{2.0in} L^2_{k_{l_0} k_{u}}\left(\Dspan{\pi, l_{m-1}}{k_{l_{m-1}}}{t_{m-1}, R_{m-1} \setminus \{l_{m-1}\}}\right)^2  \Bigg) \Bigg] \\
&\hspace*{0.0in} \leq \mathbb{E} \Bigg[40 q \sum_{\substack{l_0 \in [u - q, t + q]\\ \hspace*{0.4in} \setminus \{u\}}} \Bigg( \sum_{\substack{l_1, l_2, \cdots, l_{m} \in [u-q, 
t+q], \\ \text{all } l_s \text{ distinct, all } l_s  \ne u, l_0;  \\ \text{all }l_s \leq t_{s-1} + q, \text{where }\\  t_{s-1} = \min\{t, l_0, l_1, \cdots, l_{s-1}\}; \\  \text{all } S_{m-1} \subseteq \{ l_s | k_{l_s} = k_{l_{s-1}}\\
\hspace*{0.6in}\text{and}~s\le m-1\}; \\ k_{l_m} = k_{l_{m-1}}.}}
\Big(\prod_{\substack{l_s \in R_{m-1} \setminus S_{m-1}\\ \text{where }  R_{m-1}\\ = \{ l_1, l_2, \cdots, l_{m-1}\}.}} \frac{L^2_{k_{l_s} k_{l_{s-1}}}}{\Gamma^2}\Big) L^2_{k_{l_0} k_{u}}\\
&\hspace{3.0in}
\cdot \left( \Dspan{\pi, l_m}{k_{l_m}}{t_{m-1}, R_{m-1}}\right)^2  \Bigg) \Bigg]\\
 &\hspace*{0.2in} + \mathbb{E} \Bigg[40 q \sum_{l_0 \in [u - q, t + q] \setminus \{u\}} \Bigg( \sum_{\substack{l_1, l_2, \cdots, l_{m} \in [u-q, 
t+q], \\ \text{all } l_s \text{ distinct, all } l_s  \ne u, l_0;\\ \text{all }l_s \leq t_{s-1} + q, \text{ where }\\  t_{s-1} = \min\{t, l_0, l_1, \cdots, l_{s-1}\}; \\ S_{m-1} \subseteq \{ l_s | k_{l_s} = k_{l_{s-1}}\\ \hspace*{0.6in} \text{ and } s\le m-1\}.} }
\Big(\prod_{\substack{l_s \in R_{m-1} \setminus S_{m-1} \\ \text{where } R_{m-1} \\ = \{ l_1, l_2, \cdots, l_{m-1}\}.}} \frac{L^2_{k_{l_s} k_{l_{s-1}}}}{\Gamma^2}\Big)  \\
&\hspace*{2.0in}\cdot L^2_{k_{l_0} k_{u}}\cdot \frac{L^2_{k_{l_m} k_{l_{m-1}}}}{\Gamma^2}\left( \Dspan{\pi, l_m}{k_{l_m}}{t_{m-1}, R_{m-1}}\right)^2  \Bigg) \Bigg]\\
 &\hspace*{0.2in}+\mathbb{E}\Bigg[\sum_{\substack{l_0, l_1, \cdots, l_{m-1} \\ \in [u - q, t+ q] \setminus \{u\}}} \sum_{\substack{s \in [t - 7q, t_{m-1}+q] \\ \hspace*{0.2in} \setminus \{u\}}} 40q \Gamma^2 \frac{(\Lam^2)^{m} }{n^{m+1}}
\left(\widetilde{\Delta}^{t_{m-1}} x^{\pi, s}_{k_{s}} \right)^2 \Bigg] \\
&\hspace*{1.6in}\text{(from terms $G + H$)} \\
&\hspace*{0.8in}\text{(on relaxing the range from $s \in [t_{m-1} - 4q ,l_{m-1} - 1] \subseteq [t -7q , t_{m-1} + q]$}\\
&\hspace*{0.8in} \text{ as $t_{m-1} = \min\{l_{m-1}, t_{m-2}\} \geq \min\{l_{m-1}, l_{m-1} - q\}$.)}\\
&\hspace*{0.2in}+\mathbb{E}\Bigg[\sum_{\substack{l_0, l_1, \cdots, l_{m-1}\\ \hspace*{0.2in} \in [u - q, t+ q] \setminus \{u\}}} \sum_{\substack{s \in [t - 7q, u - q)\\ \hspace*{0.2in} \setminus \{u\}}} 40q \Gamma^2 \frac{(\Lam^2)^{m} }{n^{m+1}} \left(\Dspan{\pi, s}{k_{s}}{t_{m-1}, R_{m-1}}\right)^2  \Bigg]\\
& \hspace*{2.0in}\text{(from term $F$)}  \\
&\hspace*{0.8in}\text{(on relaxing the range from $s \in [t_{m-1} - 4q ,u-q - 1]  \subseteq [t -7q , u - q ]$)}\\
&\hspace*{0.2in} + \mathbb{E}\bigg[8  \sum_{\substack{l_0, l_1, \cdots, l_{m-1}\\ \in [u - q, t + q] \\
\hspace*{0.2in}\setminus\{u\} \\ u \leq t_{m-1} + q }} L_{\max}^2 \frac{(\Lam^2)^{m - 1} }{n^m}
\left(\left(\Dspan{\pi, u}{k_{u}}{t_{m-1}, R_{m-1}} \right)^2 + \left(\Delta^{t_{m-1}, R_{m-1}}_{\max} x^{\pi, u}_{k_u}\right)^2\right)\bigg] \\
& \hspace*{2.0in}\text{(from term $I$)}\\
&\hspace*{0.2in}+\mathbb{E}\Bigg[\sum_{\substack{l_0, l_1, \cdots, l_{m-1} \\ \hspace*{0.2in} \in [u - q, t+ q] \setminus \{u\}}}\sum_{\substack{s \in [t - 7q, t_{m-1} + q] \\ \hspace*{0.2in} \setminus \{u\}}} 40 q \Gamma^2 \frac{\Lresbar^2 (\Lam^2)^{m}}{\Gamma^2 n^{m+1}}
\left(\widetilde{\Delta}^{t_{m-1}} x^{\pi, s}_{k_{s}} \right)^2 \Bigg]\\ & \hspace*{2.0in}\text{(from terms $K+L$)}\\
&\hspace*{0.2in}+\mathbb{E}\Bigg[\sum_{\substack{l_0, l_1, \cdots, l_{m-1}\\
\hspace*{0.2in} \in [u - q, t+ q] \setminus \{u\}}}\sum_{\substack{s \in [t - 7q, u-q) \\ \hspace*{0.2in}  \setminus \{u\} }} 40q  \Gamma^2 \frac{\Lresbar^2 (\Lam^2)^{m} } { \Gamma^2 n^{m+1}}\left(\Dspan{\pi, s}{k_{s}}{t_{m-1}, R_{m-1}}\right)^2  \Bigg] \\
& \hspace*{2.0in} \text{(from term $J$)}\\
&\hspace*{0.2in}+\mathbb{E}\bigg[\sum_{\substack{l_0, l_1, \cdots, l_{m-1}\\ \hspace*{0.2in} \in [u - q, t + q] \setminus\{u\} \\ u \leq t_{m-1} + q }}8 L_{\max}^2 \frac{\Lresbar^2 (\Lam^2)^{m - 1} }{\Gamma^2 n^{m}}\\
&\hspace*{1.5in} \cdot \left(\left(\Dspan{\pi, u}{k_{u}}{t_{m-1}, R_{m-1}} \right)^2 +  \left(\Delta^{t_{m-1}, R_{m-1}}_{\max} x^{\pi, u}_{k_u} \right)^2\right)\bigg]\\
& \hspace*{2.0in}  \text{(from term $M$)}. \numberthis \label{ineq::recursive::basic}
\end{align*}
}

Therefore, the sum of the bounds on the non-recursive terms, reordered as
\ref{eqn::result-of-GH},
\ref{eqn::result-of-KL},
\ref{eqn::result-of-F},
\ref{eqn::result-of-J},
\ref{eqn::result-of-I},
\ref{eqn::result-of-M},
equals

\begin{align*}
&\mathbb{E}\bigg[\sum_{\substack{l_0, l_1, \cdots, l_{m-1}\\
\hspace*{0.2in} \in [u - q, t+ q] \setminus \{u\}}} \sum_{\substack{s \in [t_{m-1} - 4q, l_{m-1} - 1] \\ \hspace*{0.2in} \setminus \{u\}}} 40q \Gamma^2 \frac{(\Lam^2)^{m} }{n^{m+1}}
\left(\widetilde{\Delta}^{t_{m-1}} x^{\pi, s}_{k_{s}} \right)^2 \bigg] \numberthis \label{eqn::recursive::1}\\
& + \mathbb{E}\bigg[\sum_{\substack{l_0, l_1, \cdots, l_{m-1}
\\ \hspace*{0.2in} \in [u - q, t+ q] \setminus \{u\}}}\sum_{\substack{s \in [t_{m-1} - 4q, t_{m-1} + q]\\ \hspace*{0.2in} \setminus \{u\}}} 40 q \Gamma^2 \frac{\Lresbar^2 (\Lam^2)^{m}}{\Gamma^2 n^{m+1}}
\left(\widetilde{\Delta}^{t_{m-1}} x^{\pi, s}_{k_{s}} \right)^2 \bigg] \numberthis \label{eqn::recursive::2}\\
&+\mathbb{E}\bigg[\sum_{\substack{l_0, l_1, \cdots, l_{m-1} \\ \hspace{0.2in} \in [u - q, t+ q] \setminus \{u\}}} \sum_{\substack{s \in [t_{m-1} - 4q, u - q) \\ \hspace{0.8in} \setminus \{u\}}} 40q \Gamma^2 \frac{(\Lam^2)^{m}}{n^{m+1}}\left(\Dspan{\pi, s}{k_{s}}{t_{m-1}, R_{m-1}}\right)^2  \bigg] \numberthis \label{eqn::recursive::3} \\
&+\mathbb{E}\bigg[\sum_{\substack{l_0, l_1, \cdots, l_{m-1} \\ \hspace{0.6in} \in [u - q, t+ q] \setminus \{u\}}}\sum_{\substack{s \in [t_{m-1} - 4q, u-q)\\ \hspace{0.2in} \setminus \{u\}}} 40q  \Gamma^2 \frac{\Lresbar^2 (\Lam^2)^{m} }{ \Gamma^2 n^{m+1}}\left(\Dspan{\pi, s}{k_{s}}{t_{m-1}, R_{m-1}}\right)^2  \bigg] \numberthis \label{eqn::recursive::4}\\
&+\mathbb{E}\bigg[8  \sum_{\substack{l_0, l_1, \cdots, l_{m-1} \\ \hspace*{0.2in}\in [u - q, t + q] \setminus\{u\} \\ u \leq t_{m-1} + q }} L_{\max}^2 \frac{(\Lam^2)^{m - 1} }{n^m}
\left(\left(\Dspan{\pi, u}{k_{u}}{t_{m-1}, R_{m-1}} \right)^2 + \left(\Delta^{t_{m-1}, R_{m-1}}_{\max} x^{\pi, u}_{k_u}\right)^2\right)\bigg]\numberthis \label{eqn::recursive::5} \\
&\hspace*{0.2in}+\mathbb{E}\bigg[\sum_{\substack{l_0, l_1, \cdots, l_{m-1}\\ \hspace*{0.2in} \in [u - q, t + q] \setminus\{u\} \\ u \leq t_{m-1} + q }} 8 L_{\max}^2 \frac{\Lresbar^2 (\Lam^2)^{m - 1} }{\Gamma^2 n^{m}}\\
& \hspace*{1.4in} \cdot \left(\left(\Dspan{\pi, u}{k_{u}}{t_{m-1}, R_{m-1}} \right)^2 +  \left(\Delta^{t_{m-1}, R_{m-1}}_{\max} x^{\pi, u}_{k_u} \right)^2\right)\bigg].\numberthis \label{eqn::recursive::6}
\end{align*}

To obtain the final result, we combine these bounds.
We start by bounding the sum of \eqref{eqn::recursive::1} and \eqref{eqn::recursive::2}
as follows.
Noting that $\widetilde{\Delta}^{t_{m-1}}x^{\pi, s}_{k_s}, \Delta x^{\pi, u}_{k_u} \in \left[ \overline{\Delta}_{\min} x^{\pi, u}_{k_u}, \overline{\Delta}_{\max} x^{\pi, u}_{k_u} \right]$ for $s \leq t_{m-1} + q$, by recentering, $\left(\widetilde{\Delta}^{t_{m-1}}x^{\pi, s}_{k_s} \right)^2 \leq 2 \left(\Dbarspan{\pi, s}{k_s}\right)^2 + 2 \left(\Delta x^{\pi, s}_{k_s} \right)^2$, we obtain

\begin{align*}
&\mathbb{E}\bigg[\sum_{\substack{l_0, l_1, \cdots, l_{m-1}\\
\hspace*{0.2in} \in [u - q, t+ q] \setminus \{u\}}} \sum_{\substack{s \in [t_{m-1} - 4q, l_{m-1}-1] \\ \hspace*{0.2in} \setminus \{u\}}} 40q \Gamma^2 \frac{(\Lam^2)^{m} }{n^{m+1}}
\left(\widetilde{\Delta}^{t_{m-1}} x^{\pi, s}_{k_{s}} \right)^2 \bigg] \\
& \hspace*{0.6in}+ \mathbb{E}\bigg[\sum_{\substack{l_0, l_1, \cdots, l_{m-1}
\\ \hspace*{0.2in} \in [u - q, t+ q] \setminus \{u\}}}\sum_{\substack{s \in [t_{m-1} - 4q, t_{m-1} + q]\\ \hspace*{0.2in} \setminus \{u\}}} 40 q \Gamma^2 \frac{\Lresbar^2 (\Lam^2)^{m}}{\Gamma^2 n^{m+1}}
\left(\widetilde{\Delta}^{t_{m-1}} x^{\pi, s}_{k_{s}} \right)^2 \bigg] \\
& \hspace*{0.2in}\leq \mathbb{E}\bigg[\sum_{\substack{s \in [t - 7q, t + q]\\ \hspace*{0.2in} \setminus \{u\}}} 80 q \Gamma^2 \frac{ (\Lam^2)^{m + 1} (4q)^m}{ n^{m+1}}
\Big(\big(\Dbarspan{\pi, s}{k_s}\big)^2 + \big(\Delta x^{\pi, s}_{k_s} \big)^2 \Big) \bigg]. \numberthis \label{ineq::rec::1}
\end{align*}
The additional $\Lam^2$ is from $1 + \frac{\Lresbar^2}{\Gamma^2}$ on the LHS of \eqref{ineq::rec::1}, and the $(4q)^m$ is due to the $t+q-(u-q)\le (t+q) - (t-3q) = 4q$ choices of $l_0, l_1, \cdots, l_{m-1}$. We also relax the range $[t_{m-1} - 4q, t_{m-1} + q] \setminus \{u\}$ and the range $[t_{m-1} - 4q, l_{m-1} - 1] \setminus \{u\}$ to $[t - 7q, t + q] \setminus \{u\}$ as $l_{m-1} \leq t + q$, $t_{m-1} + q \leq t + q$, and $t_{m-1} - 4q \geq u - 5q \geq t - 7q$.

Next, we sum \eqref{eqn::recursive::3} and \eqref{eqn::recursive::4}.
\hide{As $t_m = \min\{t_{m-1}, l_m\}$, by Lemma~\ref{Dmax-defn-works}(ii),\\
$\Big(\Dspan{\pi, l_m}{k_{l_m}}{t_{m-1}, R_{m-1}}\Big)^2 \leq \Big(\Dspan{\pi, l_m}{k_{l_m}}{t_{m}, R_{m-1}}\Big)^2$\YKC{Why this inequality is needed?},}
As $t_{m-1} = \min\{t, l_0, \cdots, l_{m-1}\} \geq s - q$, we obtain
\begin{align*}
&\mathbb{E}\bigg[\sum_{\substack{l_0, l_1, \cdots, l_{m-1} \\ \hspace{0.2in} \in [u - q, t+ q] \setminus \{u\}}} \sum_{\substack{s \in [t_{m-1} - 4q, u - q) \\ \hspace{0.8in} \setminus \{u\}}} 40q \Gamma^2 \frac{(\Lam^2)^{m}}{n^{m+1}}\left(\Dspan{\pi, s}{k_{s}}{t_{m-1}, R_{m-1}}\right)^2  \bigg]  \\
&+\mathbb{E}\bigg[\sum_{\substack{l_0, l_1, \cdots, l_{m-1} \\ \hspace{0.6in} \in [u - q, t+ q] \setminus \{u\}}}\sum_{\substack{s \in [t_{m-1} - 4q, u-q)\\ \hspace{0.2in} \setminus \{u\}}} 40q  \Gamma^2 \frac{\Lresbar^2 (\Lam^2)^{m} }{ \Gamma^2 n^{m+1}}\left(\Dspan{\pi, s}{k_{s}}{t_{m-1}, R_{m-1}}\right)^2  \bigg] \\
&\hspace{0.2in}\le \mathbb{E}\bigg[\sum_{s \in [t - 7q, t+q] \setminus \{u\}} 40q  \Gamma^2 \frac{(\Lam^2)^{m + 1} (4q)^m }{ n^{m+1}}\left(\Dbarspan{\pi, s}{k_{s}}\right)^2  \bigg].
\numberthis \label{ineq::rec::2}
\end{align*}

Adding \eqref{ineq::rec::1} and \eqref{ineq::rec::2} yields the term \eqref{eqn::lem-rec-gen-term1}
on the RHS of Lemma~\ref{basic::recursion::gen}.

Finally, we sum \eqref{eqn::recursive::5} and \eqref{eqn::recursive::6}.
For $ u \leq t_{m-1} + q$, as  $\Delta^{t_{m-1}, R_{m-1}}_{\max} x^{\pi, u}_{k_u}$, $\Delta x^{\pi, u}_{k_u} \in \left[ \overline{\Delta}_{\min} x^{\pi, u}_{k_u}, \overline{\Delta}_{\max} x^{\pi, u}_{k_u} \right]$, by recentering $\left(\Delta^{t_{m-1}, R_{m-1}}_{\max} x^{\pi, u}_{k_u} \right)^2 \leq 2 \left(\Dbarspan{\pi, u}{k_u}\right)^2 + 2 \left(\Delta x^{\pi, u}_{k_u} \right)^2$, and $\left(\Dspan{\pi, u}{k_{u}}{t_{m-1}, R_{m-1}} \right)^2 \leq \left(\Dbarspan{\pi, u}{k_{u}} \right)^2$.
\begin{align*}&\mathbb{E}\bigg[8  \sum_{\substack{l_0, l_1, \cdots, l_{m-1} \\ \hspace*{0.2in}\in [u - q, t + q] \setminus\{u\} \\ u \leq t_{m-1} + q }} L_{\max}^2 \frac{(\Lam^2)^{m - 1} }{n^m}
\left(\left(\Dspan{\pi, u}{k_{u}}{t_{m-1}, R_{m-1}} \right)^2 + \left(\Delta^{t_{m-1}, R_{m-1}}_{\max} x^{\pi, u}_{k_u}\right)^2\right)\bigg] \\
&\hspace*{0.8in}+\mathbb{E}\bigg[\sum_{\substack{l_0, l_1, \cdots, l_{m-1}\\ \hspace*{0.2in} \in [u - q, t + q] \setminus\{u\} \\ u \leq t_{m-1} + q }} 8 L_{\max}^2 \frac{\Lresbar^2 (\Lam^2)^{m - 1} }{\Gamma^2 n^{m}}\\
& \hspace*{1.8in} \cdot \left(\left(\Dspan{\pi, u}{k_{u}}{t_{m-1}, R_{m-1}} \right)^2 +  \left(\Delta^{t_{m-1}, R_{m-1}}_{\max} x^{\pi, u}_{k_u} \right)^2\right)\bigg] \\
&\hspace*{0.2in}\leq  \mathbb{E}\left[24 L_{\max}^2 \frac{ (\Lam^2)^{m} (4q)^m }{ n^{m}}\left(\left(\Dbarspan{\pi, u}{k_{u}} \right)^2 +  \left(\Delta x^{\pi, u}_{k_u} \right)^2\right)\right].
\end{align*}
This yields the term \eqref{eqn::lem-rec-gen-term2}
on the RHS of Lemma~\ref{basic::recursion::gen},
which concludes the proof.
\end{pfof}

%% file: appendix_claim.tex
\subsection{The Claims from Section~\ref{sec:Errt-final-bdd}}
\label{sec::claims}

Recall that $\pi(k) \equiv \pi(k,t)$. The following observation will be useful.
\begin{obs}
\label{obs::path-choice-and-span}
For $u<t$,
$\Delta_{\max}^{t, \emptyset} x^{\pi(k_t), u}_{k_u} \geq \Delta_{\max}^{t, \emptyset} x^{\pi(k_u), u}_{k_u}$,
$\Delta_{\min}^{t, \emptyset} x^{\pi(k_t), u}_{k_u} \leq  \Delta_{\min}^{t, \emptyset} x^{\pi(k_u), u}_{k_u}$,
and thus $\Dspan{\pi(k_u), u}{k_u}{t, \emptyset}\le
\Dspan{\pi(k_t), u}{k_u}{t, \emptyset}$.
\end{obs}
\begin{proof}
This is immediate from the fact that replacing $k_u$ by $k_t$ on path $\pi$
allows more choices of input in calculating the update to $x_{k_u}$.
\end{proof}

\begin{pfof}{Claim \ref{clm::bound::B}}[Bounding Term $B$]
We begin by noting that
\begin{align*}
x^{\pi(k_s), t}_{k_s} &~=~ x^{\pi, t-2q}_{k_s} + \sum_{\substack{t-2q \le u < t\\ \& k_u = k_s}} \Delta x_{k_u}^{\pi(k_s), u} ~=~ x^{\pi, t-2q}_{k_s} + \sum_{\substack{t-2q \le u < t\\ \& k_u = k_s}}\Delta x^{\pi(k_u), u}_{k_u}, \text{ and}\\
x^{\pi(k_t), t}_{k_s} &= x^{\pi, t-2q}_{k_s} + \sum_{\substack{t-2q \le u < t\\ \& k_u = k_s}}\Delta x^{\pi(k_t), u}_{k_u}.
\end{align*}
Note that as $t-2q \le u$, and as the updates $\calU_r$ are fixed
for $r<t-4q$ on the RHS of the following expression, we have
$\Delta x^{\pi(k_u), u}_{k_u} \in \left[\Delta_{\min}^{t, \emptyset} x_{k_u}^{\pi(k_u), u}, \Delta_{\max}^{t, \emptyset} x_{k_u}^{\pi(k_u), u}\right]$.
Also, as $u <t$, and  by Observation~\ref{obs::path-choice-and-span}, the range$\left[\Delta_{\min}^{t, \emptyset} x_{k_u}^{\pi(k_u), u}, \Delta_{\max}^{t, \emptyset} x_{k_u}^{\pi(k_u), u}\right] \subseteq \left[\Delta_{\min}^{t, \emptyset} x_{k_u}^{\pi(k_t), u}, \Delta_{\max}^{t, \emptyset} x_{k_u}^{\pi(k_t), u}\right]$. Further note that $\Delta x^{\pi(k_t), u}_{k_u}$ lies in the same range, $\left[\Delta_{\min}^{t, \emptyset} x_{k_u}^{\pi(k_t), u}, \Delta_{\max}^{t, \emptyset} x_{k_u}^{\pi(k_t), u}\right]$.

Therefore,
\begin{align*}
&\mathbb{E}\bigg[\frac{2}{3n^2} \sum_{\substack{t-2q \le s < t\\ \& s = \prev(t,\ks)}}\sum_{k_t=1}^n \Gamma \cdot  \left(x^{\pi(k_s), t}_{k_s} - x^{\pi(k_t), t}_{k_s} \right)^2\bigg] \\
&\hspace*{0.2in}\le  \frac{2\Gamma} {3 n^2} \cdot \mathbb{E} \bigg[\sum_{\substack{t-2q \le s < t\\ \& s = \prev(t,\ks)}} \sum_{k_t=1}^n \bigg[\sum_{\substack{t-2q \le u < t\\ \& k_u = k_s}} \bigg(\Dspan{\pi(k_t), u}{k_u}{t, \emptyset} \bigg)\bigg]^2\bigg]\\
&\hspace*{0.2in}\le  \frac{2\Gamma \cdot 2q} {3 n^2} \cdot \mathbb{E} \bigg[\sum_{k_t=1}^n \sum_{t-2q \le u < t} \left(\Dspan{\pi(k_t), u}{k_u}{t, \emptyset}\right)^2 \bigg]\comm{by the Cauchy-Schwarz inequality}\\
&\hspace*{0.2in} \le  \frac {4q}{3n} \sum_{u\in [t-2q, t-1] } \G\cdot \left(\Df_u\right)^2 ~=~ \frac{\nuo}{15q}  \sum_{s\in [t-2q, t-1] } \G\cdot \left(\Df_s\right)^2 \comm{recall that $\nu_1= \frac{20q^2}{n}$}
\end{align*}
\end{pfof}

\begin{pfof}{Claim~\ref{clm::bound::A}}[Bounding Term $A$]
We begin by observing:

\begin{align*}
&\mathbb{E}\bigg[ \sum_{\substack{t-2q \le s < t\\ \hspace*{0.2in}  \text{~\&~} s = \prev(t, k_s)}} \sum_{k_t=1}^n
\frac 1 {2n^2\G} \cdot \left(\tilde{g}_{k_s}^{\pi(k_t), s} - \gkspiktt\right)^2\bigg] \\
&\hspace*{1in} \leq \mathbb{E}\bigg[  \sum_{t-2q \le s < t}~\sum_{k_t=1}^n
\frac 1 {2n^2\G} \cdot \left(\tilde{g}_{k_s}^{\pi(k_t), s} - \gkspiktt\right)^2\bigg]. \numberthis \label{clm::bound::A::1}
\end{align*}

Note that $\tilde{g}_{k_s}^{\pi(k_t), s}$ and $\gkspiktt$ are on the same path $\pi(k_t)$. The difference of these two values is bounded by a sum of terms of the form $L_{k_{l_0} k_s} \cdot \big|\Delta x_{k_{l_0}}^{\pi(k_t), l_0} \big| $, where $l_0$ needs to satisfy two conditions:  1. $l_0 \in [s - 2q, \max\{s + q, t- 1\}]$; 2. If  $l_0 \geq t$, $\calU_{l_0}$ commits before $\calU_s$ reads coordinate $x_{k_{l_0}}$. As $t - 2q \leq s < t$, if $l_0$ satisfies these conditions,
$\big|\Delta x_{k_{l_0}}^{\pi(k_t), l_0} \big|  \leq
\max\Big\{\big| \Delta_{\min}^{t, \emptyset} x_{k_{l_0}}^{\pi(k_t), l_0}\big|, \big|\Delta_{\max}^{t, \emptyset} x_{k_{l_0}}^{\pi(k_t), l_0} \big|\Big\} \le
\Dvar{\pi(k_t), l_0}{k_{l_0}}{t, \emptyset}$.
\hide{
where $\Dvar{\pi(k_t), l_0}{k_{l_0}}{t, \emptyset} = \max\left\{\Delta_{\max}^{t, \emptyset} x_{k_{l_0}}^{\pi(k_t), l_0} - \Delta_{\min}^{t, \emptyset} x_{k_{l_0}}^{\pi(k_t), l_0}, \left| \Delta_{\max}^{t, \emptyset} x_{k_{l_0}}^{\pi(k_t), l_0}\right|, \left|\Delta_{\min}^{t, \emptyset} x_{k_{l_0}}^{\pi(k_t), l_0} \right| \right\}$.
}
Thus
\begin{align*}
&\mathbb{E}\bigg[  \sum_{t-2q \le s < t}\sum_{k_t=1}^n
\frac 1 {2n^2\G} \cdot \left(\tilde{g}_{k_s}^{\pi(k_t), s} - \gkspiktt\right)^2\bigg] \\
&\hspace*{0.2in}\leq \mathbb{E}\bigg[  \sum_{t-2q \le s < t} \sum_{k_t = 1}^{n} \frac{1}{n^2 \Gamma}\bigg(\bigg(\sum_{l_0 \in [s - 2q, \max \{s + q, t - 1 \}] \setminus \{s\}} L_{k_{l_0} k_s} \Dvar{\pi(k_t), l_0}{k_{l_0}}{t, \emptyset} \bigg)^2   \\
& \hspace*{3.2in}+ \Big(L_{k_{s} k_s} \Delta x_{k_s}^{\pi(k_t),s}\Big)^2 \bigg)\bigg] \\
&\hspace*{0.2in} \leq \mathbb{E}\bigg[  \sum_{t-2q \le s < t} \frac{1}{n \Gamma}\bigg(\Big(\sum_{l_0 \in [t - 4q, t + q] \setminus \{s\}} L_{k_{l_0} k_s} \Dvar{\pi, l_0}{k_{l_0}}{t, \emptyset}\Big)^2 + \left(L_{k_{s} k_s} \Delta x_{k_s}^{\pi,s}\right)^2 \bigg)  \bigg]\\
&\hspace*{2.0in}\text{(recall that  $\pi = \pi(\kt)$.)}
\end{align*}

Then, using Lemma~\ref{lem::grad::sync::bound::gen} for the first inequality, $\Gamma \geq L_{\max}$ for the second inequality,
and Lemma~\ref{lem::key-rec-bound-full} 
to bound $\G \big(\Df_t \big)^2$
for the third inequality, yields

\begin{align*}
&\mathbb{E}\bigg[  \sum_{\substack{t-2q \le s < t\\ \hspace*{0.2in}  \text{~\&~} s = \prev(t, k_s)}}\sum_{k_t=1}^n
\frac 1 {2n^2\G} \cdot \left(\tilde{g}_{k_s}^{\pi(k_t), \prev(t, k_s)} - \gkspiktt\right)^2\bigg] \\
&\hspace*{0.2in} \leq \frac{\nu_3 \Gamma}{qn}\sum_{t-2q \le s < t}
\sum_{r \in [t - 7q, t+q] \setminus \{s\}}
\left[ \big(\Df_r \big)^2 + \big( \DE_r \big)^2 \right]
+ \sum_{t - 2q \leq s <t} \frac{\nu_4 \Gamma}{n} \left[\big(\Df_s \big)^2 + \big( \DE_s \big)^2 \right] \\
&\hspace*{0.6in} + \sum_{t - 2q < s < t} \frac{L^2_{\max} }{n \Gamma^2} \cdot \Gamma \big(\DE_s \big)^2 \mbox{~~~(as $L_{k_s k_s} \leq L_{\max}$)}\\
&\hspace*{0.2in} \leq \frac{2(\nutr+\nuf)}{n}\cdot \G \sum_{s \in [t - 7q, t+q] \setminus \{t\}}
\left[ \big(\Df_s \big)^2 + \big( \DE_s \big)^2 \right] ~+~ \frac{2\nutr}{n} \cdot \G \left[ \big(\Df_t \big)^2 + \big( \DE_t \big)^2 \right]\\
&\hspace*{0.6in} + \frac \G n \sum_{s \in [t - 2q, t - 1]}\big( \DE_s\big)^2 \\
&\hspace*{0.2in} \leq \frac{2(\nutr + \nuf)}{n}  \cdot \Gamma
\sum_{s \in [t - 7q, t+q] \setminus \{t\}}
\left[ \big(\Df_s \big)^2 + \big( \DE_s \big)^2 \right] + \frac{\G}{n } \sum_{s \in [t - 2q, t - 1]}\big( \DE_s\big)^2\\
&\hspace*{0.6in}+\frac{2\nutr}{n} \cdot \Gamma \Big(\frac{\nuo}{q} +  \frac{\nut}{q}\Big)
\sum_{s\in[t-5q,t+q]\setminus \{t\}}
\left[ \big(\Df_s \big)^2 + \big( \DE_s \big)^2 \right]
+ \frac{ 2\nutr }{n} \cdot \Gamma\big( \DE_t \big)^2.
\end{align*}
\end{pfof}

\begin{pfof}{Claim~\ref{clm::bound::C}} [Bounding Term $C$]
\begin{align*}
\mathbb{E}\bigg[ \frac {\G}{2n^2}  \sum_{\substack{t-2q \le s < t\\ \& s = \prev(t,\ks)}}\sum_{k_t=1}^n
 \left(\Delta \xkspikts \right)^2\bigg]
 & ~\le~\mathbb{E}\bigg[ \frac {\G}{2n}  \sum_{t-2q \le s < t}   \left(\Delta x_{k_s}^{\pi,s}\right)^2\bigg]\\
&\le \frac {\G} {2n} \sum_{t-2q \le s \le t-1}\big(\DE_s\big)^2.
\end{align*}
\end{pfof}

\begin{pfof}{Claim~\ref{clm::bound::D}}[Bounding Term $D$]
We bound the term\\
$\mathbb{E}\Big[ \frac 2{3\G n^2} \sum_{k=1}^n \sum_{\kt =1}^n \cdot \left(\gkpikt - \gkpiktt \right)^2\Big]$.
We will use the term $g_k^{\pi(k), t - 2q}$ as an intermediary
to allow us to compare values on two different paths, as follows.
\begin{align}
\frac 23 \left(\gkpikt - \gkpiktt \right)^2 & \le \frac 43\left(\gkpikt - g_k^{\pi(k), t - 2q} \right)^2 + \frac 43 \left(g_k^{\pi(k), t - 2q} - \gkpiktt \right)^2
\nonumber\\
& \le
\frac 43\left(\gkpikt - g_k^{\pi(k), t - 2q} \right)^2 + \frac 43 \left(g_k^{\pi(k_t), t - 2q} - \gkpiktt \right)^2,
\label{eqn::grad-bound-using-sync-intermed}
\end{align}
as $g_k^{\pi(k), t - 2q} = g_k^{\pi(k_t), t - 2q}$ since the gradients
at $x^{\pi,t- 2q}$ do not depend on update $\calU_t$.

For the first term on the RHS of \eqref{eqn::grad-bound-using-sync-intermed}, we apply Lemma~\ref{lem::grad::sync::bound::gen} for the third inequality,
and then apply Lemma~\ref{lem::key-rec-bound-full} to bound the $\left(\mathcal{D}_t\right)^2$ term which was generated by Lemma~\ref{lem::grad::sync::bound::gen}, as follows.
\begin{align*}
&\mathbb{E}\bigg[\frac{1}{n} \sum_{k = 1}^n \frac 43\left(\gkpikt - g^{\pi(k), t - 2q}_{k} \right)^2\bigg]  \leq \mathbb{E}\bigg[\frac{1}{n} \sum_{k = 1}^n \frac 43 \Big(\sum_{t-2q\le l_0 \le t-1}L_{k_{l_0}  k}\Delta x^{\pi(k),l_0}_{k_{l_0}}\Big)^2\bigg] \\
&~~~~ \leq \mathbb{E}\bigg[\frac{1}{n} \sum_{k = 1}^n \frac 43 \Big(\sum_{t-2q\le l_0 \le t-1}L_{k_{l_0} k}\Delta_{\mathsf{var}} x^{\pi(k),l_0}_{k_{l_0}}\Big)^2\bigg] \\
&~~~~ \leq   \frac{4\nutr \Gamma^2}{3q}
\sum_{s \in [t - 7q, t+q] \setminus \{t\}}
\left[ \left(\Df_s \right)^2 + \big( \DE_s \big)^2 \right]
+ \frac {4\nuf\Gamma^2 }{3} \left[\big(\Df_t \big)^2 + \big( \DE_t \big)^2 \right] \\
&~~~~\leq \frac{ 4\nutr\Gamma^2}{3q}
\sum_{s \in [t - 7q, t+q] \setminus \{t\}}
\left[ \big(\Df_s \big)^2 + \big( \DE_s \big)^2 \right]\\
&~~~~~~~~+\frac{ 4\nuf\Gamma^2}{3} \cdot \left( \frac{\nuo}{q} + \frac {\nut}{q}\right)
\sum_{s\in[t-5q,t+q]\setminus \{t\}}
\left[ \big(\Df_s \big)^2 + \big( \DE_s \big)^2 \right]
+ \frac{4\nuf \Gamma^2}{3} \cdot \big( \DE_t \big)^2. \numberthis \label{eqn::grad-bound-using-sync-intermed::A}
\end{align*}

For the second term on the RHS of \eqref{eqn::grad-bound-using-sync-intermed}, for $t-2q \le s \le t-1$, as $\Delta x_{k_s}^{\pi(k_t), s}$ and $\Delta_{\max}^{t, \{t\}} x_{k_s}^{\pi(k_t), s}$ are in the range $\left[\Delta_{\min}^{t, \emptyset} x^{\pi(k_t), s}_{k_s}, \Delta_{\max}^{t, \emptyset} x^{\pi(k_t), s}_{k_s}\right]$, by recentering and by the Cauchy-Schwarz inequality,

\begin{align*}
\frac 43 \left(g^{\pi(k_t), t - 2q}_k - \gkpiktt \right)^2 &= \frac{4}{3} \bigg(\sum_{s \in [t - 2q, t-1]} L_{k_s, k} \Delta x_{k_s}^{\pi(k_t), s} \bigg)^2 \\
&  = \frac{16}{3}q \sum_{s \in [t - 2q, t - 1]} L^2_{k_s, k} \left(\left( \Dspan{\pi(k_t), s}{k_s}{t, \emptyset} \right)^2+ \left( \Delta_{\max}^{t, \{t\}} x_{k_s}^{\pi(k_t),s}\right)^2\right).
\end{align*}

This yields
\begin{align*}
& \mathbb{E}\bigg[ \frac{1}{n^2} \sum_{k = 1}^n \sum_{k_t = 1}^n \frac 43 \left(g^{\pi(k_t), t - 2q}_k - \gkpiktt \right)^2\bigg] \\
 &\hspace*{0.35in}\leq \mathbb{E}\bigg[ \frac{1}{n^2}\sum_{k_t = 1}^n   \sum_{k = 1}^n \frac{16}{3}q \sum_{s \in [t - 2q, t - 1]} L^2_{k_s, k} \left(\left(\Dspan{\pi(k_t), s}{k_s}{t, \emptyset}\right)^2 
 + \left( \Delta_{\max}^{t, \{t\}} x_{k_s}^{\pi(k_t), s}\right)^2 \right) \bigg] \\
 &\hspace*{0.35in}\leq \mathbb{E}\bigg[ \frac{1}{n^2}\sum_{k_t = 1}^n   \frac{16}{3}q \sum_{s \in [t - 2q, t - 1]} \Lresbar^2 \left(\left(\Dspan{\pi(k_t), s}{k_s}{t, \emptyset}\right)^2
 + \left( \Delta_{\max}^{t, \{t\}} x_{k_s}^{\pi(k_t), s}\right)^2 \right) \bigg].
\end{align*}

By Lemma~\ref{Dmax-defn-works}, $ \Big(\Dspan{\pi(k_t), s}{k_s}{t, \emptyset}\Big)^2 \leq  \Big(\Dbarspan{\pi(k_t), s}{k_s}\Big)^2 $ and $\Big( \Delta_{\max}^{t, \{t\}} x_{k_s}^{\pi(k_t)}\Big)^2 \leq 2\Big(\Dbarspan{\pi(k_t), s}{k_s}\Big)^2 + 2 \Big( \Delta x_{k_s}^{\pi(k_t), s} \Big)^2$. Thus,
\begin{align*}
  \mathbb{E}\bigg[ \frac{1}{n^2} \sum_{k = 1}^n \sum_{k_t = 1}^n \frac 43 \left(g^{\pi(k_t), t - 2q}_k - \gkpiktt \right)^2\bigg]
&\leq \frac{16 q \Lresbar^2}{n} \sum_{s \in [t - 2q, t - 1]} \Big(\big(\Df_s\big)^2 + \big(\Delta_s^X\big)^2 \Big)\\
& \leq  \frac {2\nut \Gamma^2}{3q} \sum_{s \in [t - 2q, t - 1]} \Big(\big(\Df_s\big)^2 + \big(\Delta_s^X\big)^2 \Big). \numberthis \label{eqn::grad-bound-using-sync-intermed::B}
\end{align*}

Combining \eqref{eqn::grad-bound-using-sync-intermed}, \eqref{eqn::grad-bound-using-sync-intermed::A} and \eqref{eqn::grad-bound-using-sync-intermed::B} yields the result.
\end{pfof}

\begin{clm}
\label{clm::convert-nu-bound-to-r-bound}
The term $G$ in~\eqref{eqn:nu-coeff-bound}
is bounded by $\frac{1}{q} \left[ \frac{2r}{3} + \frac{3r^2}{1280} + \frac{9r^3}{25600}  + \frac{3r^2}{1 - r}  + \frac{r^3}{426(1 - r)} + \frac{r^4}{2844(1-r)} \right]$.
\end{clm}

\begin{proof}
Recall that $r =\frac{160 q^2}{n}\cdot \left( \frac{L_{\overline{res}}^2}{\Gamma^2} + 1\right)$
(see the second paragraph of Section~\ref{sec::grad-bounds}).
As stated before Lemma~\ref{lem::key-rec-bound-full},
$\nu_1 = \frac{20 q^2}{n}$, $\nu_2 = \frac{24 q^2 \Lresbar^2}{n\Gamma^2}$,
and as stated before Lemma~\ref{lem::grad::sync::bound::gen},
$\nu_3 = \frac{3}{16} \left( \frac{ r^2 }{ 1 - r} +  r\right)$, $\nu_4 = \frac{6r}{1 - r}$.
Also, note that $\frac{1}{n} \leq \frac{q}{n}$ as $q \geq 1$, $\nu_1 + \nu_2 \leq \frac{3r}{20}$, and $\nu_2 + \frac{24 q^2}{n} = \frac{3r}{20}$.
\begin{align*}
&\max \left\{\frac{ \nu_1}{15q}, \frac{3}{2n}\right\} + \frac{8\nu_2}{11q} +  \frac{2(\nutr + \nu_4)}{n} + \frac{7 \nu_3}{3q} + \left(\frac{2\nu_3}{nq} +  \frac{7 \nu_4}{3q} \right)\left(\nu_1 + \nu_2\right)\\
& \hspace*{0.2in}\leq  \frac{1}{q} \left[ \max \left\{ \frac{ \nu_1}{15}, \frac{3q^2}{2n}
\right\}
+ \frac{8 \nu_2}{11} + \left(\frac{3}{4}\left(\frac{ r^2}{ 1 - r} + r \right) + \frac{24r}{1 - r}\right)\frac{r}{320} + \frac{7}{16} \left( \frac{r^2}{1 - r} + r \right) \right. \\
&\hspace*{0.8in}\left.+ \left(\frac{3}{4} \left(\frac{ r^2 }{1 - r} + r \right) \frac{r}{320} + \frac{14 r}{1-r}\right) \frac{3r}{20} \right]\\
&\hspace*{0.2in} \leq  \frac{1}{q} \left[r \left(\frac{1}{9} + \frac{7}{16}\right) + r^2 \frac{3}{1280} + r^3 \frac{9}{25600} + \frac{r^2}{1-r} \left( \frac{3}{40} + \frac{7}{16} + \frac{21}{10} \right)  \right.\\
&\hspace*{0.8in}\left. + \frac{r^3}{1 - r} \left( \frac{3}{1280}\right) + \frac{r^4}{1-r}\left( \frac{9}{25600}\right) \right] \\
&\hspace*{0.2in} \leq \frac{1}{q} \left[ \frac{2r}{3} + \frac{3r^2}{1280} + \frac{9r^3}{25600}  + \frac{3r^2}{1 - r}  + \frac{r^3}{426(1 - r)} + \frac{r^4}{2844(1-r)} \right].
\end{align*}
\end{proof}

%% file: app-gen-cvgce-thm.tex
 \subsection{Proofs for the Amortized Analysis, Section~\ref{sec::amortization}:  Theorem~\ref{thm:meta-new::re} and Lemma~\ref{lem::ApAmProg}}
\label{app:good-progress}

The following lemma is key to the demonstration of progress in both the strongly convex and convex cases.

$$
\text{For any $t \ge 1$, we define:}\hspace*{0.65in}\PRG(t) \triangleq \sum_{k=1}^n ~\hWk(\nabla_k f(x^{t}),x_{k}^{t}).\hspace*{0.65in}
$$

We will use the following lemma from~\cite[Lemmas 4,6]{richtarik2014iteration}.
The version we present here is slightly different from the one in~\cite{richtarik2014iteration}, but the proofs are essentially the same.

\begin{lemma}[{\cite[Lemmas 4,6]{richtarik2014iteration}}]\label{lem:good-progress}
	~\\
	(a) Suppose that $f,F$ are strongly convex with parameters $\mu_f,\mu_F > 0$ respectively,
	and suppose that $\G \ge \muf$. Then
	\[\PRGe(t) \geq \frac{\mu_F}{\mu_F + \G - \mu_f}\cdot F(x^t).\]
	(b) Suppose that $f,F$ are convex functions. Suppose that $\calR := \min_{x^*\in X^*} \|\rjc{x^t} - x^*\| < \infty$. Then
	\[\PRGe(t) \geq \min\Big\{\frac 12, \frac{F(x^t)}{2\G \calR^2}\Big\}\cdot F(x^t).\]
\end{lemma}
\begin{pfof}{Theorem~\ref{thm:meta-new::re}} %
We begin by showing (i).
By assumption (c) and Lemma \ref{lem:good-progress}(a),
\begin{align*}
 H(t) - H(t+1)
&\ge
 \Big[\frac \alpha n \cdot \mathbb{E}\big[\PRG(t)\big] + \frac{\beta}{n} \cdot \Ap(t)\Big]\\
&\ge
 \Big[\frac \alpha n \cdot \frac{\mu_F}{\mu_F + \G - \mu_f} \cdot  \mathbb{E}\big[F(x^{t})\big]
 + \frac{\beta}{n} \cdot \Ap(t) \Big]
\ge \delta \cdot H(t),
\end{align*}
where $\delta \triangleq \min\big\{ \frac \alpha n \cdot \frac{\mu_F}{\mu_F + \G - \mu_f}, \frac{\beta}{n}\big\}$.

Thus $H(t+1) \le \left( 1-\delta  \right) H(t)$ for all $t\ge 1$.
Iterating the above inequality $T$ times yields
$H(T+1) \le \left( 1-\delta  \right)^{T} H(1).$

To finish the proof note that since $\Ap(1) = 0$  and $\Am(1)\ge 0$, $H(1) \le F(x^1)$.

\smallskip

Now we show (ii).
By the second assumption, Lemma \ref{lem:good-progress} and the fact that $\mathbb{E}[X^2] \geq \mathbb{E}[X]^2$,
\begin{align*}
&H(t) - H(t+1)\\
&\hspace*{0.4in}\ge
 \big[ \frac \alpha n \cdot \mathbb{E}\big[\PRG(t)\big] + \frac{\beta}{n} \cdot \Ap(t)\big]
\ge
 \big[ \frac \alpha n \cdot \min\left\{\frac 12, \frac{\mathbb{E}[F(x^{t})]}{2\G\calR ^2}\right\}\cdot \mathbb{E}[F(x^{t})] + \frac{\beta}{n} \cdot \Ap(t) \big].
\end{align*}

We consider two cases:

\begin{itemize}
	\item If $\mathbb{E}[F(x^{t})] \le \Ap(t)$, then $\Ap(t) \ge \frac{H(t)}{2}$, thus
	\begin{align*}
	\frac \alpha n \cdot \min\left\{\frac 12, \frac{\mathbb{E}[F(x^{t})]}{2\G\calR ^2}\right\}\cdot \mathbb{E}[F(x^{t})] ~+~ \frac{\beta}{n} \cdot \Ap(t)
	&~\ge~ \frac{\beta}{2n}\cdot H(t).
	\end{align*}
	\item If $\mathbb{E}[F(x^{t})] > \Ap(t)$, then $\mathbb{E}[F(x^{t})] > \frac{H(t)}{2}$, thus
\[
\frac \alpha n \cdot \min\left\{\frac 12, \frac{\mathbb{E}[F(x^{t})]}{2\G\calR ^2}\right\}\cdot\mathbb{E}[F(x^{t})] + \frac{\beta}{n} \cdot \Ap(t)
~>~ \frac \alpha {2n} \cdot \min\left\{\frac 12, \frac{H(t)}{4\G~\calR ^2}\right\}\cdot H(t).
\]
\end{itemize}
Since $H$ is a decreasing function, $H(t)\le H(1) \le F(x^1)$.
Thus, unconditionally,
\begin{align*}
&\frac \alpha n \cdot \min\left\{\frac 12, \frac{\mathbb{E}[F(x^{t})]}{2\G\calR ^2}\right\}\cdot \mathbb{E}[F(x^{t})] + \frac{\beta}{n} \cdot \Ap(t)
\ge \min \left\{ \frac{\beta}{2n}, \frac{\alpha}{4n}, \frac{\alpha\cdot H(t)}{8n\G \calR ^2} \right\} \cdot H(t)\\
& \hspace*{2in} \ge \min \left\{ \frac{\beta}{2n\cdot F(x^1)}, \frac{\alpha}{4n\cdot F(x^1)}, \frac{\alpha}{8n\G\calR ^2} \right\} \cdot H(t)^2.
\end{align*}
Note that the term $\min \left\{\frac{\beta}{2n~F(x^1)}, \frac{\alpha}{4n F(x^1)}, \frac{\alpha}{8n\G \calR ^2} \right\}$
is independent of $t$. We denote it by $\varepsilon$.
Thus,
$H(t) - H(t+1) \ge \varepsilon~ H(t)^2$.
Dividing both sides by $H(t) \cdot H(t+1)$ yields
\[
\frac{1}{H(t+1)} - \frac{1}{H(t)} \ge \varepsilon \frac{H(t)}{H(t+1)} \ge \varepsilon.
\]
\[
\text{Iterating the above inequality $T$ times yields}~~~~\frac{1}{H(T+1)} - \frac{1}{H(1)} \ge \varepsilon T,\hspace*{1in}
\]
\[
\text{and hence}~~~~\frac{1}{H(T+1)} \ge \varepsilon T + \frac{1}{H(1)} ~\ge~ \varepsilon T + \frac{1}{F(x^1)}. \comm{since $H(1) \leq F(x^1)$}\hspace*{1in}
\]
(ii) follows by taking the reciprocal on both sides of the above inequality.
%
\end{pfof}

%% file: app-amort-section-proofs.tex

\begin{pfof}{Lemma~\ref{lem::ApAmProg}}
By calculation,
\begin{align*}
 A^+(t) - A^+(t+1)&
 = \sum_{s=t-7q}^{t-1}\frac {1}{\left(1 - \frac 1{3n}\right)} \left[ c \big( \Df_s \big)^2 + c \big( \DE_s\big)^2\right]\\
 &\hspace*{0.4in} +
\sum_{s=t-7q+1}^{t-1} \sum_{v=t+1}^{s+7q} \frac 1{3n} \frac {1}{\left(1 - \frac 1{3n}\right)^{v-t+1}} \left[ c \big( \Df_s \big)^2 + c \big( \DE_s\big)^2\right]\\
&\hspace*{0.4in}- \sum_{v=t+1}^{t+7q} \frac {1}{\left(1 - \frac 1{3n}\right)^{v-t}} \left[ c \big( \Df_t \big)^2 + c \big( \DE_t\big)^2\right]\\
&=\frac 1{3n}A^+(t) + \sum_{s=t-7q}^{t-1} \left[ c \big( \Df_s \big)^2 + c \big( \DE_s\big)^2\right] \\
&\hspace*{0.4in}- \sum_{v=t+1}^{t+7q} \frac {1}{\left(1 - \frac 1{3n}\right)^{v-t}} \left[ c \big( \Df_t \big)^2 + c \big( \DE_t\big)^2\right]\\
\end{align*}
\begin{align*}
A^-(t+1) -A^-(t)& = \sum_{v=t+1}^{t+q} \left[ c \left( \Df_v \right)^2 + c \left( \DE_v\right)^2\right]- \sum_{s=t-q}^{t-1} \left[ c \left( \Df_t \right)^2 + c \left( \DE_t\right)^2\right].
\end{align*}
Therefore
\begin{align*}
&\Big[\Big(1 - \frac{1}{3n} \Big) A^+(t) - A^-(t)\Big] - \Big[ A^+(t+1) - A^-(t+1) \Big] \\
&\hspace*{0.2in} = \sum_{s=t-7q}^{t-1} \left[ c \big( \Df_s \big)^2 + c \big( \DE_s\big)^2\right] - \sum_{v=t+1}^{t+7q} \frac {1}{\left(1 - \frac 1{3n}\right)^{v-t}} \left[ c \big( \Df_t \big)^2 + c \big( \DE_t\big)^2\right] \\
&\hspace*{0.4in} + \sum_{v=t+1}^{t+q} \left[ c \big( \Df_v \big)^2 + c \big( \DE_v\big)^2\right]- \sum_{s=t-q}^{t-1} \left[ c \big( \Df_t \big)^2 + c \big( \DE_t\big)^2\right] \\
&\hspace*{0.2in } = \sum_{s \in [ t - 7q, t + q] \setminus \{t\}} \left[ c \big( \Df_s \big)^2 + c \big( \DE_s\big)^2\right] \\
&\hspace*{0.4in }- \bigg(\sum_{v = t+1}^{t + \yxt{7}q} \frac {1}{\left(1 - \frac 1{3n}\right)^{v-t}} + q \bigg) \left[ c \big( \Df_t \big)^2 + c \big( \DE_t\big)^2\right].
 \numberthis \label{eqn::A::plus}
\end{align*}
In order to achieve~\eqref{ineq::final::H},
we compare the coefficients of each of the terms
$c\left( \Df_t \right)^2$,
$c \left( \DE_t\right)^2$, $c\left( \Df_s \right)^2$,
$c \left( \DE_s\right)^2$ in \eqref{ineq::final::H}
and \eqref{eqn::A::plus}.
Since $c = \varpi +  (\gamma + \varpi) \left(\frac{\nuo}{q} + \frac{\nut}{q}\right)\Gamma$ the coefficient of  $\left( \Df_s \right)^2$ and $\left( \DE_s\right)^2$ in \eqref{eqn::A::plus} is at least as big as in 
\eqref{ineq::final::H}.
Therefore, it suffices to have the coefficients of
$\left( \Df_t \right)^2$ and $\left( \DE_t\right)^2$ satisfy the following inequalities.
\begin{align*}
\gamma &\ge \frac {c}{\Gamma}
 \cdot
\bigg[ \sum_{i = 1}^{7q} \frac{1}{\left(1 - \frac{1}{3n}\right)^i} +  q\bigg]\\
 &= \frac 1q
\left[ q \varpi
+(\gamma + \varpi)\left(\nuo + \nut\right) \right] \cdot
\bigg[ 3n\bigg( \frac {1}{\left(1 - \frac 1{3n} \right)^{7q + 1}} - \frac {1}{1 - \frac 1{3n}} \bigg) +  q\bigg]
\end{align*}
\begin{align*}
\text{and}~\hspace*{0.3in} \varrho - \varpi & \ge \frac{1}{q} \left[
 q \varpi
+ (\gamma + \varpi) (\nuo+ \nut)\right]
 \cdot \bigg[ 3n\bigg( \frac {1}{\left(1 - \frac 1{3n} \right)^{7q + 1}} - \frac {1}{1 - \frac 1{3n}} \bigg) +  q\bigg].
\end{align*}

If $7q < 3n - 2$, then by the fact that $(1 + x)^s \leq 1 + \frac{sx}{1 - (s-1)x}$ for any $x < \frac{1}{s - 1}$ and $s \geq 1$,
\begin{align*}
3n \bigg(\frac{1}{\left( 1 - \frac{1}{3n} \right)^{7q + 1}} - \frac{1}{1 - \frac{1}{3n}} \bigg) + q
&\leq 3n\bigg[\Big( 1 + \frac {1}{3n-1}\Big)^{7q+1} - \frac {3n}{3n-1}\bigg] + q \\
& \leq 3n \bigg[1 + \frac {(7q + 1) \left(\frac {1}{3n-1}\right)} {1 -\frac{7q}{3n-1}} -1 \bigg] + q \\
&\leq 3n \left( \frac{7q + 1}{3n - 1 - 7q} \right) + q \leq 14q + 2 +  q ,
\end{align*}
if $\frac{3n}{3n-1-7q}\le 2$, i.e., if $7q \le n - 1$.

\begin{align*}
\text{Then it suffices that}\hspace*{0.3in}&\gamma  \ge \frac 1q
\left[ q \varpi + ( \gamma + \varpi)(\nuo + \nut)\right](15q + 2)\hspace*{2in}\\
\text{and}\hspace*{1.15in}&\varrho - \varpi \ge \frac 1q
\left[ q \varpi + (\gamma + \varpi)(\nuo+ \nut) \right]  (15q+2).\hspace*{2in}\end{align*}

Recall that
$\nuo = \frac{20q^2}{n} \le \frac {r}{8}$ and
$\nu_2 = \frac{24 q^2 \Lresbar^2}{n\Gamma^2} \leq \frac{r}{6}$,
$\varrho = \frac{1}{8} - \frac{15 r}{1 - r}$
(see the first line of Section~\ref{sec::amortization}),
and
$\varpi = \frac{1}{q} \left[ \frac{2r}{3} + \frac{3r^2}{1280} + \frac{9r^3}{25600}  + \frac{3r^2}{1 - r}  + \frac{r^3}{426(1 - r)} + \frac{r^4}{2844(1-r)} \right]$.
One choice of values that suffices is $\gamma= \varrho - \varpi$ and $r \leq \frac 1{225}$.
\end{pfof}

%% file: parameter.tex
\section*{A Table of Definitions and Parameters}


\renewcommand{\arraystretch}{1.3}

\begin{tabular}{|c|c|c|}
\hline
\textbf{Notation / } & \textbf{Definition / Description} & \textbf{First}\\
\textbf{Parameter} && \textbf{Appearance}\\
\hline
$F:\rr^n\ra \rr$ & $F(x) = f(x) + \sum_{k=1}^n \Psik(x_k)$ & \\[-0.1em]
$f:\rr^n\ra \rr$ & $F$ is the convex function we & Abstract\\[-0.2em]
$\Psik: \rr \ra \rr$ & want to minimize. & \\
\hline
$e_j$ & the unit vector along coordinate $j$ & \multirow{7}{*}{Definition~\ref{def:Lipschitz-parameters}}\\
\cline{1-2}
$L_{jk}$ & $|\nabla_k f(x+r\cdot e_j) - \nabla_k f(x)| ~\le~ \Ljk\cdot |r|$ & \\
\cline{1-2}
$\Lres$ & $||\nabla f(x+r\cdot e_j) - \nabla f(x)|| ~\le~ \Lres\cdot |r|$ & \\
\cline{1-2}
&&\\ [-1.1em]
$\Lresbar$ & $\Lresbar \triangleq \max_k \left(\sum_{j=1}^n (\Lkj)^2\right)^{1/2}$ & \\ [0.4em]
\cline{1-2}
\multirow{2}{*}{$\Lmax$} & $\Lmax ~\triangleq~ \max_{j,k} \Ljk$ & \\[-0.1em]
& if $f$ is twice differentiable, 
then $\Lmax = \max_j L_{jj}$ & \\
\hline
\multirow{3}{*}{$\muf,\muF$} & strong convexity parameters of $f,F$ & \\
& $f(y) - f(x) \ge \langle \nabla f(x), y-x \rangle + \frac 12 \muf ||y-x||^2$ & Definition~\ref{def::str-conv}\\
& $F(y) - F(x) \ge \langle \nabla F(x), y-x \rangle + \frac 12 \muF ||y-x||^2$ & \\
\hline
$\mathcal{U}_t$ & the $t$-th update (in the SCC order)& \\
\cline{1-2}
$\G$ & parameter used in the update rule & \\
\cline{1-2}
$W_j(d,g,x)$ & $W_j(d,g,x) \triangleq gd + \G d^2 / 2 +\Psij(x+d) - \Psij(x)$ & \\
\cline{1-2}
&&\\ [-1.1em]
$\hWj(g,x)$ & $\hWj(g,x) \triangleq \max_d W(d,g,x)$ & Near\\ [0.4em]
\cline{1-2}
&&\\ [-1.1em]
$\hdj(g,x)$ & $\hdj(g,x) \triangleq \argmax_d W(d,g,x)$ & Eqn.~\eqref{eq:update-rule}\\ [0.4em]
\cline{1-2}
&&\\ [-1.2em]
asynchronous & $x_j^{t+1} \la x_j^t + \hdj(\tg_j,x_j^t)$, & \\
update rule & where $\tg_j$ is the (inaccurate) measured gradient. & \\
\hline
\multirow{3}{*}{$q$} & the updates that can interfere with the update& \\ [-0.2em]
& at time $t$ in the ST order are those that & Assumption~\ref{assume:SACD-q}\\[-0.2em]
& commit at times $t+1,t+2,\cdots,t+q$ &\\
\hline
$\Delta x_{k_t}^{\pi, t}$ & the increment computed by the $t$-th update on path $\pi$ & Near Lemma~\ref{clm:same-older-value} \\ [0.4em]
\hline
$k_t$ & the coordinate that is updated at time $t$ & Beginning\\
\cline{1-2}
\multirow{2}{*}{$g_k^t$} & $g_k^t \triangleq \nabla_k f(x^t)$ & of\\
& accurate gradient along coordinate $k$ at time $t$ & Section~\ref{sect:keyidea}\\
\hline
{$\pi$} & a root-to-leaf path 
in the branching tree & \\
\cline{1-2}
\multirow{2}{*}{$\pi(k,t)$} &  the root-to-leaf path with time $t$& Beginning of\\[-0.2em]
&  coordinate on path $\pi$ replaced by coordinate $k$& Section~\ref{sect:simple}\\
\cline{1-2}
$\pi(\kt,t)$ & $\pi(\kt,t) = \pi$ & \\
\hline
\multirow{2}{*}{$\bullet^{\pi,t}$} & for any variable $\bullet$, $\bullet^{\pi,t}$ denotes its value & \\[-0.2em]
& at time $t$ along the path $\pi$ & \\ 
\hline
\end{tabular}

\renewcommand{\arraystretch}{1.3}

\noindent \begin{tabular}{|c|c|c|}
\hline
\textbf{Notation / } & \textbf{Definition / Description} & \textbf{First}\\
\textbf{Parameter} && \textbf{Appearance}\\
\hline
\multirow{2}{*}{$A^{\pi, u}$} & $A^{\pi, u} = \{ r | u - 4q \le r < u - 2q \text{ and } \mathcal{U}_r \text{ has committed}$ &
\multirow{8}{*}{Section~\ref{sec::addnl-notation}} \\[-0.1em]
&~~~~$\text{before some } \mathcal{U}_p \text{ for } p \leq u - 4q-1$\} &
\\ [0.2em]
\cline{1-2}
&&\\ [-1.2em]
$\Delta_{\max}^{u, R,S} x_{k_s}^{\pi, s}$ & see the table in the next page & \\
\cline{1-2}
&&\\ [-1.2em]
$\Delta_{\max}^{u, R} x_{k_s}^{\pi, s}$& $\Delta_{\max}^{u, R} x_{k_s}^{\pi, s} ~~\triangleq~~ \max_S \Delta_{\max}^{u, R,S} x_{k_s}^{\pi, s}$ & \\ [0.4em]
\cline{1-2}
&&\\ [-1.2em]
$\Dspan{\pi, s}{k_s}{u, R}$& $ \Dspan{\pi, s}{k_s}{u, R}~~\triangleq~~ \Delta_{\max}^{u, R} x_{k_s}^{\pi, s}- \Delta_{\min}^{u, R} x_{k_s}^{\pi, s}$ &\\ [0.4em]
\cline{1-2}
&&\\ [-1.2em]
$\Delta^{u}_{\max} x^{\pi, s}_{k_s}$& $ \Delta^{u}_{\max} x^{\pi, s}_{k_s}~~\triangleq~~ \Delta^{u, \emptyset}_{\max} x^{\pi, s}_{k_s}$ & \\ [0.4em]
\hline
&&\\ [-1.2em]
$\Dvar{\pi, s}{k_s}{u, R}$& $ \Dvar{\pi, s}{k_s}{u, R}~~\triangleq~~ \max \big\{\Dspan{\pi, s}{k_s}{u, R}, \big|\Delta_{\max}^{u, R} x_{k_s}^{\pi, s}\big|,  \big|\Delta_{\min}^{u, R} x_{k_s}^{\pi, s} \big|\big\}$ &\multirow{17}{*}{Section~\ref{sec::addnl-notation}}\\ [0.4em]
\cline{1-2}
&&\\ [-1.2em]
$\Dmaxbar^R \xkspis$ & $\Dmaxbar^R \xkspis ~\triangleq~ \max_{s-q\le t \le s} \Delta_{\max}^{t, R} \xkspis$ &\\ [0.4em]
\cline{1-2}
&&\\ [-1.2em]
$\DbarspanR{\pi, s}{k_s}$& $ \DbarspanR{\pi, s}{k_s}~~\triangleq~~ \Dmaxbar^R x_{k_s}^{\pi, s}- \Dminbar^R x_{k_s}^{\pi, s}$ &\\ [0.4em]
\cline{1-2}
&&\\ [-1.2em]
$\DbarvarR{\pi, s}{k_s}$& $ \DbarvarR{\pi, s}{k_s}~~\triangleq~~ \max \big\{\DbarspanR{\pi, s}{k_s}, \big|\Dmaxbar^R x_{k_s}^{\pi, s}\big|,  \big|\Dminbar^R x_{k_s}^{\pi, s} \big|\big\}$ &\\ [0.4em]
\cline{1-2}
&&\\ [-1.2em]
$\Dmaxbar \xkspis$ & $\Dmaxbar \xkspis ~\triangleq~ \Dmaxbar^{\emptyset} \xkspis$ &\\ [0.4em]
\cline{1-2}
&&\\ [-1.2em]
$\Dbarspan{\pi, s}{k_s}$& $ \Dbarspan{\pi, s}{k_s}~~\triangleq~~ \Dbarspanalt{\emptyset} \xkspis $ &\\ [0.4em]
\cline{1-2}
&&\\ [-1.2em]
$\Dbarvar{\pi, s}{k_s}$& $ \Dbarvar{\pi, s}{k_s}~~\triangleq~~ \Dbarvaralt{\emptyset} \xkspis$ &\\ [0.4em]
\cline{1-2}
\hide{&&\\ [-1.2em]
$\gDelmaxkspis$ & the value of $\gkspis$ used to evaluate $\Dmaxbar \xkspis$ & \\ [0.4em]
\cline{1-2}}
&&\\ [-1.2em]
$\widetilde{g}^{u, R, S, \pi, s}_{\max, k_s}$ & see the next table & \\ [0.4em]
\cline{1-2}
&&\\ [-1.2em]
$\widetilde{g}^{u, R, \pi, s}_{\max, k_s}$ & $\widetilde{g}^{u, R, \pi, s}_{\max, k_s} = \max_{S \subseteq [u - 4q, u + q] \setminus (A^{\pi, u} \cup \{s\})}  \widetilde{g}^{u, R, S, \pi, s}_{\max, k_s}$& \\ [0.4em]
\cline{1-2}
&&\\ [-1.2em]
$\overline{g}^{R, \pi, s}_{\max, k_s}$ & $\overline{g}^{R, \pi, s}_{\max, k_s} = \max_{s - q \leq u \leq s}  \widetilde{g}^{u, R, \pi, s}_{\max, k_s} $& \\ [0.4em]
\cline{1-2}
&&\\ [-1.2em]
$\overline{g}^{\pi, s}_{\max, k_s}$ & $\overline{g}^{\pi, s}_{\max, k_s} =  \overline{g}^{ \emptyset, \pi, s}_{\max, k_s} $& \\ [0.4em]
\cline{1-2}
&&\\ [-1.2em]
$\overline{g}_{\mathsf{span},k_s}^{\pi,s}$ & $\overline{g}_{\mathsf{span},k_s}^{\pi,s}$
~$\triangleq$~
$\overline{g}_{\max,k_s}^{\pi,s} - \overline{g}_{\min,k_s}^{\pi,s}$& \\ [0.4em]
\cline{1-2}
\hide{&&\\ [-1.2em]
$\gmaxkspis$ & $\gmaxkspis \triangleq \max_{s-q\le u \le s} \gmaxkspius$ & \\ [0.4em]
\cline{1-2}
&&\\ [-1.2em]
\multirow{2}{*}{$g_{k}^{\mathcal{S},\pi,t}$} &  gradient in the direction $x_k$
at the point $x^{\pi,t}$ with the &\\
& updates from time $t-2q+1$ to $t-1$
performed  sequentially & \\ [0.4em]
\hline
\multirow{2}{*}{$\gktSpikt$} &  gradient in the direction  $k_s$ at the point $x^{\pi(k_t), s}$ with the & \multirow{2}{*}{Claim \ref{clm::bound::A}}\\
& updates from time $t - 2q+1$ to $t$ performed sequentially   & \\ [0.4em]}
\hline
&&\\ [-1.2em]
$\left(\Df_t \right)^2$ & $\left(\Df_t \right)^2 \triangleq \mathbb{E}\big[\big( \Dmaxbar \xktpit - \Dminbar \xktpit \big)^2\big]$ & \multirow{2}{*}{Section~\ref{sec::addnl-notation}}\\ [0.4em]
\cline{1-2}
&&\\ [-1.2em]
$\left( \DE_t \right)^2$ & $\left( \DE_t \right)^2 \triangleq \mathbb{E}\big[\big( \Delta  \xktpit \big)^2\big]$ & \\ [0.4em]
\hline
\hide{&&\\ [-1.2em]
$\Dvar{\pi, s}{k_s}{t, R}$& $ \Dvar{\pi, s}{k_s}{t, R}~~\triangleq~~\max\Big\{\Dspan{\pi, s}{k_s}{t, R}, \big|\Delta_{\max}^{t, R} x_{k_s}^{\pi, s}\big|, \big|\Delta_{\min}^{t, R} x_{k_s}^{\pi, s}\big|\Big\}$ &\\ [0.4em]
\hline}
 & The time of the most recent update to & \\[-0.2em]
$\prev(t,k)$ & coordinate $k$, if any, in the time range & Lemma~\ref{lem::What-on-two-paths}\\[-0.2em]
& $[t-2q,t-1]$; otherwise, we set it to $t$. & \\
\hline
&&\\ [-1.1em]
$\nuo,\nut$ & $\nuo \triangleq  \frac{20q^2}{n}$ and $\nut \triangleq \frac{24 q^2\Lresbar^2} {n\G^2}$ & Lemma~\ref{lem::key-rec-bound-full}\\[0.4em]
\hline
\end{tabular}

\vspace{0.1in}

\noindent \begin{tabular}{|c|c|c|}
\hline
\textbf{Notation / } & \textbf{Definition / Description} & \textbf{First}\\
\textbf{Parameter} && \textbf{Appearance}\\
\hline
&&\\ [-1.1em]
$\Lam$, $r$ & $\Lam \triangleq  \frac{L_{\overline{res}}^2}{\Gamma^2} +  1$ and $r \triangleq\frac{160 q^2}{n}\cdot \Lam^2$ & \multirow{2}{*}{Lemma~\ref{lem::grad::sync::bound::gen}}\\
$\nutr,\nuf$ & $\nutr \triangleq \frac{3}{16}\left(\frac{r^2r} {1-r} + r\right)$ and $\nuf \triangleq\frac {6r}{1-r}$ & \\ [0.4em]
\hline
$\mathcal{V}_m$ &  & Eqn.~\eqref{def::V::m} \\
\hline
$\widetilde{\Delta}^{t_{m-1}} x_{k_s}^{\pi, s}$ & Before Observation~\ref{obs::tilde::delta} & \\
\hline
&&\\ [-1.1em]
$\varpi$ & $\varpi \triangleq \frac{1}{q} \left[ \frac{2r}{3} + \frac{3r^2}{1280} + \frac{9r^3}{25600}  + \frac{3r^2}{1 - r}  + \frac{r^3}{426(1 - r)} + \frac{r^4}{2844(1-r)} \right]$ &
Lemma~\ref{lem::final-bound-on-Errt}\\ [0.4em]
\hline
$\varrho$ & $\varrho \triangleq \frac{1}{8} - \frac{15 r}{1 - r}$ & Section~\ref{sec::amortization} \\ [0.4em]
\hline
$\gamma$ & a parameter introduced for amortization & Eqn.~\eqref{ineq::new::3}\\
\hline
$c$ & $c \triangleq \varpi +  (\gamma + \varpi) \left(\frac{\nuo}{q} + \frac{\nut}{q}\right)$ & Lemma~\ref{lem::ApAmProg}\\
\hline
\end{tabular}

\bigskip

\noindent \begin{tabular}{|rl|}
\hline
\multicolumn{2}{|l|}{For any set $R, S \subset [u-4q,u+q] \setminus A^{\pi, u} \cup\{s\}$, when the first $(u-4q - 1)$ updates on } \\[-0.1em]
\multicolumn{2}{|l|}{path $\pi$ have been fixed, and update $\calU_v$ is excluded from the computation of $\calU_s$ for }\\[-0.1em]
\multicolumn{2}{|l|}{$v \in R\cup S$ and for $v > u+ q$; for all $r \in A^{\pi,u}$, the value of the update $\calU_r$ is already fixed,  }\\[-0.1em]
\multicolumn{2}{|l|}{then:}\\
\hline
&\\ [-1.2em]
{$\Delta_{\max}^{u, R,S} x_{k_s}^{\pi, s}~\triangleq~$}&
the maximum value that $\Delta \xkspis$ can assume \\ [0.4em]
\hline
&\\ [-1.2em]
{$\widetilde{g}^{u, R, S, \pi, s}_{\max, k_s}~\triangleq~$}
& the maximum value of $\widetilde{g}_{\ks}^{\pi,s}$ can assume \\[0.4em]
\hline
\end{tabular}
